\documentclass{amsart}
\usepackage{amsmath}
\usepackage{amssymb}
\usepackage{amsthm}
\usepackage{epsfig}
\usepackage{amsfonts}
\usepackage{amscd}
\usepackage{mathrsfs}
\usepackage{enumerate}
\usepackage{latexsym}
\usepackage{url}

\newtheorem{theorem}{Theorem}[section]
\newtheorem{lemma}[theorem]{Lemma}
\newtheorem{proposition}[theorem]{Proposition}
\newtheorem{corollary}[theorem]{Corollary}

\theoremstyle{definition}
\newtheorem{definition}[theorem]{Definition}
\newtheorem{example}[theorem]{Example}

\newtheorem{remark}[theorem]{Remark}
\newtheorem{notation}[theorem]{Notation}

\theoremstyle{question}

\numberwithin{equation}{section}

\newtheorem*{xclaim}{Claim}

\begin{document}

\title[Continuous mappings with null support]{Continuous mappings with null support}

\author{M.R. Koushesh}
\address{Department of Mathematical Sciences, Isfahan University of Technology, Isfahan 84156--83111, Iran}
\address{School of Mathematics, Institute for Research in Fundamental Sciences (IPM), P.O. Box: 19395--5746, Tehran, Iran}
\email{koushesh@cc.iut.ac.ir}

\subjclass[2010]{Primary 54D35, 54D60, 46J10, 46J25, 46E25, 46E15; Secondary 54C35, 54D65, 46H05, 16S60}

\keywords{Stone--\v{C}ech compactification, ideal of sets, real Banach algebra, commutative Gelfand--Naimark Theorem, Hewitt realcompactification}

\begin{abstract}
Let $X$ be a (topological) space and let ${\mathscr I}$ be an ideal in $X$, that is, a collection of subsets of $X$ which contains all subsets of its elements and is closed under finite unions. The elements of ${\mathscr I}$ are called \textit{null}. The space $X$ is \textit{locally null} if each $x$ in $X$ has a null neighborhood in $X$. Let $C_b(X)$ denote the normed algebra of all continuous bounded real-valued mappings on $X$ equipped with the supremum norm, $C_0(X)$ denote the subalgebra of $C_b(X)$ consisting of elements vanishing at infinity and $C_{00}(X)$ the subalgebra of $C_b(X)$ consisting of elements with compact support.

We study the normed subalgebra $C^{\mathscr I}_{00}(X)$ of $C_b(X)$ consisting of all $f$ in $C_b(X)$ whose support has a null neighborhood in $X$, and the Banach subalgebra $C^{\mathscr I}_0(X)$ of $C_b(X)$ consisting of all $f$ in $C_b(X)$ such that $|f|^{-1}([1/n,\infty))$ has a null neighborhood in $X$ for all positive integer $n$. We prove that if $X$ is a normal locally null space then $C^{\mathscr I}_{00}(X)$ and $C^{\mathscr I}_0(X)$ are respectively isometrically isomorphic to $C_{00}(Y)$ and $C_0(Y)$ for a unique locally compact Hausdorff space $Y$; furthermore, $C^{\mathscr I}_{00}(X)$ is dense in $C^{\mathscr I}_0(X)$. We construct $Y$ explicitly as a subspace of the Stone--\v{C}ech compactification $\beta X$ of $X$. The space $Y$ is locally compact (and countably compact, in certain cases), contains $X$ densely, and in specific cases turns out to be familiar subspaces of $\beta X$. The known topological structure of $Y$ enables us to establish several commutative Gelfand--Naimark type theorems and derive results not generally expected to be deducible from the standard Gelfand theory.
\end{abstract}

\maketitle

\tableofcontents

\section{Introduction}

Throughout this article by a \textit{space} we mean a topological space. Let $X$ be a space. We denote by $C(X)$ the set of all continuous $f:X\rightarrow\mathbb{R}$ and we denote by $C_b(X)$ the set of all bounded elements of $C(X)$. Let $f\in C(X)$. The \textit{zero-set} of $f$ is $f^{-1}(0)$ and is denoted by $\mathrm{z}(f)$, the \textit{cozero-set} of $f$ is $X\setminus f^{-1}(0)$ and is denoted by $\mathrm{coz}(f)$, and the \textit{support} of $f$ is $\mathrm{cl}_X\mathrm{coz}(f)$ and is denoted by $\mathrm{supp}(f)$. The set of all zero-sets of $X$ is denoted by $\mathscr{Z}(X)$ and the set of all cozero-sets of $X$ is denoted by $\mathrm{Coz}(X)$. We will denote by $C_0(X)$ the set of all $f\in C_b(X)$ which vanish at infinity (that is, $|f|^{-1}([1/n,\infty))$ is compact for each positive integer $n$) and we will denote by $C_{00}(X)$ the set of all $f\in C_b(X)$ with compact support.

An \textit{upper semi-lattice} $(L,\leq)$ is a partially ordered set that together with any two elements $a,b\in L$ it contains their least upper bound $a\vee b$. Let $(L,\leq)$ be an upper semi-lattice. A non-empty subset $I$ of $L$ is an \textit{ideal} in $L$ if it satisfies the following:
\begin{itemize}
\item If $a,b\in I$ then $a\vee b\in I$.
\item If $t\in L$ and $t\leq a$ with $a\in I$ then $t\in I$.
\end{itemize}
Suppose that $(L,\leq)$ contains the least upper bound for any countable number of its elements. An ideal $I$ of $L$ is a \textit{$\sigma$-ideal} in $L$ if it satisfies the following:
\begin{itemize}
\item If $J\subseteq I$ is countable then $\bigvee J\in I$.
\end{itemize}
We are particularly interested in upper semi-lattices $(\mathscr{L},\leq)$ such that $\mathscr{L}\subseteq\mathscr{P}(X)$, where $X$ is a space, and $\leq$ is the set-theoretic inclusion $\subseteq$; the elements of ${\mathscr I}$ are then called \textit{null} (or \textit{negligible}). The ideal $\mathscr{I}$ is \textit{proper} if $X$ in not null. The space $X$ is \textit{locally null} if each of its points has a null neighborhood in $X$.

Roughly speaking, an ideal may be viewed as a collection of sets that are considered somehow to be ``small" or ``negligible". Every element of the upper semi-lattice contained in an element of the ideal must also be in the ideal; this explains what is meant by ``smallness".

Let $X$ be a locally separable metrizable space. In \cite{Ko6} (also \cite {Ko11} and \cite{Ko10}) we have introduced and studied the Banach subalgebra $C_s(X)$ of $C_b(X)$ consisting of all $f\in C_b(X)$ with separable support. This has been accomplished through the critical introduction of a certain subspace $\lambda X$ of the Stone--\v{C}ech compactification $\beta X$ of $X$ and relating the algebraic structure of $C_s(X)$ to the topology of $\lambda X$, using techniques we have already developed in \cite{Ko3} and \cite{Ko12}. Here we aim to follow the same route. We consider various subalgebras of $C_b(X)$ and study their structures by relating them to the topology of subspaces of $\beta X$.

The article is divided into two parts.

In Part \ref{HFPG}, for a space $X$ and an ideal $\mathscr{I}$ of an upper semi-lattice $(\mathscr{L},\subseteq)$, where $\mathscr{L}\subseteq\mathscr{P}(X)$, we consider the subsets $C^{\mathscr I}_{00}(X)$ and $C^{\mathscr I}_0(X)$ of $C_b(X)$ defined by
\[C^{\mathscr I}_{00}(X)=\big\{f\in C_b(X):\mathrm{supp}(f)\mbox{ has a null neighborhood in }X\big\}\]
and
\[C^{\mathscr I}_0(X)=\big\{f\in C_b(X):|f|^{-1}\big([1/n,\infty)\big)\mbox{ has a null neighborhood in $X$ for each }n\big\}.\]
The sets $C^{\mathscr I}_{00}(X)$ and $C^{\mathscr I}_0(X)$ coincide with $C_{00}(X)$ and $C_0(X)$, respectively, if $X$ is locally compact and $\mathscr{I}$ is the set of all subspaces of $X$ with compact closure. Under certain conditions the expression of either $C^{\mathscr I}_{00}(X)$ or $C^{\mathscr I}_0(X)$ simplifies; in particular
\[C^{\mathscr I}_{00}(X)=\big\{f\in C_b(X):\mathrm{supp}(f)\mbox{ is null}\big\}\]
if $X$ is Lindel\"{o}f and locally null, $\mathscr{I}$ is a $\sigma$-ideal and $\mathscr{L}$ contains the set $\mathscr{RC}(X)$ of all regular closed subspaces of $X$ (a subspace of $X$ is called \textit{regular closed} in $X$ if it equals the closure of its interior),
\[C^{\mathscr I}_0(X)=\big\{f\in C_b(X):|f|^{-1}\big([1/n,\infty)\big)\mbox{ is null for each }n\big\}\]
if $\mathscr{L}$ contains the set $\mathscr{Z}(X)$ of all zero-sets of $X$,
\[C^{\mathscr I}_0(X)=\big\{f\in C_b(X):\mathrm{coz}(f)\mbox{ has a null neighborhood in }X\big\}\]
if $\mathscr{I}$ is a $\sigma$-ideal and
\[C^{\mathscr I}_0(X)=\big\{f\in C_b(X):\mathrm{coz}(f)\mbox{ is null}\big\}\]
if $\mathscr{I}$ is a $\sigma$-ideal and $\mathscr{L}$ contains the set $\mathrm{Coz}(X)$ of all cozero-sets of $X$. In general $C^{\mathscr I}_{00}(X)$ is an ideal in $C_b(X)$ and $C^{\mathscr I}_0(X)$ is a closed ideal in $C_b(X)$ which contains $C^{\mathscr I}_{00}(X)$. Furthermore, if $X$ is completely regular, then the following are equivalent
\begin{itemize}
  \item $C^{\mathscr I}_{00}(X)$ is of empty hull,
  \item $C^{\mathscr I}_0(X)$ is of empty hull,
  \item $X$ is locally null,
\end{itemize}
(where an ideal $H$ of $C_b(X)$ is \textit{of empty hull} if for every $x\in X$ we have $f(x)\neq 0$ for some $f\in H$), and if $X$ is locally null, then the following are equivalent
\begin{itemize}
  \item $C^{\mathscr I}_{00}(X)$ is unital,
  \item $C^{\mathscr I}_0(X)$ is unital,
  \item $\mathscr{I}$ is non-proper.
\end{itemize}
The main result of this part states that if $X$ is normal and locally null, then $C^{\mathscr I}_{00}(X)$ and $C^{\mathscr I}_0(X)$ are respectively isometrically isomorphic to $C_{00}(Y)$ and $C_0(Y)$ for a unique (up to homeomorphism) locally compact Hausdorff space $Y$, namely $Y=\lambda_{\mathscr I} X$, where
\[\lambda_{\mathscr I} X=\bigcup\big\{\mathrm{int}_{\beta X}\mathrm{cl}_{\beta X}C:C\in\mathrm{Coz}(X)\mbox{ and }\mathrm{cl}_XC\mbox{ has a null neighborhood in $X$}\big\},\]
considered as a subspace of the Stone--\v{C}ech compactification $\beta X$ of $X$. Furthermore,
\begin{itemize}
  \item $X$ is dense in $Y$,
  \item $C^{\mathscr I}_{00}(X)$ is dense in $C^{\mathscr I}_0(X)$,
  \item $Y$ is compact if and only if $C^{\mathscr I}_{00}(X)$ is unital if and only if $C^{\mathscr I}_0(X)$ is unital.
\end{itemize}
(For the case of $C^{\mathscr I}_0(X)$ we may assume that $X$ is only completely regular provided that ${\mathscr I}$ is an ideal in $(\mathrm{Coz}(X),\subseteq)$.) Moreover, if $X$ is also Lindel\"{o}f and $\mathscr{I}$ is a $\sigma$-ideal, then
\begin{itemize}
  \item $C^{\mathscr I}_0(X)=C^{\mathscr I}_{00}(X)$,
  \item $C_0(Y)=C_{00}(Y)$; in particular, $Y$ is countably compact,
\end{itemize}
and
\[Y=\bigcup\{\mathrm{int}_{\beta X}\mathrm{cl}_{\beta X}U:U\mbox{ is null}\}.\]

In Part \ref{KHGJ} we deal with specific examples. This specification, either of the space $X$ or the ideal ${\mathscr I}$, enables us to study $C^{\mathscr I}_{00}(X)$ and $C^{\mathscr I}_0(X)$ further and deeper. This part is divided into four sections.

In Section \ref{HFLH}, we consider certain well known ideals of $\mathbb{N}$. These include the \textit{summable} ideal
\[{\mathscr S}=\bigg\{A\subseteq\mathbb{N}:\sum_{n\in A}\frac{1}{n}\mbox{ converges}\bigg\}\]
and the \textit{density} ideal
\[{\mathscr D}=\bigg\{A\subseteq\mathbb{N}:\limsup_{n\rightarrow\infty}\frac{|A\cap\{1,\ldots,n\}|}{n}=0\bigg\}\]
of $\mathbb{N}$. Denote
\[\mathfrak{s}_{00}=C^{\mathscr S}_{00}(\mathbb{N}),\;\; \mathfrak{d}_{00}=C^{\mathscr D}_{00}(\mathbb{N}),\;\; \mathfrak{s}_0=C^{\mathscr S}_0(\mathbb{N})\;\mbox{ and }\; \mathfrak{d}_0=C^{\mathscr D}_0(\mathbb{N}).\]
Then
\[\mathfrak{s}_{00}=\bigg\{\mathbf{x}\in\ell_\infty:\sum_{\mathbf{x}(n)\neq0}\frac{1}{n}\mbox{ converges}\bigg\},\]
and
\[\mathfrak{d}_{00}=\bigg\{\mathbf{x}\in\ell_\infty:\limsup_{n\rightarrow\infty}\frac{|\{k\leq n:\mathbf{x}(k)\neq0\}|}{n}=0\bigg\},\]
are normed subalgebras of $\ell_\infty$ and
\[\mathfrak{s}_0=\bigg\{\mathbf{x}\in\ell_\infty:\sum_{|\mathbf{x}(n)|\geq\epsilon}\frac{1}{n}\mbox{ converges for each }\epsilon>0\bigg\},\]
and
\[\mathfrak{d}_0=\bigg\{\mathbf{x}\in\ell_\infty:\limsup_{n\rightarrow\infty}\frac{|\{k\leq n:|\mathbf{x}(k)|\geq\epsilon\}|}{n}=0\mbox{ for each }\epsilon>0\bigg\}\]
are Banach subalgebras of $\ell_\infty$. It follows from our discussion that $\mathfrak{s}_{00}$ and $\mathfrak{s}_0$ are respectively isometrically isomorphic to $C_{00}(S)$ and $C_0(S)$ with
\[S=\bigcup\bigg\{\mathrm{cl}_{\beta\mathbb{N}}A:A\subseteq\mathbb{N}\mbox{ and }\sum_{n\in A}\frac{1}{n}\mbox{ converges}\bigg\},\]
and $\mathfrak{d}_{00}$ and $\mathfrak{d}_0$ are respectively isometrically isomorphic to $C_{00}(D)$ and $C_0(D)$ with
\[D=\bigcup\bigg\{\mathrm{cl}_{\beta\mathbb{N}}A:A\subseteq\mathbb{N}\mbox{ and }\limsup_{n\rightarrow\infty}\frac{|A\cap\{1,\ldots,n\}|}{n}=0\bigg\},\]
where $S$ and $D$ are considered as subspaces of $\beta\mathbb{N}$. In particular, $\mathfrak{s}_{00}$ and $\mathfrak{d}_{00}$ are dense in $\mathfrak{s}_0$ and $\mathfrak{d}_0$, respectively, and are all non-unital. As corollaries of the above representations it follows that $\mathfrak{s}_{00}$ and $\mathfrak{d}_{00}$ each contains an isometric copy of the normed algebra
\[\bigoplus_{n=1}^\infty\ell_\infty,\]
$\mathfrak{s}_{00}/c_{00}$ and $\mathfrak{d}_{00}/c_{00}$ each contains a copy of the algebra
\[\bigoplus_{i<2^\omega}\frac{\ell_\infty}{c_{00}}\]
and $\mathfrak{s}_0/c_0$ and $\mathfrak{d}_0/c_0$ each contains an isometric copy of the normed algebra
\[\bigoplus_{i<2^\omega}\frac{\ell_\infty}{c_0}.\]

In Section \ref{KHDS} we consider certain subalgebras of $C_b(X)$ and we prove a few representation theorems. More importantly, for a completely regular space $X$ and a Banach subalgebra $H$ of $C_b(X)$ of empty hull which satisfies
\begin{itemize}
\item $f\in H$ whenever $f\in C_b(X)$ and $\mathrm{coz}(f)\subseteq\mathrm{coz}(h)$ for some $h\in H$,
\end{itemize}
we prove that $H$ is isometrically isomorphic to $C_0(Y)$ with the (unique) locally compact Hausdorff space $Y$ defined by
\[Y=\bigcup\big\{\mathrm{int}_{\beta X}\mathrm{cl}_{\beta X}C:C\in\mathrm{Coz}(X)\mbox{ and }\mathrm{cl}_XC\subseteq\mathrm{coz}(h)\mbox{ for some }h\in H\big\},\]
considered as a subspace of $\beta X$. Furthermore, $X$ is dense in $Y$, and the following are equivalent
\begin{itemize}
  \item $H$ is unital,
  \item $H$ contains a non-vanishing element,
  \item $Y$ is compact.
\end{itemize}
In the special case in which $X$ is Lindel\"{o}f and $H$ is a closed $z$-ideal in $C_b(X)$ (that is, if $\mathrm{z}(f)=\mathrm{z}(h)$, where $f\in C_b(X)$ and $h\in H$, then $f\in H$) we have
\[Y=\bigcup\big\{\mathrm{int}_{\beta X}\mathrm{cl}_{\beta X}\mathrm{coz}(h):h\in H\big\}\]
and
\[C_0(Y)=C_{00}(Y);\]
in particular, in this case, $Y$ is countably compact.

In Section \ref{HGF}, for a topological measure space $(X,{\mathscr O},{\mathscr B},\mu)$ we consider the ideal
\[{\mathscr M}=\big\{B\in{\mathscr B}:\mu(B)=0\big\}.\]
This leads to the consideration of the Banach subalgebra $\mathfrak{M}_0(X)$ of $C_b(X)$ consisting of elements with $\mu$-null cozero-set. In particular, if $X$ is completely regular then $\mathfrak{M}_0(X)$ is of empty hull if and only if the measure $\mu$ is locally null, and if so, then $\mathfrak{M}_0(X)$ is unital if and only if the measure $\mu$ is trivial.

Finally, in Section \ref{KIJGD}, for a space $X$ and a topological property $\mathfrak{P}$ we will be dealing with subsets of $C_b(X)$ consisting of elements with support having $\mathfrak{P}$.

Let $X$ be a space and $\mathfrak{P}$ a topological property. Let
\[{\mathscr I}_\mathfrak{P}=\{A\subseteq X:\mathrm{cl}_XA\mbox{ has }\mathfrak{P}\}.\]
Then ${\mathscr I}_\mathfrak{P}$ is an ideal in $({\mathscr P}(X),\subseteq)$, provided that $\mathfrak{P}$ is closed hereditary (that is, any closed subspace of a space with $\mathfrak{P}$ has $\mathfrak{P}$) and is preserved under finite closed sums (that is, any space which is a finite union of its closed subspaces each having $\mathfrak{P}$ has $\mathfrak{P}$). We consider $C^{\mathscr I}_{00}(X)$ and $C^{\mathscr I}_0(X)$ where ${\mathscr I}={\mathscr I}_\mathfrak{P}$. For simplicity of the notation denote
\[C^\mathfrak{P}_{00}(X)=C^{{\mathscr I}_\mathfrak{P}}_{00}(X)\;\;\;\;\mbox{ and }\;\;\;\; C^\mathfrak{P}_0(X)=C^{{\mathscr I}_\mathfrak{P}}_0(X)\]
and
\[\lambda_\mathfrak{P}X=\lambda_{{\mathscr I}_\mathfrak{P}}X.\]
In this context we have
\[C^\mathfrak{P}_{00}(X)=\big\{f\in C_b(X):\mathrm{supp}(f)\mbox{ has a closed neighborhood in $X$ with }\mathfrak{P}\big\}\]
and
\[C^\mathfrak{P}_0(X)=\big\{f\in C_b(X):|f|^{-1}\big([1/n,\infty)\big)\mbox{ has }\mathfrak{P}\mbox{ for each }n\big\}.\]
It follows that $C^\mathfrak{P}_{00}(X)$ is a normed subalgebra of $C_b(X)$ and $C^\mathfrak{P}_0(X)$ is a Banach subalgebra of $C_b(X)$ which contains $C^\mathfrak{P}_{00}(X)$.
Furthermore,
\[C^\mathfrak{P}_{00}(X)=\big\{f\in C_b(X):\mathrm{supp}(f)\mbox{ has a neighborhood in $X$ with }\mathfrak{P}\big\}\]
if $X$ is normal and
\[C^\mathfrak{P}_0(X)=\big\{f\in C_b(X):\mathrm{coz}(f)\mbox{ has }\mathfrak{P}\big\}\]
if $\mathfrak{P}$ is preserved under countable closed sums (that is, any space which is a countable union of its closed subspaces each having $\mathfrak{P}$ has $\mathfrak{P}$). Moreover, if $X$ is regular, then the following are equivalent:
\begin{itemize}
  \item $X$ is locally null,
  \item $X$ is locally-$\mathfrak{P}$,
\end{itemize}
and if $X$ is completely regular, then the following are equivalent:
\begin{itemize}
  \item $C^\mathfrak{P}_{00}(X)$ is of empty hull,
  \item $C^\mathfrak{P}_0(X)$ is of empty hull,
  \item $X$ is locally-$\mathfrak{P}$.
\end{itemize}
(A space is called \textit{locally-$\mathfrak{P}$}, where $\mathfrak{P}$ is a topological property, if each of its points has a neighborhood in it with $\mathfrak{P}$.) Also, if $X$ is completely regular and locally-$\mathfrak{P}$, then the following are equivalent:
\begin{itemize}
  \item $C^\mathfrak{P}_{00}(X)$ is unital,
  \item $C^\mathfrak{P}_0(X)$ is unital,
  \item $X$ has $\mathfrak{P}$.
\end{itemize}
In particular, if $X$ is normal and locally-$\mathfrak{P}$, then $C^\mathfrak{P}_{00}(X)$ and $C^\mathfrak{P}_0(X)$ are respectively isometrically isomorphic to $C_{00}(Y)$ and $C_0(Y)$ for the (unique) locally compact Hausdorff space $Y=\lambda_\mathfrak{P}X$. Furthermore,
\begin{itemize}
  \item $X$ is dense in $Y$,
  \item $C^\mathfrak{P}_{00}(X)$ is dense in $C^\mathfrak{P}_0(X)$,
  \item $Y$ is compact if and only if $X$ has $\mathfrak{P}$.
\end{itemize}
For the case of $C^\mathfrak{P}_0(X)$ the normality requirement of $X$ can be relaxed to complete regularity provided that $\mathfrak{P}$ is preserved under countable closed sums.

Let $X$ be a regular locally-$\mathfrak{P}$ space with $\mathfrak{Q}$, where $\mathfrak{P}$ and $\mathfrak{Q}$ are topological properties such that
\begin{itemize}
  \item $\mathfrak{P}$ and $\mathfrak{Q}$ are closed hereditary,
  \item a space with both $\mathfrak{P}$ and $\mathfrak{Q}$ is Lindel\"{o}f,
  \item $\mathfrak{P}$ is preserved under countable closed sums,
  \item a space with $\mathfrak{Q}$ having a dense subspace with $\mathfrak{P}$ has $\mathfrak{P}$.
\end{itemize}
Then we have
\[C^\mathfrak{P}_{00}(X)=\big\{f\in C_b(X):\mathrm{supp}(f)\mbox{ has }\mathfrak{P}\big\}=C^\mathfrak{P}_0(X).\]
If $X$ is normal, then $C^\mathfrak{P}_0(X)$ is isometrically isomorphic to $C_0(Y)$ for the (unique) locally compact Hausdorff space $Y=\lambda_\mathfrak{P}X$. In this case
\[C_{00}(Y)=C_0(Y);\]
in particular, $Y$ is countably compact. The special case in which $\mathfrak{P}$ is the Lindel\"{o}f property and $\mathfrak{Q}$ is metrizability (or paracompactness) is studied in great detail. In particular, if $X$ is a locally separable metrizable space and
\[C_s(X)=\big\{f\in C_b(X):\mathrm{supp}(f)\mbox{ is separable}\big\}\]
then $C_s(X)$ is a Banach algebra which is isometrically isomorphic to $C_0(Y)$ for the (unique) locally compact Hausdorff space $Y$ defined by
\[Y=\bigcup\{\mathrm{cl}_{\beta X}S:S\subseteq X\mbox{ is separable}\};\]
furthermore,
\begin{itemize}
\item $\dim C_s(X)=d(X)^{\aleph_0}$,
\item $Y$ is non-normal if $X$ is non-separable,
\item $C_0(Y)=C_{00}(Y)$; in particular, $Y$ is countably compact.
\end{itemize}
Here $d(X)$ (called the \textit{density} of $X$) is defined by
\[d(X)=\min\big\{|D|:D\mbox{ is dense in }X\big\}+\aleph_0.\]
An analogous result holds if $\mathfrak{P}$ is the Lindel\"{o}f property and $\mathfrak{Q}$ is local compactness paracompactness, with $C_s(X)$ being replaced by
\[C_l(X)=\big\{f\in C_b(X):\mathrm{supp}(f)\mbox{ is Lindel\"{o}f}\big\},\]
the space $Y$ being replaced by
\[Y=\bigcup\{\mathrm{cl}_{\beta X}L:L\subseteq X\mbox{ is Lindel\"{o}f}\},\]
and $d(X)$ being replaced by $\ell(X)$, where $\ell(X)$ (the \textit{Lindel\"{o}f number} of $X$) is defined by
\[\ell(X)=\min\{\mathfrak{n}:\mbox{any open cover of $X$ has a subcover of cardinality}\leq\mathfrak{n}\}+\aleph_0.\]

Locally Lindel\"{o}f metrizable spaces and locally compact paracompact spaces, besides sharing the above theorem, share the following property. Let $X$ be either a locally Lindel\"{o}f metrizable space or a locally compact paracompact space. Then there is a chain of length $\lambda$, in which $\aleph_\lambda=\ell(X)$, consisting of closed ideals $H_\mu$'s of $C_b(X)$ such that
\[C_0(X)\subsetneqq H_0\subsetneqq H_1\subsetneqq\cdots\subsetneqq H_\lambda=C_b(X).\]
Here $H_\mu$, for each $\mu<\lambda$, is isometrically isomorphic to
\[C_0(Y_\mu)=C_{00}(Y_\mu),\]
where $Y_\mu$ is a locally compact countably compact Hausdorff space which contains $X$ densely.

The concluding results in this section deal with realcompactness and pseudocompactness. (Recall that a completely regular space $X$ is \textit{realcompact} if it is homeomorphic to a closed subspaces of some product $\mathbb{R}^\mathfrak{m}$ and is \textit{pseudocompact} if there is no unbounded continuous $f:X\rightarrow\mathbb{R}$.) Realcompactness is a closed hereditary topological property, and though is not in general preserved under finite closed sums, it is so in the realm of normal spaces. We show that if $\mathfrak{P}$ is realcompactness and $X$ is normal, then
\[\lambda_\mathfrak{P}X=\beta X\setminus\mathrm{cl}_{\beta X}(\upsilon X\setminus X)\]
where $\upsilon X$ is the Hewitt realcompactification of $X$. Also, if $X$ is completely regular, then
\[\lambda_{\mathscr U}X=\mathrm{int}_{\beta X}\upsilon X\]
for the ideal
\[{\mathscr U}=\big\{A\in\mathscr{RC}(X):A\mbox{ is pseudocompact}\big\}\]
of the partially ordered set $(\mathscr{RC}(X),\subseteq)$ of all regular closed subspaces of $X$. (Recall that the \textit{Hewitt realcompactification} of a completely regular space $X$ is a realcompact space $\upsilon X$ which contains $X$ as a dense subspace and is such that every continuous $f:X\rightarrow\mathbb{R}$ is continuously extendible over $\upsilon X$; we may assume that $\upsilon X\subseteq\beta X$.)

In the preprint \cite{T}, for a completely regular space $X$ and a filter base ${\mathscr B}$ of open subspaces of $X$, the author defined $C_{\mathscr B}(X)$ to be the set of all $f\in C(X)$ whose support is contained in $X\setminus B$ for some $B\in{\mathscr B}$, and $C_{\infty{\mathscr B}}(X)$ to be the set of all $f\in C(X)$ such that $|f|^{-1}([1/n,\infty))$ is contained in $X\setminus B$ for some $B\in{\mathscr B}$ for each positive integer $n$. (See \cite{AN} for certain special cases.) Also, if ${\mathscr I}$ is an ideal of closed subspaces of $X$, in \cite{AG}, the authors defined $C_{\mathscr I}(X)$ to be the set of all $f\in C(X)$ whose support is contained in ${\mathscr I}$, and $C_\infty^{\mathscr I}(X)$ to be the set of all $f\in C(X)$ such that $|f|^{-1}([1/n,\infty))$ is contained in ${\mathscr I}$ for each positive integer $n$. Despite certain similarities between our definitions and the definitions given in \cite{T} or \cite{AG}, the existing differences between definitions have caused this work to have little in common with either \cite{T} or \cite{AG}.

\subsection*{Preliminaries and notation}\label{GYYI}

This section contains certain notion and known facts that we will use frequently throughout this article. The first part is to provide examples of upper semi-lattices; upper semi-lattices are the natural setting to state and prove our results. The second part review certain properties of the Stone--\v{C}ech compactification; the Stone--\v{C}ech compactification is the main tool in our study. For more information on the theory of the Stone--\v{C}ech compactification we refer the reader to \cite{E}, \cite{GJ}, \cite{PW} and \cite{We}.

\subsubsection*{Examples of upper semi-lattices} In the following we give examples of upper semi-lattices. (Indeed, the examples are all lattices; however, this will not be used in the sequel.)
\begin{enumerate}
 \item Let $X$ be a space. A subspace $A$ of $X$ is \textit{regular closed} in $X$ if $A=\mathrm{cl}_X\mathrm{int}_XA$. Denote by $\mathscr{RC}(X)$ the set of all regular closed subspaces of $X$. Note that the closure of each open subspace of $X$ is regular closed in $X$. The partially ordered set $(\mathscr{RC}(X),\subseteq)$ is an upper semi-lattice. If $\mathscr{B}\subseteq\mathscr{RC}(X)$ is non-empty then
     \[\bigvee\mathscr{B}=\mathrm{cl}_X\Big(\mathrm{int}_X\Big(\bigcup\mathscr{B}\Big)\Big).\]
     In particular, if $A,B\in\mathscr{RC}(X)$ then
     \[A\vee B=A\cup B.\]
 \item Let $X$ be a space. A subspace $A$ of $X$ is \textit{regular open} in $X$ if $A=\mathrm{int}_X\mathrm{cl}_XA$. Denote by $\mathscr{RO}(X)$ the set of all regular open subspaces of $X$. Note that the interior of each closed subspace of $X$ is regular open in $X$. The partially ordered set $(\mathscr{RO}(X),\subseteq)$ is an upper semi-lattice. If $\mathscr{B}\subseteq\mathscr{RO}(X)$ is non-empty then
     \[\bigvee\mathscr{B}=\mathrm{int}_X\Big(\mathrm{cl}_X\Big(\bigcup\mathscr{B}\Big)\Big).\]
     In particular, if $A,B\in\mathscr{RO}(X)$ then
     \[A\vee B=\mathrm{int}_X\mathrm{cl}_X(A\cup B).\]
 \item Let $X$ be a space. A subspace $Z$ of $X$ is a \textit{zero-set} in $X$ if $Z=f^{-1}(0)$ for some continuous $f:X\rightarrow[0,1]$. Denote by $\mathscr{Z}(X)$ the set of all zero-sets of $X$. The partially ordered set $(\mathscr{Z}(X),\subseteq)$ is an upper semi-lattice. If $Z,S\in\mathscr{Z}(X)$ then
     \[S\vee Z=S\cup Z;\]
     to see this, let $f,g:X\rightarrow[0,1]$ be continuous and observe that
     \[\mathrm{z}(f)\cup\mathrm{z}(g)=\mathrm{z}(fg).\]
 \item Let $X$ be a space. A subspace $C$ of $X$ is a \textit{cozero-set} in $X$ if $C=X\setminus Z$ for some zero-set $Z$ in $X$. Denote by $\mathrm{Coz}(X)$ the set of all cozero-sets of $X$. The partially ordered set $(\mathrm{Coz}(X),\subseteq)$ is an upper semi-lattice. Indeed, if $C,D\in\mathrm{Coz}(X)$ then
     \[C\vee D=C\cup D.\]
     More generally, if $\mathscr{B}\subseteq\mathrm{Coz}(X)$ is non-empty and countable then
     \[\bigvee\mathscr{B}=\bigcup\mathscr{B};\]
     as, if $f_n:X\rightarrow[0,1]$ is continuous for each positive integer $n$, then
     \[\bigcup_{n=1}^\infty\mathrm{coz}(f_n)=\mathrm{coz}(f)\]
     where
     \[f=\sum_{n=1}^\infty \frac{f_n}{2^n}.\]
     (That $f$ is continuous follows from the Weierstrass $M$-test.)
\end{enumerate}

The fact that $(\mathscr{L},\subseteq)$ is an upper semi-lattice does \textit{not} necessarily mean that if $A,B\in\mathscr{L}$ then $A\vee B=A\cup B$. (For instance, let $\mathscr{L}=\mathscr{RO}(X)$ where $X$ is a space. Then $A\vee B=\mathrm{int}_X\mathrm{cl}_X(A\cup B)$ for any $A,B\in\mathscr{L}$, as noted previously.) However, since we have $A\subseteq A\vee B$ and $B\subseteq A\vee B$, it is always true that $A\cup B\subseteq A\vee B$.

\subsubsection*{The Stone--\v{C}ech compactification} Let $X$ be a completely regular space. A \textit{compactification} $\gamma X$ of $X$ is a compact Hausdorff space $\gamma X$ containing $X$ as a dense subspace. The \textit{Stone--\v{C}ech compactification} of $X$, denoted by $\beta X$, is the compactification of $X$ which is characterized among all compactifications of $X$ by the following property: Every continuous $f:X\rightarrow K$, where $K$ is a compact Hausdorff space, is continuously extendable over $\beta X$; denote by $f_\beta$ this continuous extension of $f$. For a completely regular space the Stone--\v{C}ech compactification always exists. In what follows we will use the following properties of $\beta X$. (See Sections 3.5 and 3.6 of \cite{E}.)
\begin{itemize}
  \item $X$ is locally compact if and only if $X$ is open in $\beta X$.
  \item Any open-closed subspace of $X$ has open-closed closure in $\beta X$.
  \item If $X\subseteq T\subseteq\beta X$ then $\beta T=\beta X$.
  \item If $X$ is normal then $\beta T=\mathrm{cl}_{\beta X}T$ for any closed subspace $T$ of $X$.
  \item Disjoint zero-sets of $X$ have disjoint closures in $\beta X$.
\end{itemize}

\part{General theory}\label{HFPG}

This part is divided into two sections. Section \ref{GPG} introduces the normed algebra $C^{\mathscr I}_{00}(X)$ and Section \ref{JGGG} introduces the Banach algebra $C^{\mathscr I}_0(X)$. Results of this part are stated and proved in the most general context; Part \ref{KHGJ} will be devoted subsequently to the consideration of specific examples.

\section{The normed algebra $C^{\mathscr I}_{00}(X)$}\label{GPG}

Let $X$ be a space and let $\mathscr{I}$ be an ideal in an upper semi-lattice $(\mathscr{L},\subseteq)$, where $\mathscr{L}\subseteq\mathscr{P}(X)$. In this section we consider the subset $C^{\mathscr I}_{00}(X)$ of $C_b(X)$ consisting of those elements of $C_b(X)$ whose support has a null neighborhood in $X$. The set $C^{\mathscr I}_{00}(X)$ coincides with $C_{00}(X)$ if $X$ is locally compact and $\mathscr{I}$ is the set of all subspaces of $X$ with compact closure (and $\mathscr{L}=\mathscr{P}(X)$). We show that $C^{\mathscr I}_{00}(X)$ is in general an ideal in $C_b(X)$. Furthermore, if $X$ is completely regular, then $C^{\mathscr I}_{00}(X)$ is of empty hull if and only if $X$ is locally null, and if so, then $C^{\mathscr I}_{00}(X)$ is unital if and only if $\mathscr{I}$ is non-proper. The main result of this section states that if $X$ is normal and locally null then the normed algebra $C^{\mathscr I}_{00}(X)$ is isometrically isomorphic to $C_{00}(Y)$ for some unique (up to homeomorphism) locally compact Hausdorff space $Y$; this will be made possible through the introduction of the subspace $\lambda_{\mathscr I} X$ of the Stone--\v{C}ech compactification of $X$. Furthermore, $Y$ contains $X$ densely, and is compact if and only if $C^{\mathscr I}_{00}(X)$ is unital. Our final result in this section simplifies the expression of $C^{\mathscr I}_{00}(X)$, namely
\[C^{\mathscr I}_{00}(X)=\big\{f\in C_b(X):\mathrm{supp}(f)\mbox{ is null}\big\}\]
whenever $X$ is Lindel\"{o}f and locally null, $\mathscr{I}$ is a $\sigma$-ideal and $\mathscr{L}$ contains the set $\mathscr{RC}(X)$ of all regular closed subspaces of $X$.

We now proceed with the formal treatment of the subject. Results of this section will generalize those we have already obtained in \cite{Ko6}, \cite{Ko10} and \cite{Ko11}.

Let $X$ be a space and let $A$ be a subspace of $X$. A subspace $U$ of $X$ is called a \textit{neighborhood} of $A$ in $X$ if $A\subseteq\mathrm{int}_XU$.

\begin{definition}\label{HG}
Let $X$ be a space and let $\mathscr{I}$ be an ideal in an upper semi-lattice $(\mathscr{L},\subseteq)$, where $\mathscr{L}\subseteq\mathscr{P}(X)$. Define
\[C^{\mathscr I}_{00}(X)=\big\{f\in C_b(X):\mathrm{supp}(f)\mbox{ has a null neighborhood in }X\big\}.\]
\end{definition}

The following example justifies our use of the notation $C^{\mathscr I}_{00}(X)$.

\begin{example}\label{HUJ}
Let $X$ be a locally compact space and let
\[\mathscr{I}=\{A\subseteq X:\mathrm{cl}_X A \mbox{ is compact}\}.\]
Trivially, $\mathscr{I}$ is an ideal in $(\mathscr{P}(X),\subseteq)$. As we see now, in this case we have
\[C^{\mathscr I}_{00}(X)=C_{00}(X).\]
This justifies our use of the notation $C^{\mathscr I}_{00}(X)$ here. Let $f\in C_b(X)$. It is obvious that if $\mathrm{supp}(f)$ has a null neighborhood $U$ in $X$, then $\mathrm{supp}(f)$ is compact, as it is closed in $\mathrm{cl}_X U$ and the latter is so. Now, suppose that $\mathrm{supp}(f)$ is compact. For each $x\in X$ let $V_x$ be an open neighborhood of $x$ in $X$ such that $\mathrm{cl}_X V_x$ is compact. The set $\{V_x:x\in X\}$ forms an open cover for $\mathrm{supp}(f)$. Therefore
\[\mathrm{supp}(f)\subseteq V_{x_1}\cup\cdots\cup V_{x_n}=V,\]
where $x_i\in X$ for each $i=1,\ldots,n$. Clearly, $V$ is a neighborhood of $\mathrm{supp}(f)$ in $X$, and it is null, as
\[\mathrm{cl}_XV=\mathrm{cl}_XV_{x_1}\cup\cdots\cup\mathrm{cl}_X V_{x_n},\]
(being the union of a finite number of compact subspaces) is compact.
\end{example}

\begin{theorem}\label{TTG}
Let $X$ be a space and let $\mathscr{I}$ be an ideal in an upper semi-lattice $(\mathscr{L},\subseteq)$, where $\mathscr{L}\subseteq\mathscr{P}(X)$. Then
$C^{\mathscr I}_{00}(X)$ is an ideal in $C_b(X)$.
\end{theorem}

\begin{proof}
Note that $C^{\mathscr I}_{00}(X)$ is non-empty, as it contains the zero vector $\mathbf{0}$. (Observe that there always exists a null subset of $X$; this constitutes a neighborhood for $\emptyset=\mathrm{supp}(\mathbf{0})$ in $X$.) To show that $C^{\mathscr I}_{00}(X)$ is closed under addition, let $f,g\in C^{\mathscr I}_{00}(X)$. Then, there exist null neighborhoods $U$ and $V$ of $\mathrm{supp}(f)$ and $\mathrm{supp}(g)$ in $X$, respectively. Note that $\mathrm{coz}(f+g)\subseteq\mathrm{coz}(f)\cup\mathrm{coz}(g)$. Thus
\[\mathrm{supp}(f+g)\subseteq\mathrm{supp}(f)\cup\mathrm{supp}(g)\subseteq\mathrm{int}_X U\cup\mathrm{int}_X V\subseteq\mathrm{int}_X (U\cup V)\subseteq U\cup V\subseteq U\vee V.\]
Therefore $\mathrm{supp}(f+g)$ has a null neighborhood in $X$, namely $U\vee V$. Then $f+g\in C^{\mathscr I}_{00}(X)$. Next, let $f\in C^{\mathscr I}_{00}(X)$ and $g\in C_b(X)$. Note that $\mathrm{coz}(fg)\subseteq\mathrm{coz}(f)$. Thus $\mathrm{supp}(fg)\subseteq\mathrm{supp}(f)$. In particular, $\mathrm{supp}(fg)$ has a null neighborhoods in $X$, as $\mathrm{supp}(f)$ does. Therefore $fg\in C^{\mathscr I}_{00}(X)$. That $C^{\mathscr I}_{00}(X)$ is closed under scalar multiplication follows trivially.
\end{proof}

\begin{remark}\label{TUJ}
Following \cite{GJ}, we call an ideal $H$ of $C_b(X)$ a \textit{$z$-ideal} if $\mathrm{z}(f)=\mathrm{z}(h)$, where $f\in C_b(X)$ and $h\in H$, implies that $f\in H$. It is worth noting that in Theorem \ref{TTG} the ideal $C^{\mathscr I}_{00}(X)$ is indeed a $z$-ideal; to see this let $f\in C_b(X)$ and $g\in C^{\mathscr I}_{00}(X)$ with $\mathrm{z}(f)=\mathrm{z}(g)$. Then $\mathrm{coz}(f)=\mathrm{coz}(g)$ and thus $\mathrm{supp}(f)=\mathrm{supp}(g)$. Therefore $\mathrm{supp}(f)$ has a null neighborhood in $X$, as $\mathrm{supp}(g)$ does. That is $f\in C^{\mathscr I}_{00}(X)$.
\end{remark}

The subspace $\lambda_{\mathscr I} X$ of $\beta X$ introduced below plays a crucial role in our study. The space $\lambda_{\mathscr I} X$ has been first considered in \cite{Ko3} (in a different context) to study certain classes of topological extensions. (See also \cite{Ko4}, \cite{Ko7}, \cite{Ko13}, \cite{Ko8} and \cite{Ko12}.)

\begin{definition}\label{HWA}
Let $X$ be a completely regular space and let $\mathscr{I}$ be an ideal in an upper semi-lattice $(\mathscr{L},\subseteq)$, where $\mathscr{L}\subseteq\mathscr{P}(X)$. Define
\[\lambda_{\mathscr I} X=\bigcup\big\{\mathrm{int}_{\beta X}\mathrm{cl}_{\beta X}C:C\in\mathrm{Coz}(X)\mbox{ and }\mathrm{cl}_XC\mbox{ has a null neighborhood in $X$}\big\},\]
considered as a subspace of $\beta X$.
\end{definition}

The space $\lambda_{\mathscr I} X$ just defined may have a better expression if one requires the elements of ${\mathscr I}$ to be closed subspaces of $X$. This is the subject matter of the next proposition. We need the following simple lemma.

Observe that for a space $X$ and a dense subspace $D$ of $X$ we have
\[\mathrm{cl}_XU=\mathrm{cl}_X(U\cap D)\]
for every open subspace $U$ of $X$.

\begin{lemma}\label{LKG}
Let $X$ be a completely regular space. Let $f:X\rightarrow[0,1]$ be continuous and $0<r<1$. Then
\[f_\beta^{-1}\big([0,r)\big)\subseteq\mathrm{int}_{\beta X}\mathrm{cl}_{\beta X}f^{-1}\big([0,r)\big).\]
\end{lemma}

\begin{proof}
Note that
\[f_\beta^{-1}\big([0,r)\big)\subseteq\mathrm{int}_{\beta X}\mathrm{cl}_{\beta X}f_\beta^{-1}\big([0,r)\big)\]
and
\[\mathrm{cl}_{\beta X}f_\beta^{-1}\big([0,r)\big)=\mathrm{cl}_{\beta X}\big(X\cap f_\beta^{-1}\big([0,r)\big)\big)=\mathrm{cl}_{\beta X}f^{-1}\big([0,r)\big).\]
\end{proof}

For a space $X$ denote by $\mathscr{C}(X)$ the set of all closed subspaces of $X$.

\begin{proposition}\label{KJHF}
Let $X$ be a completely regular space and let $\mathscr{I}$ be an ideal in an upper semi-lattice $(\mathscr{L},\subseteq)$, where $\mathscr{L}\subseteq\mathscr{C}(X)$. Then
\[\lambda_{\mathscr I} X=\bigcup\{\mathrm{int}_{\beta X}\mathrm{cl}_{\beta X}U:U\mbox{ is null}\,\}.\]
\end{proposition}

\begin{proof}
Denote
\[T=\bigcup\{\mathrm{int}_{\beta X}\mathrm{cl}_{\beta X}U:U\mbox{ is null}\}.\]
It follows trivially from the definition of $\lambda_{\mathscr I} X$ that $\lambda_{\mathscr I} X\subseteq T$; we need to show the reverse inclusion. Let $t\in T$. Note that $T$ is an open subspace of $\beta X$. There exists a continuous $f:\beta X\rightarrow[0,1]$ with
\[f(t)=0\;\;\;\;\mbox{ and }\;\;\;\;f|_{\beta X\setminus T}=\mathbf{1}.\]
Observe that
$f^{-1}([0,1/2])$ is contained in $T$, and it is compact, as it is closed in $\beta X$. Therefore
\begin{equation}\label{FGFF}
f^{-1}\big([0,1/2]\big)\subseteq\mathrm{int}_{\beta X}\mathrm{cl}_{\beta X}U_1\cup\cdots\cup\mathrm{int}_{\beta X}\mathrm{cl}_{\beta X}U_n,
\end{equation}
where $U_i$ is null for each $i=1,\ldots,n$. Note that by our assumption the null elements are closed in $X$. Thus, if we intersect the two sides of (\ref{FGFF}) with $X$ we have
\[X\cap f^{-1}\big([0,1/2]\big)\subseteq\mathrm{cl}_XU_1\cup\cdots\cup\mathrm{cl}_XU_n=U_1\cup\cdots\cup U_n.\]
Observe that
\[U_1\cup\cdots\cup U_n\subseteq U_1\vee\cdots\vee U_n=U.\]
Let $C=X\cap f^{-1}([0,1/3))$. Then $C\in\mathrm{Coz}(X)$ as $f^{-1}([0,1/3))\in\mathrm{Coz}(\beta X)$. (To see the latter, let
\[g=\max\Big\{\mathbf{0},\mathbf{\frac{1}{3}}-f\Big\}\]
and observe that $\mathrm{coz}(g)=f^{-1}([0,1/3))$.) Now $\mathrm{cl}_XC$ has a null neighborhood in $X$, namely $U$. Therefore $\mathrm{int}_{\beta X}\mathrm{cl}_{\beta X}C\subseteq\lambda_{\mathscr I} X$. But then $t\in\lambda_{\mathscr I} X$, as $t\in f^{-1}([0,1/3))$ and  $f^{-1}([0,1/3))\subseteq\mathrm{int}_{\beta X}\mathrm{cl}_{\beta X}C$ by Lemma \ref{LKG}. This shows that $T\subseteq\lambda_{\mathscr I} X$, and concludes the proof.
\end{proof}

\begin{definition}
Let $X$ be a space and let $\mathscr{I}$ be an ideal in an upper semi-lattice $(\mathscr{L},\subseteq)$, where $\mathscr{L}\subseteq\mathscr{P}(X)$. Then $X$ is called \textit{locally null} (with respect to $\mathscr{I}$) if every point of $X$ has a null neighborhood in $X$.
\end{definition}

Observe that if $f:\beta X\rightarrow[0,1]$ is continuous then $(f|_X)_\beta=f$, as they are both continuous and coincide on the dense subspace $X$ of $\beta X$.

Let $X$ be a space. For an ideal $H$ of $C_b(X)$ the \textit{hull} of $H$ is defined to be
\[\mathfrak{h}(H)=\bigcap\big\{\mathrm{z}(h):h\in H\big\};\]
the ideal $H$ is said to be \textit{of empty hull} (or \textit{free}) if $\mathfrak{h}(H)=\emptyset$.

\begin{theorem}\label{BBV}
Let $X$ be a completely regular space and let $\mathscr{I}$ be an ideal in an upper semi-lattice $(\mathscr{L},\subseteq)$, where $\mathscr{L}\subseteq\mathscr{P}(X)$. The following are equivalent:
\begin{itemize}
\item[\rm(1)] $X\subseteq\lambda_{\mathscr I} X$.
\item[\rm(2)] $X$ is locally null.
\item[\rm(3)] $C^{\mathscr I}_{00}(X)$ is of empty hull.
\end{itemize}
\end{theorem}

\begin{proof}
(1) {\em implies} (2). Let $x\in X$. Then $x\in\lambda_{\mathscr I} X$ and therefore $x\in\mathrm{int}_{\beta X}\mathrm{cl}_{\beta X}D$ for some $D\in\mathrm{Coz}(X)$ such that $\mathrm{cl}_XD$ has a null neighborhood $V$ in $X$. But $V$ is then a neighborhood of $x$ in $X$ as well, as $x\in\mathrm{cl}_{\beta X}D\cap X=\mathrm{cl}_XD$.

(2) {\em implies} (1). Let $x\in X$ and let $U$ be a null neighborhood of $x$ in $X$. Let $f:X\rightarrow[0,1]$ be continuous with $f(x)=0$ and $f|_{X\setminus \mathrm{int}_XU}=\mathbf{1}$. Let $C=f^{-1}([0,1/2))$. Then $C\in\mathrm{Coz}(X)$. Now $\mathrm{cl}_XC\subseteq f^{-1}([0,1/2])$ and $f^{-1}([0,1/2])\subseteq\mathrm{int}_XU$. Thus $U$ is a null neighborhood of $\mathrm{cl}_XC$ in $X$. Therefore $\mathrm{int}_{\beta X}\mathrm{cl}_{\beta X}C\subseteq\lambda_{\mathscr I} X$. But then $x\in\lambda_{\mathscr I} X$, as $x\in f_\beta^{-1}([0,1/2))$ and  $f_\beta^{-1}([0,1/2))\subseteq\mathrm{int}_{\beta X}\mathrm{cl}_{\beta X}C$ by Lemma \ref{LKG}.

(2) {\em implies} (3). Note that $C^{\mathscr I}_{00}(X)$ is an ideal in $C_b(X)$ by Theorem \ref{TTG}. Let $x\in X$ and let $V$ be a null neighborhood of $x$ in $X$. Let $W$ be an open neighborhood of $x$ in $X$ with $\mathrm{cl}_XW\subseteq\mathrm{int}_XV$. Let $g:X\rightarrow[0,1]$ be continuous with $g(x)=1$ and $g|_{X\setminus W}=\mathbf{0}$. Then $\mathrm{supp}(g)\subseteq\mathrm{cl}_XW$, as $\mathrm{coz}(g)\subseteq W$. Thus $V$ is a null neighborhood of $\mathrm{supp}(g)$ in $X$. Therefore $g\in C^{\mathscr I}_{00}(X)$.

(3) {\em implies} (2). Let $x\in X$. Then $x\notin\mathrm{z}(h)$ for some $h\in C^{\mathscr I}_{00}(X)$. Since $\mathrm{supp}(h)$ has a null neighborhood in $X$ and $x\in\mathrm{supp}(h)$, as $x\in\mathrm{coz}(h)$, it then follows that $x$ has a null neighborhood in $X$.
\end{proof}

\begin{lemma}\label{HDHD}
Let $X$ be a completely regular space and let $\mathscr{I}$ be an ideal in an upper semi-lattice $(\mathscr{L},\subseteq)$, where $\mathscr{L}\subseteq\mathscr{P}(X)$. For any subspace $A$ of $X$, if $\mathrm{cl}_{\beta X}A\subseteq\lambda_{\mathscr I} X$ then $\mathrm{cl}_XA$ has a null neighborhood in $X$.
\end{lemma}

\begin{proof}
By compactness, we have
\begin{equation}\label{TDJB}
\mathrm{cl}_{\beta X}A\subseteq\mathrm{int}_{\beta X}\mathrm{cl}_{\beta X}C_1\cup\cdots\cup\mathrm{int}_{\beta X}\mathrm{cl}_{\beta X}C_n,
\end{equation}
where $C_i\in\mathrm{Coz}(X)$ for each $i=1,\ldots,n$, and $\mathrm{cl}_XC_i$ has a null neighborhood $U_i$ in $X$. Intersecting both sides of (\ref{TDJB}) with $X$ we have
\[\mathrm{cl}_XA\subseteq\mathrm{cl}_XC_1\cup\cdots\cup\mathrm{cl}_XC_n.\]
Since
\[\mathrm{cl}_XC_1\cup\cdots\cup\mathrm{cl}_XC_n\subseteq U_1\cup\cdots\cup U_n\subseteq U_1\vee\cdots\vee U_n=W\]
it follows that $W$ is a null neighborhood of $\mathrm{cl}_XA$ in $X$.
\end{proof}

\begin{theorem}\label{HGBV}
Let $X$ be a completely regular space locally null with respect to an ideal $\mathscr{I}$ of an upper semi-lattice $(\mathscr{L},\subseteq)$, where $\mathscr{L}\subseteq\mathscr{P}(X)$. The following are equivalent:
\begin{itemize}
\item[\rm(1)] $\lambda_{\mathscr I} X$ is compact.
\item[\rm(2)] ${\mathscr I}$ is non-proper.
\item[\rm(3)] $C^{\mathscr I}_{00}(X)$ is unital.
\end{itemize}
\end{theorem}

\begin{proof}
(1) {\em  implies} (2). Note that $X\subseteq\lambda_{\mathscr I} X$ by Theorem \ref{BBV}, as $X$ is locally null. Since $\lambda_{\mathscr I} X$ is compact we have $\mathrm{cl}_{\beta X}X\subseteq\lambda_{\mathscr I} X$. Therefore $X$ is null by Lemma \ref{HDHD}.

(2) {\em  implies} (3). Suppose that $X$ is null. Then the mapping $\mathbf{1}$ is the unit element of $C^{\mathscr I}_{00}(X)$.

(3) {\em  implies} (1). Suppose that $C^{\mathscr I}_{00}(X)$ has a unit element $u$. Let $x\in X$. Let $U_x$ be a null neighborhood of $x$ in $X$ and let $V_x$ be an open neighborhoods of $x$ in $X$ such that $\mathrm{cl}_XV_x\subseteq\mathrm{int}_X U_x$. Let $f_x:X\rightarrow[0,1]$ be continuous and such that $f_x(x)=1$ and $f_x|_{X\setminus V_x}=\mathbf{0}$. Then $U_x$ is a neighborhood of $\mathrm{supp}(f_x)$ in $X$, as $\mathrm{supp}(f_x)\subseteq\mathrm{cl}_XV_x$. Therefore $f_x\in C^{\mathscr I}_{00}(X)$. We have
\[u(x)=u(x)f_x(x)=f_x(x)=1.\]
Thus $u=\mathbf{1}$ and therefore $X=\mathrm{supp}(u)$ is null. Since $X\in\mathrm{Coz}(X)$ trivially, it follows that $\lambda_{\mathscr I} X=\beta X$ is compact.
\end{proof}

\begin{definition}\label{WWA}
Let $X$ be a completely regular space locally null with respect to an ideal $\mathscr{I}$ of an upper semi-lattice $(\mathscr{L},\subseteq)$, where $\mathscr{L}\subseteq\mathscr{P}(X)$. For any $f\in C_b(X)$ denote
\[f_\lambda=f_\beta|_{\lambda_{\mathscr I} X}.\]
\end{definition}

Observe that by Theorem \ref{BBV} the mapping $f_\lambda$ extends $f$.

\begin{lemma}\label{TES}
Let $X$ be a normal space locally null with respect to an ideal $\mathscr{I}$ of an upper semi-lattice $(\mathscr{L},\subseteq)$, where $\mathscr{L}\subseteq\mathscr{P}(X)$. For any $f\in C_b(X)$ the following are equivalent:
\begin{itemize}
\item[\rm(1)] $f\in C^{\mathscr I}_{00}(X)$.
\item[\rm(2)] $f_\lambda\in C_{00}(\lambda_{\mathscr I} X)$.
\end{itemize}
\end{lemma}

\begin{proof}
Note that $X\subseteq\lambda_{\mathscr I} X$ by Theorem \ref{BBV}, as $X$ is locally null.

(1) {\em  implies} (2). Let $U$ be a null neighborhood of $\mathrm{supp}(f)$ in $X$. Thus $\mathrm{supp}(f)\subseteq\mathrm{int}_X U$. Since $X$ is normal, by the Urysohn Lemma, there exists a continuous $g:X\rightarrow[0,1]$ such that
\[g|_{\mathrm{supp}(f)}=\mathbf{0}\;\;\;\;\mbox{ and }\;\;\;\;g|_{X\setminus\mathrm{int}_X U}=\mathbf{1}.\]
Let
\[C=g^{-1}\big([0,1/2)\big)\in\mathrm{Coz}(X).\]
Then $\mathrm{cl}_XC$ has a null neighborhood in $X$, namely $U$, as
\[\mathrm{cl}_XC\subseteq g^{-1}\big([0,1/2]\big)\subseteq\mathrm{int}_X U.\]
Therefore $\mathrm{int}_{\beta X}\mathrm{cl}_{\beta X}C\subseteq\lambda_{\mathscr I} X$. But $g_\beta^{-1}([0,1/2))\subseteq\mathrm{int}_{\beta X}\mathrm{cl}_{\beta X}C$ by Lemma \ref{LKG}, and thus
\[\mathrm{cl}_{\beta X}\mathrm{coz}(f)\subseteq\mathrm{z}(g_\beta)\subseteq g_\beta^{-1}\big([0,1/2)\big)\subseteq\lambda_{\mathscr I} X.\]
This implies that
\begin{eqnarray*}
\mathrm{supp}(f_\lambda)=\mathrm{cl}_{\lambda_{\mathscr I} X}\mathrm{coz}(f_\lambda)&=&\mathrm{cl}_{\lambda_{\mathscr I} X}\big(X\cap \mathrm{coz}(f_\lambda)\big)\\&=&\mathrm{cl}_{\lambda_{\mathscr I} X}\mathrm{coz}(f)=\lambda_{\mathscr I}X\cap\mathrm{cl}_{\beta X}\mathrm{coz}(f)=\mathrm{cl}_{\beta X}\mathrm{coz}(f)
\end{eqnarray*}
is compact, as it is closed in $\beta X$.

(2) {\em  implies} (1). Note that $\mathrm{cl}_{\beta X}(\mathrm{supp}(f))\subseteq\mathrm{supp}(f_\lambda)$, as $\mathrm{supp}(f)\subseteq\mathrm{supp}(f_\lambda)$ and the latter is compact. Thus $\mathrm{cl}_{\beta X}(\mathrm{supp}(f))\subseteq\lambda_{\mathscr I}X$. By Lemma \ref{HDHD} it follows that $\mathrm{supp}(f)$ has a null neighborhood in $X$.
\end{proof}

A version of the classical Banach--Stone Theorem (see Theorem 7.1 of \cite{Be}) states that for any locally compact Hausdorff spaces $X$ and $Y$, the rings $C_{00}(X)$ and $C_{00}(Y)$ are isomorphic if and only if the spaces $X$ and $Y$ are homeomorphic. (See \cite{A}.) This will be used in the proof of the following theorem.

\begin{theorem}\label{UUS}
Let $X$ be a normal space locally null with respect to an ideal $\mathscr{I}$ of an upper semi-lattice $(\mathscr{L},\subseteq)$, where $\mathscr{L}\subseteq\mathscr{P}(X)$. Then $C^{\mathscr I}_{00}(X)$ is a normed algebra isometrically isomorphic to $C_{00}(Y)$ for some unique (up to homeomorphism) locally compact Hausdorff space $Y$, namely $Y=\lambda_{\mathscr I} X$. Furthermore,
\begin{itemize}
\item[\rm(1)] $X$ is dense in $Y$.
\item[\rm(2)] $C^{\mathscr I}_{00}(X)$ is of empty hull.
\item[\rm(3)] $C^{\mathscr I}_{00}(X)$ is unital if and only if ${\mathscr I}$ is non-proper if and only if $Y$ is compact.
\end{itemize}
\end{theorem}

\begin{proof}
Observe that $C^{\mathscr I}_{00}(X)$ is a normed subalgebra of $C_b(X)$ by Theorem \ref{TTG}. Define
\[\psi:C^{\mathscr I}_{00}(X)\rightarrow C_{00}(\lambda_{\mathscr I} X)\]
by
\[\psi(f)=f_\lambda\]
for any $f\in C^{\mathscr I}_{00}(X)$. By Lemma \ref{TES} the mapping $\psi$ is well-defined. It is clear that $\psi$ is an algebra homomorphism and that it is injective. (Again, note that $X\subseteq\lambda_{\mathscr I} X$, and use the fact that any two scalar-valued continuous mapping on $\lambda_{\mathscr I} X$ coincide, provided that they agree on the dense subspace $X$ of $\lambda_{\mathscr I} X$.) To show that $\psi$ is surjective, let $g\in C_{00}(\lambda_{\mathscr I} X)$. Then $(g|_X)_\lambda=g$ and thus  $g|_X\in C^{\mathscr I}_{00}(X)$ by Lemma \ref{TES}. Note that $\psi(g|_X)=g$. To show that $\psi$ is an isometry, let $h\in C^{\mathscr I}_{00}(X)$. Then
\[|h_\lambda|(\lambda_{\mathscr I} X)=|h_\lambda|(\mathrm{cl}_{\lambda_{\mathscr I} X}X)\subseteq\mathrm{cl}_{\mathbb{R}}\big(|h_\lambda|(X)\big)=\mathrm{cl}_{\mathbb{R}}\big(|h|(X)\big)\subseteq\big[0,\|h\|\big]\]
which yields $\|h_\lambda\|\leq\|h\|$. That $\|h\|\leq\|h_\lambda\|$ is clear, as $h_\lambda$ extends $h$.

Note that $\lambda_{\mathscr I} X$ is locally compact, as it is open in the compact Hausdorff space $\beta X$.

The uniqueness of $\lambda_{\mathscr I} X$ follows from the fact that for any locally compact Hausdorff space $T$ the ring $C_{00}(T)$ determines the topology of $T$.

(1). By Theorem \ref{BBV} we have $X\subseteq\lambda_{\mathscr I} X$. That $X$ is dense in $\lambda_{\mathscr I} X$ is then obvious.

(2). This follows from Theorem \ref{BBV}.

(3). This follows from Theorem \ref{HGBV}.
\end{proof}

Under certain conditions the representation given for $C^{\mathscr I}_{00}(X)$ simplifies. This is the context of our next result. First, we need a lemma.

\begin{lemma}\label{GYUS}
Let $X$ be a Lindel\"{o}f space locally null with respect to a $\sigma$-ideal $\mathscr{I}$ of an upper semi-lattice $(\mathscr{L},\subseteq)$, where $\mathscr{L}\subseteq\mathscr{P}(X)$. Then the closure in $X$ of each subset of $X$ has a null neighborhood in $X$.
\end{lemma}

\begin{proof}
Let $A$ be a subset of $X$. Since $X$ is locally null, there exists a null neighborhood $U_x$ of $x$ in $X$ for each $x\in X$. Since
\[\mathrm{cl}_XA\subseteq\bigcup_{x\in X}\mathrm{int}_XU_x\]
and $\mathrm{cl}_XA$ is Lindel\"{o}f, as it is closed in $X$ and $X$ is so, we have
\[\mathrm{cl}_XA\subseteq\bigcup_{n=1}^\infty\mathrm{int}_XU_{x_n}.\]
But then, since $\mathscr{I}$ is a $\sigma$-ideal, we have
\[\bigcup_{n=1}^\infty\mathrm{int}_XU_{x_n}\subseteq\bigcup_{n=1}^\infty U_{x_n}\subseteq\bigvee_{n=1}^\infty U_{x_n}=W.\]
That is, $\mathrm{cl}_XA$ has a null neighborhood in $X$, namely $W$.
\end{proof}

Recall that $\mathscr{RC}(X)$ denotes the set of all regular closed subspaces of a space $X$.

\begin{theorem}\label{GRHS}
Let $X$ be a Lindel\"{o}f space locally null with respect to a $\sigma$-ideal $\mathscr{I}$ of an upper semi-lattice $(\mathscr{L},\subseteq)$, where $\mathscr{L}\subseteq\mathscr{P}(X)$. Let $\mathscr{RC}(X)\subseteq\mathscr{L}$. Then
\[C^{\mathscr I}_{00}(X)=\big\{f\in C_b(X):\mathrm{supp}(f)\mbox{ is null}\big\}.\]
\end{theorem}

\begin{proof}
Let $f\in C^{\mathscr I}_{00}(X)$. Then $\mathrm{supp}(f)\subseteq U$ for some null $U$. Since $\mathrm{supp}(f)\in\mathscr{L}$, as $\mathrm{supp}(f)$ is regular closed in $X$, it follows that $\mathrm{supp}(f)$ is null. The converse trivially follows from Lemma \ref{GYUS}.
\end{proof}

\section{The Banach algebra $C^{\mathscr I}_0(X)$}\label{JGGG}

Let $X$ be a space and let $\mathscr{I}$ be an ideal in an upper semi-lattice $(\mathscr{L},\subseteq)$, where $\mathscr{L}\subseteq\mathscr{P}(X)$. In this section we consider the subset $C^{\mathscr I}_0(X)$ of $C_b(X)$ consisting of those elements $f$ of $C_b(X)$ such that $|f|^{-1}([1/n,\infty))$ has a null neighborhood in $X$ for each positive integer $n$. The set $C^{\mathscr I}_0(X)$ coincides with $C_0(X)$ if $X$ is locally compact and $\mathscr{I}$ is the set of all subspaces of $X$ with compact closure (and $\mathscr{L}=\mathscr{P}(X)$). Under certain conditions the expression of $C^{\mathscr I}_0(X)$ simplifies; in particular
\[C^{\mathscr I}_0(X)=\big\{f\in C_b(X):|f|^{-1}\big([1/n,\infty)\big)\mbox{ is null for each }n\big\}.\]
if $\mathscr{L}$ contains the set $\mathscr{Z}(X)$ of all zero-sets of $X$. We show that $C^{\mathscr I}_0(X)$ is in general a closed ideal (and in particular, a Banach subalgebra) in $C_b(X)$ containing $C^{\mathscr I}_{00}(X)$. Furthermore, if $X$ is completely regular, then $C^{\mathscr I}_0(X)$ is of empty hull if and only if $X$ is locally null, and if so, then $C^{\mathscr I}_0(X)$ is unital if and only if $\mathscr{I}$ is non-proper. The main result of this section states that if $X$ is normal and locally null then the Banach algebra $C^{\mathscr I}_0(X)$ is isometrically isomorphic to $C_0(Y)$ for some unique (up to homeomorphism) locally compact Hausdorff space $Y$. Furthermore, $C^{\mathscr I}_0(X)$ contains $C^{\mathscr I}_{00}(X)$ densely, $Y$ contains $X$ densely, and $Y$ is compact if and only if $C^{\mathscr I}_0(X)$ is unital. Also, if $X$ is moreover Lindel\"{o}f and $\mathscr{I}$ is a $\sigma$-ideal, then $Y$ is countably compact and $C^{\mathscr I}_0(X)=C^{\mathscr I}_{00}(X)$.

We now proceed with the formal treatment of the subject.

\begin{definition}\label{HHLG}
Let $X$ be a space and let $\mathscr{I}$ be an ideal in an upper semi-lattice $(\mathscr{L},\subseteq)$, where $\mathscr{L}\subseteq\mathscr{P}(X)$. Define
\[C^{\mathscr I}_0(X)=\big\{f\in C_b(X):|f|^{-1}\big([1/n,\infty)\big)\mbox{ has a null neighborhood in $X$ for each }n\big\}.\]
\end{definition}

The following is to justify our use of the notation $C^{\mathscr I}_0(X)$.

\begin{example}\label{GHUJ}
Let $X$ be a locally compact space. Consider the ideal
\[\mathscr{I}=\{A\subseteq X:\mathrm{cl}_X A \mbox{ is compact}\}.\]
of $(\mathscr{P}(X),\subseteq)$. Then, an argument analogous to the one given in Example \ref{HUJ} shows that
\[C^{\mathscr I}_0(X)=C_0(X).\]
\end{example}

Under certain conditions the representation of $C^{\mathscr I}_0(X)$ given in Definition \ref{HHLG} simplifies. This is the context of the next few results. Recall that $\mathscr{Z}(X)$ and $\mathrm{Coz}(X)$ denote the set of all zero-sets and the set of all cozero-sets of a space $X$, respectively.

\begin{proposition}\label{JJG}
Let $X$ be a space and let $\mathscr{I}$ be an ideal in an upper semi-lattice $(\mathscr{L},\subseteq)$, where $\mathscr{L}\subseteq\mathscr{P}(X)$. Let $\mathscr{Z}(X)\subseteq\mathscr{L}$. Then
\[C^{\mathscr I}_0(X)=\big\{f\in C_b(X):|f|^{-1}\big([1/n,\infty)\big)\mbox{ is null for each }n\big\}.\]
\end{proposition}

\begin{proof}
Let $f\in C^{\mathscr I}_0(X)$. Let $n$ be a positive integer. Then $|f|^{-1}([1/n,\infty))$ is contained in a null subset of $X$, and is therefore null itself. (Note that $|f|^{-1}([1/n,\infty))$ is contained in $\mathscr{L}$, as it is contained in $\mathscr{Z}(X)$; to see the latter, let
\[g=\max\Big\{\mathbf{0},\mathbf{\frac{1}{n}}-|f|\Big\}\]
and observe that $\mathrm{z}(g)=|f|^{-1}([1/n,\infty))$.)

Next, let $f\in C_b(X)$ such that $|f|^{-1}([1/n,\infty))$ is null for each positive integer $n$. Since
\[|f|^{-1}\big([1/n,\infty)\big)\subseteq|f|^{-1}\Big(\Big(\frac{1}{n+1},\infty\Big)\Big)\subseteq|f|^{-1}\Big(\Big[\frac{1}{n+1},\infty\Big)\Big),\]
it follows that the latter is a null neighborhood of $|f|^{-1}([1/n,\infty))$ in $X$, for each positive integer $n$. That is $f\in C^{\mathscr I}_0(X)$.
\end{proof}

\begin{proposition}\label{JFJG}
Let $X$ be a space and let $\mathscr{I}$ be a $\sigma$-ideal in an upper semi-lattice $(\mathscr{L},\subseteq)$, where $\mathscr{L}\subseteq\mathscr{P}(X)$. Then
\[C^{\mathscr I}_0(X)=\big\{f\in C_b(X):\mathrm{coz}(f)\mbox{ has a null neighborhood in }X\big\}.\]
In particular, if $\mathrm{Coz}(X)\subseteq\mathscr{L}$, then
\[C^{\mathscr I}_0(X)=\big\{f\in C_b(X):\mathrm{coz}(f)\mbox{ is null}\big\}.\]
\end{proposition}

\begin{proof}
We verify only the first representation of $C^{\mathscr I}_0(X)$; the second representation follows readily from this, as if $\mathrm{Coz}(X)\subseteq\mathscr{L}$, then $\mathrm{coz}(f)$ has a null neighborhood in $X$, where $f\in C_b(X)$, if and only if $\mathrm{coz}(f)$ is null itself.

Let $f\in C^{\mathscr I}_0(X)$. Let $n$ be a positive integer. Then $|f|^{-1}([1/n,\infty))$ has a null neighborhood $U_n$ in $X$.
We have
\[\mathrm{coz}(f)=\bigcup_{n=1}^\infty|f|^{-1}\big([1/n,\infty)\big)\subseteq\bigcup_{n=1}^\infty\mathrm{int}_XU_n\subseteq\bigcup_{n=1}^\infty U_n\subseteq\bigvee_{n=1}^\infty U_n=V.\]
Thus $\mathrm{coz}(f)$ has a null neighborhood in $X$, namely $V$.

The converse is trivial, as if $\mathrm{coz}(f)$ has a null neighborhood in $X$, where $f\in C_b(X)$, then so does $|f|^{-1}([1/n,\infty))$ for each positive integer $n$, as it is contained in $\mathrm{coz}(f)$.
\end{proof}

\begin{theorem}\label{DGH}
Let $X$ be a space and let $\mathscr{I}$ be an ideal in an upper semi-lattice $(\mathscr{L},\subseteq)$, where $\mathscr{L}\subseteq\mathscr{P}(X)$. Then
$C^{\mathscr I}_0(X)$ is a closed ideal in $C_b(X)$. Furthermore,
\[C^{\mathscr I}_{00}(X)\subseteq C^{\mathscr I}_0(X).\]
\end{theorem}

\begin{proof}
We start from the last statement. Let $f\in C^{\mathscr I}_{00}(X)$. Let $n$ be a positive integer. Then $|f|^{-1}([1/n,\infty))$ has a null neighborhood in $X$, as it is contained in $\mathrm{supp}(f)$ and the latter does. Thus $f\in C^{\mathscr I}_0(X)$.

Next, we show that $C^{\mathscr I}_0(X)$ is an ideal in the algebra $C_b(X)$. To show that $C^{\mathscr I}_0(X)$ is closed under addition, let $f,g\in C^{\mathscr I}_0(X)$. Let $n$ be a positive integer. There exist null neighborhoods $U$ and $V$ of $|f|^{-1}([1/(2n),\infty))$ and $|g|^{-1}([1/(2n),\infty))$ in $X$, respectively. Then
\begin{eqnarray*}
|f+g|^{-1}\big([1/n,\infty)\big)&\subseteq&|f|^{-1}\Big(\Big[\frac{1}{2n},\infty\Big)\Big)\cup|g|^{-1}\Big(\Big[\frac{1}{2n},\infty\Big)\Big)\\&\subseteq&\mathrm{int}_X U\cup\mathrm{int}_X V\subseteq U\cup V\subseteq U\vee V.
\end{eqnarray*}
Thus $|f+g|^{-1}([1/n,\infty))$ has a null neighborhood in $X$, namely $U\vee V$. Therefore $f+g\in C^{\mathscr I}_0(X)$. Next, let $f\in C^{\mathscr I}_0(X)$ and $g\in C_b(X)$. Let $m$ be a positive integer such that $|g(x)|\leq m$ for each $x\in X$. Since
\[|fg|^{-1}\big([1/n,\infty)\big)\subseteq|f|^{-1}\Big(\Big[\frac{1}{mn},\infty\Big)\Big),\]
and $|f|^{-1}([1/(mn),\infty))$ has a null neighborhoods in $X$ it follows that $|fg|^{-1}([1/n,\infty))$ has a null neighborhoods in $X$. Therefore $fg\in C^{\mathscr I}_0(X)$. That $C^{\mathscr I}_0(X)$ is closed under scalar multiplication follows analogously.

Finally, we show that $C^{\mathscr I}_0(X)$ is closed in $C_b(X)$. Let $f$ be in the closure in $C_b(X)$ of $C^{\mathscr I}_0(X)$. Let $n$ be a positive integer. There exists some $g\in C^{\mathscr I}_0(X)$ with $\|f-g\|<1/(2n)$. Let $t\in |f|^{-1}([1/n,\infty))$. Then
\[\frac{1}{n}\leq\big|f(t)\big|\leq\big|f(t)-g(t)\big|+\big|g(t)\big|\leq\|f-g\|+\big|g(t)\big|\leq\frac{1}{2n}+\big|g(t)\big|\]
and thus $|g(t)|\geq 1/(2n)$. That is $t\in |g|^{-1}([1/(2n),\infty))$. Therefore
\[|f|^{-1}\big([1/n,\infty)\big)\subseteq|g|^{-1}\Big(\Big[\frac{1}{2n},\infty\Big)\Big).\]
Since the latter has a null neighborhood in $X$, so does $|f|^{-1}([1/n,\infty))$. Thus $f\in C^{\mathscr I}_0(X)$.
\end{proof}

\begin{theorem}\label{KGV}
Let $X$ be a completely regular space and let $\mathscr{I}$ be an ideal in an upper semi-lattice $(\mathscr{L},\subseteq)$, where $\mathscr{L}\subseteq\mathscr{P}(X)$. The following are equivalent:
\begin{itemize}
\item[\rm(1)] $X\subseteq\lambda_{\mathscr I} X$.
\item[\rm(2)] $X$ is locally null.
\item[\rm(3)] $C^{\mathscr I}_0(X)$ is of empty hull.
\item[\rm(4)] $C^{\mathscr I}_{00}(X)$ is of empty hull.
\end{itemize}
\end{theorem}

\begin{proof}
The equivalence of (1), (2) and (4) follows from Theorem \ref{BBV}.

(4) {\em implies} (3). Note that $C^{\mathscr I}_0(X)$ contains $C^{\mathscr I}_{00}(X)$ by Theorem \ref{DGH}. Thus $C^{\mathscr I}_0(X)$ is of empty hull if $C^{\mathscr I}_{00}(X)$ is so.

(3) {\em implies} (2). Let $x\in X$. Then $x\notin\mathrm{z}(f)$ for some $f\in C^{\mathscr I}_0(X)$. That is $|f(x)|>0$. Let $n$ be a positive integer such that $|f(x)|\geq1/n$. Then $x\in |f|^{-1}([1/n,\infty))$. Since $|f|^{-1}([1/n,\infty))$ has a null neighborhood in $X$ it follows that $x$ has a null neighborhood in $X$.
\end{proof}

\begin{theorem}\label{HFBY}
Let $X$ be a completely regular space locally null with respect to an ideal $\mathscr{I}$ of an upper semi-lattice $(\mathscr{L},\subseteq)$, where $\mathscr{L}\subseteq\mathscr{P}(X)$. The following are equivalent:
\begin{itemize}
\item[\rm(1)] $\lambda_{\mathscr I} X$ is compact.
\item[\rm(2)] ${\mathscr I}$ is non-proper.
\item[\rm(3)] $C^{\mathscr I}_0(X)$ is unital.
\item[\rm(4)] $C^{\mathscr I}_{00}(X)$ is unital.
\end{itemize}
\end{theorem}

\begin{proof}
The equivalence of (1), (2) and (4) follows from Theorem \ref{HGBV}.

(2) {\em implies} (3). Suppose that $X$ is null. Then the mapping $\mathbf{1}$ is the unit element of $C^{\mathscr I}_0(X)$.

(3) {\em implies} (2).  Suppose that $C^{\mathscr I}_0(X)$ has a unit element $u$. Let $U_x$, $V_x$ and $f_x$ be as defined in the proof of Theorem \ref{HGBV}. Note that $f_x\in C^{\mathscr I}_0(X)$, as $f_x\in C^{\mathscr I}_{00}(X)$ and $C^{\mathscr I}_{00}(X)\subseteq C^{\mathscr I}_0(X)$ by Theorem \ref{DGH}. Arguing as in the proof of Theorem \ref{HGBV} we have $u=\mathbf{1}$. This implies that $X=u^{-1}([1,\infty))$ is null.
\end{proof}

\begin{lemma}\label{TTES}
Let $X$ be a normal space locally null with respect to an ideal $\mathscr{I}$ of an upper semi-lattice $(\mathscr{L},\subseteq)$, where $\mathscr{L}\subseteq\mathscr{P}(X)$. For any $f\in C_b(X)$ the following are equivalent:
\begin{itemize}
\item[\rm(1)] $f\in C^{\mathscr I}_0(X)$.
\item[\rm(2)] $f_\lambda\in C_0(\lambda_{\mathscr I} X)$.
\end{itemize}
\end{lemma}

\begin{proof}
Note that $X\subseteq\lambda_{\mathscr I} X$ by Theorem \ref{BBV}, as $X$ is locally null.

(1) {\em implies} (2). Let $k$ be a positive integer. Consider a null neighborhood $U_k$ of $|f|^{-1}([1/k,\infty))$ in $X$. Then $|f|^{-1}([1/k,\infty))\subseteq\mathrm{int}_X U_k$. Since $X$ is normal, by the Urysohn Lemma, there exists a continuous $g_k:X\rightarrow[0,1]$ such that
\[g_k|_{|f|^{-1}([1/k,\infty))}=\mathbf{0}\;\;\;\;\mbox{ and }\;\;\;\;g_k|_{X\setminus\mathrm{int}_X U_k}=\mathbf{1}.\]
Let
\[C_k=g_k^{-1}\big([0,1/2)\big)\in\mathrm{Coz}(X).\]
Then $\mathrm{cl}_XC_k$ has a null neighborhood in $X$, namely $U_k$, as
\[\mathrm{cl}_XC_k\subseteq g_k^{-1}\big([0,1/2]\big)\subseteq\mathrm{int}_XU_k.\]
Therefore $\mathrm{int}_{\beta X}\mathrm{cl}_{\beta X}C_k\subseteq\lambda_{\mathscr I} X$. Arguing as in the proof of Lemma \ref{LKG} we have
\[|f_\beta|^{-1}\big((1/k,\infty)\big)\subseteq\mathrm{cl}_{\beta X}\big(|f|^{-1}\big((1/k,\infty)\big)\big).\]
Since
\[\mathrm{cl}_{\beta X}\big(|f|^{-1}\big((1/k,\infty)\big)\big)\subseteq\mathrm{cl}_{\beta X}\big(\mathrm{z}(g_k)\big)\subseteq\mathrm{z}\big((g_k)_\beta\big)\subseteq(g_k)_\beta^{-1}\big([0,1/2)\big)\]
and
\[(g_k)_\beta^{-1}\big([0,1/2)\big)\subseteq\mathrm{int}_{\beta X}\mathrm{cl}_{\beta X}C_k\]
by Lemma \ref{LKG}, it follows that
\begin{equation}\label{DDSD}
|f_\beta|^{-1}\big((1/k,\infty)\big)\subseteq\lambda_{\mathscr I} X.
\end{equation}
Now, let $n$ be a positive integer. Using (\ref{DDSD}), we have
\[|f_\beta|^{-1}\big([1/n,\infty)\big)\subseteq|f_\beta|^{-1}\Big(\Big(\frac{1}{n+1},\infty\Big)\Big)\subseteq\lambda_{\mathscr I} X.\]
Therefore
\[|f_\lambda|^{-1}\big([1/n,\infty)\big)=\lambda_{\mathscr I} X\cap|f_\beta|^{-1}\big([1/n,\infty)\big)=|f_\beta|^{-1}\big([1/n,\infty)\big)\]
is compact, as it is closed in $\beta X$.

(2) {\em implies} (1). Let $n$ be a positive integer. Then since $|f_\lambda|^{-1}([1/n,\infty))$ contains $|f|^{-1}([1/n,\infty))$ and it is compact, we have
\[\mathrm{cl}_{\beta X}\big(|f|^{-1}\big([1/n,\infty)\big)\big)\subseteq|f_\lambda|^{-1}\big([1/n,\infty)\big)\subseteq\lambda_{\mathscr I} X.\]
But then $|f|^{-1}([1/n,\infty))$ has a null neighborhood in $X$ by Lemma \ref{HDHD}.
\end{proof}

There is a version of the classical Banach--Stone Theorem which states that for any locally compact Hausdorff spaces $X$ and $Y$, the rings $C_0(X)$ and $C_0(Y)$ are isomorphic if and only if the spaces $X$ and $Y$ are homeomorphic. (See \cite{A}.) This will be used in the proof of the following theorem.

\begin{theorem}\label{UDR}
Let $X$ be a normal space locally null with respect to an ideal $\mathscr{I}$ of an upper semi-lattice $(\mathscr{L},\subseteq)$, where $\mathscr{L}\subseteq\mathscr{P}(X)$. Then $C^{\mathscr I}_0(X)$ is a Banach algebra isometrically isomorphic to $C_0(Y)$ for some unique (up to homeomorphism) locally compact Hausdorff space $Y$, namely $Y=\lambda_{\mathscr I} X$. Furthermore,
\begin{itemize}
\item[\rm(1)] $X$ is dense in $Y$.
\item[\rm(2)] $C^{\mathscr I}_0(X)$ is of empty hull.
\item[\rm(3)] $C^{\mathscr I}_{00}(X)$ is dense in $C^{\mathscr I}_0(X)$.
\item[\rm(4)] $C^{\mathscr I}_0(X)$ is unital if and only if ${\mathscr I}$ is non-proper if and only if $Y$ is compact.
\end{itemize}
\end{theorem}

\begin{proof}
Observe that $C^{\mathscr I}_0(X)$ is a Banach subalgebra of $C_b(X)$ by Theorem \ref{DGH}. Define
\[\phi:C^{\mathscr I}_0(X)\rightarrow C_0(\lambda_{\mathscr I} X)\]
by
\[\phi(f)=f_\lambda\]
for any $f\in C^{\mathscr I}_0(X)$. By Lemma \ref{TTES} the mapping $\phi$ is well-defined, and arguing as in the proof of Theorem \ref{UUS}, it follows that $\phi$ is an isometric algebra isomorphism.

Note that $\lambda_{\mathscr I} X$ is locally compact, as it is open in the compact Hausdorff space $\beta X$, and $\lambda_{\mathscr I} X$ contains $X$ (as a dense subspace) by Theorem \ref{BBV}.

The uniqueness of $\lambda_{\mathscr I} X$ follows from the fact that the topology of any locally compact Hausdorff space $T$ is determined by the algebraic structure of the ring $C_0(T)$.

(1). By Theorem \ref{KGV} we have $X\subseteq\lambda_{\mathscr I} X$. That $X$ is dense in $\lambda_{\mathscr I} X$ is then obvious.

(2). This follows from Theorem \ref{KGV}.

(3). Let $\phi$ be as defined in the above and let $\psi=\phi|_{C^{\mathscr I}_{00}(X)}$. Then
\[\psi:C^{\mathscr I}_{00}(X)\rightarrow C_{00}(\lambda_{\mathscr I} X),\]
and $\psi$ is surjective by the proof of Theorem \ref{UUS}. The result now follows from the well known fact that $C_{00}(T)$ is dense in $C_0(T)$ for any locally compact Hausdorff space $T$.

(4). This follows from Theorem \ref{HFBY}.
\end{proof}

Our next result is dual to Theorem \ref{UDR}. It will require the space $X$ under consideration to be only completely regular at the price of limiting $\mathscr{I}$ to be an ideal in $(\mathrm{Coz}(X),\subseteq)$. (Recall that $(\mathrm{Coz}(X),\subseteq)$ is an upper semi-lattice, indeed a lattice, for any space $X$.) We need the following modification of Lemma \ref{TTES}. We will also need to use the well known fact that in a space any two disjoint zero-sets are \textit{completely separated}, in the sense that, there exists a continuous $[0,1]$-valued mapping on the space which equals to $\mathbf{0}$ on one and $\mathbf{1}$ on the other. (To see this, let $S$ and $T$ be a pair of disjoint zero-sets in a space $X$. Let $S=\mathrm{z}(f)$ and $T=\mathrm{z}(g)$ for some continuous $f,g:X\rightarrow[0,1]$. The mapping
\[h=\frac{f}{f+g}:X\rightarrow[0,1]\]
is well defined and continuous with $h|_S=\mathbf{0}$ and $h|_T=\mathbf{1}$.)

\begin{lemma}\label{FSS}
Let $X$ be a completely regular space locally null with respect to an ideal $\mathscr{I}$ of $(\mathrm{Coz}(X),\subseteq)$. For any $f\in C_b(X)$ the following are equivalent:
\begin{itemize}
\item[\rm(1)] $f\in C^\mathscr{I}_0(X)$.
\item[\rm(2)] $f_\lambda\in C_0(\lambda_\mathscr{I} X)$.
\end{itemize}
\end{lemma}

\begin{proof}
The proof is analogous to the proof we have already given for Lemma \ref{TTES}, except for the proof that (1) implies (2), which we now slightly modify. Let $k$ be a given positive integer. Let $U_k$ be a null neighborhood of $|f|^{-1}([1/k,\infty))$ in $X$. Note that $X\setminus U_k$ and $|f|^{-1}([1/k,\infty))$ are zero-sets in $X$ and they are disjoint. Thus $X\setminus U_k$ and $|f|^{-1}([1/k,\infty))$ are completely separated in $X$. Denote by $g_k:X\rightarrow[0,1]$ the continuous mapping such that
\[g_k|_{|f|^{-1}([1/k,\infty))}=\mathbf{0}\;\;\;\;\mbox{ and }\;\;\;\;g_k|_{X\setminus U_k}=\mathbf{1}.\]
The remaining part of the proof now follows from the same line of argument as given in the proof of Lemma \ref{TTES}.
\end{proof}

The proof of the following is analogous to the proof of Theorem \ref{UDR} with the usage of Lemma \ref{FSS} in place of Lemma \ref{TTES}.

\begin{theorem}\label{UF}
Let $X$ be a completely regular space locally null with respect to an ideal $\mathscr{I}$ of $(\mathrm{Coz}(X),\subseteq)$. Then $C^{\mathscr I}_0(X)$ is a Banach algebra isometrically isomorphic to $C_0(Y)$ for some unique (up to homeomorphism) locally compact Hausdorff space $Y$, namely $Y=\lambda_{\mathscr I} X$. Furthermore,
\begin{itemize}
\item[\rm(1)] $X$ is dense in $Y$.
\item[\rm(2)] $C^{\mathscr I}_0(X)$ is of empty hull.
\item[\rm(3)] $C^{\mathscr I}_{00}(X)$ is dense in $C^{\mathscr I}_0(X)$.
\item[\rm(4)] $C^{\mathscr I}_0(X)$ is unital if and only if ${\mathscr I}$ is non-proper if and only if $Y$ is compact.
\end{itemize}
\end{theorem}

\begin{remark}\label{HFJD}
Assuming that $C^{\mathscr I}_0(X)$ is a Banach algebra, it follows from the commutative Gelfand--Naimark Theorem that $C^{\mathscr I}_0(X)$ is isometrically isomorphic to $C_0(Y)$ for some locally compact Hausdorff space $Y$. Our approach here in Theorems \ref{UDR} and \ref{UF}, apart from the independence of its proof, has the advantage of giving extra information, about either the Banach algebra $C^{\mathscr I}_0(X)$ or the space $Y$, not generally expected to be deducible from the standard Gelfand Theory. This fact is particularly highlighted in the second part of the article in which we consider specific examples of the space $X$ or the ideal ${\mathscr I}$.
\end{remark}

Let $X$ be a locally compact Hausdorff space. It is known that $C_0(X)=C_{00}(X)$ if and only if every $\sigma$-compact subspace of $X$ is contained in a compact subspace of $X$. (See Problem 7G.2 of \cite{GJ}.) In particular, $C_0(X)=C_{00}(X)$ implies that $X$ is countably compact. (Recall that a space $X$ is \textit{countably compact} if every countable open cover of $X$ has a finite subcover, equivalently, if every countable infinite subspace of $X$ has an accumulation point in $X$; see Theorem 3.10.3 of \cite{E}.) This will be used in the proof of the following result which examines conditions under which $C^{\mathscr I}_{00}(X)$ and $C^{\mathscr I}_0(X)$ coincide.

\begin{theorem}\label{YUS}
Let $X$ be a regular Lindel\"{o}f space locally null with respect to a $\sigma$-ideal $\mathscr{I}$ of an upper semi-lattice $(\mathscr{L},\subseteq)$, where $\mathscr{L}\subseteq\mathscr{P}(X)$. Then $C^{\mathscr I}_0(X)$ is a Banach algebra isometrically isomorphic to $C_0(Y)$ for the unique (up to homeomorphism) locally compact Hausdorff space
\[Y=\bigcup\{\mathrm{int}_{\beta X}\mathrm{cl}_{\beta X}U:U\mbox{ is null}\,\},\]
considered as a subspace of $\beta X$. Furthermore,
\begin{itemize}
\item[\rm(1)] $X$ is dense in $Y$.
\item[\rm(2)] $C^{\mathscr I}_0(X)=C^{\mathscr I}_{00}(X)$.
\item[\rm(3)] $C_0(Y)=C_{00}(Y)$.
\item[\rm(4)] $Y$ is countably compact, and is compact, if and only if ${\mathscr I}$ is non-proper if and only if $C^{\mathscr I}_0(X)$ is unital.
\end{itemize}
\end{theorem}

\begin{proof}
Note that every regular Lindel\"{o}f space is normal. The first part of the theorem, as well as (1) and a part of (4), follow from Theorem \ref{UDR}, provided that we prove $Y=\lambda_\mathscr{I} X$; this we will do next.

Denote
\[T=\bigcup\{\mathrm{int}_{\beta X}\mathrm{cl}_{\beta X}U:U\mbox{ is null}\}.\]
We check that $T=\lambda_\mathscr{I} X$. By the definition of $\lambda_\mathscr{I} X$ it is clear that $\lambda_\mathscr{I} X\subseteq T$. Let $U$ be a null subset of $X$. Then $\mathrm{cl}_XU$ has a null neighborhood in $X$, say $W$, by Lemma \ref{GYUS}. By the Urysohn Lemma there exists a continuous $f:X\rightarrow[0,1]$ such that
\[f|_{\mathrm{cl}_X U}=\mathbf{1}\;\;\;\;\mbox{ and }\;\;\;\;f|_{X\setminus\mathrm{int}_X W}=\mathbf{0}.\]
Let $C=f^{-1}((1/2,1])$. Then $C\in\mathrm{Coz}(X)$ and $\mathrm{cl}_X C$ has a null neighborhood in $X$, namely, $W$ itself, as $\mathrm{cl}_X C\subseteq f^{-1}([1/2,1])$ and $f^{-1}([1/2,1])\subseteq\mathrm{int}_X W$. Therefore $\mathrm{int}_{\beta X}\mathrm{cl}_{\beta X}U\subseteq\lambda_\mathscr{I} X$, as $\mathrm{cl}_XU\subseteq C$ and $\mathrm{int}_{\beta X}\mathrm{cl}_{\beta X}C\subseteq\lambda_\mathscr{I} X$. This shows that $T\subseteq\lambda_\mathscr{I} X$.

Next, we show that every $\sigma$-compact subspace of $Y$ is contained in a compact subspace of $Y$; this will show (3). Since every countable set is $\sigma$-compact, this in particular proves that every countable infinite subspace of $Y$ has an accumulation point in $Y$. That is, $Y$ is countably compact, which shows (4).

Let $A$ be a $\sigma$-compact subspace of $Y$. Let
\[A=\bigcup_{n=1}^\infty A_n,\]
where $A_n$, for each positive integer $n$, is compact. Let $n$ be a positive integer. By compactness of $A_n$ we have
\begin{equation}\label{DSD}
A_n\subseteq\mathrm{int}_{\beta X}\mathrm{cl}_{\beta X}U^n_1\cup\cdots\cup\mathrm{int}_{\beta X}\mathrm{cl}_{\beta X}U^n_{k_n},
\end{equation}
where $U^n_i$ for each $i=1,\ldots,k_n$ is a null subset of $X$. Let
\[V=\bigvee_{n=1}^\infty\bigvee_{i=1}^{k_n}U^n_i.\]
That $V$ is well defined follows from the fact that $\mathscr{I}$ is a $\sigma$-ideal. By Lemma \ref{GYUS} there exists a null neighborhood $W$ of $\mathrm{cl}_X V$ in $X$. By the Urysohn Lemma there exists a continuous $g:X\rightarrow[0,1]$ with
\[g|_{\mathrm{cl}_X V}=\mathbf{0}\;\;\;\;\mbox{ and }\;\;\;\;g|_{X\setminus\mathrm{int}_X W}=\mathbf{1}.\]
We verify that $\mathrm{z}(g_\beta)$ is the desired compact subspace of $Y$ which contains $A$. Let $C=g^{-1}([0,1/2))$. Observe that  $g_\beta^{-1}([0,1/2))\subseteq\mathrm{int}_{\beta X}\mathrm{cl}_{\beta X}C$ by Lemma \ref{LKG}, $C\subseteq W$ and $\mathrm{int}_{\beta X}\mathrm{cl}_{\beta X}W\subseteq Y$. Thus $g_\beta^{-1}([0,1/2))\subseteq Y$. Since $\mathrm{z}(g_\beta)\subseteq g_\beta^{-1}([0,1/2))$ it follows that $\mathrm{z}(g_\beta)\subseteq Y$. Observe that $\mathrm{z}(g_\beta)$ is closed in $\beta X$ and is therefore compact. Let $n=1,2,\ldots$ and $i=1,\ldots,k_n$. Note that $U^n_i\subseteq\mathrm{z}(g)$, as $U^n_i\subseteq V$ and $V\subseteq\mathrm{z}(g)$. Therefore $\mathrm{cl}_{\beta X}U^n_i\subseteq\mathrm{cl}_{\beta X}\mathrm{z}(g)$. But $\mathrm{cl}_{\beta X}\mathrm{z}(g)\subseteq\mathrm{z}(g_\beta)$. It now follows from (\ref{DSD}) that $A_n\subseteq\mathrm{z}(g_\beta)$ for each positive integer $n$. That is, $\mathrm{z}(g_\beta)$ contains $A$.

Finally, to prove (2), note that by the proofs of Theorems \ref{UUS} and \ref{UDR} the mapping
\[\psi:C^{\mathscr I}_0(X)\rightarrow C_0(Y)\]
defined by
\[\psi(f)=f_\lambda\]
for any $f\in C^{\mathscr I}_0(X)$ is a bijection with
\[\psi\big(C^{\mathscr I}_{00}(X)\big)=C_{00}(Y);\]
it is now trivial that (2) follows from (3).
\end{proof}

\begin{remark}\label{GGJ}
As we have seen in Theorem \ref{YUS}, if ${\mathscr I}$ is a proper $\sigma$-ideal then $\lambda_{\mathscr I} X$ is a non-compact countably compact space with $C_0(\lambda_{\mathscr I} X)=C_{00}(\lambda_{\mathscr I} X)$. In particular, if $\mathfrak{P}$ is a topological property such that
\[\mbox{$\mathfrak{P}$ $+$ countable compactness $\rightarrow$ compactness}\]
then $\lambda_{\mathscr I} X$ is non-$\mathfrak{P}$ either. The list of such topological properties is quite wide, including topological properties such as the Lindel\"{o}f property, paracompactness, realcompactness, metacompactness, subparacompactness, submetacompactness (or $\theta$-refinability), the meta-Lindel\"{o}f property, the submeta-Lindel\"{o}f property (or $\delta\theta$-refinability), weak submetacompactness (or weak $\theta$-refinability) and the weak submeta-Lindel\"{o}f property (or weak $\delta\theta$-refinability) among others. (See Parts 6.1 and 6.2 of \cite{Va}.)
\end{remark}

\part{Examples}\label{KHGJ}

In this part we study specific examples of either $C^{\mathscr I}_{00}(X)$ or $C^{\mathscr I}_0(X)$. We will see how this specification, of either the space $X$ or the ideal ${\mathscr I}$, enables us to study $C^{\mathscr I}_{00}(X)$ and $C^{\mathscr I}_0(X)$ further and deeper.

\section{Subalgebras of $\ell_\infty$}\label{HFLH}

In this section, which comprises our first set of examples, we consider certain ideals in $({\mathscr P}(\mathbb{N}),\subseteq)$ (with $\mathbb{N}$ endowed with the discrete topology). This leads to the introduction of certain subalgebras of $\ell_\infty$.

By $\ell_\infty$, $c_0$ and $c_{00}$, respectively, we denote the set of all bounded sequences in $\mathbb{R}$, the set of all vanishing sequences in $\mathbb{R}$, and the set of all sequences in $\mathbb{R}$ with only finitely many non-zero terms. Note that $\ell_\infty=C_b(\mathbb{N})$, $c_0=C_0(\mathbb{N})$ and $c_{00}=C_{00}(\mathbb{N})$, if $\mathbb{N}$ is given the discrete topology.

Let
\[{\mathscr S}=\bigg\{A\subseteq\mathbb{N}:\sum_{n\in A}\frac{1}{n}\mbox{ converges}\bigg\}.\]
Then ${\mathscr S}$ is an ideal in $({\mathscr P}(\mathbb{N}),\subseteq)$, called the \textit{summable ideal} in $\mathbb{N}$. A subset of $\mathbb{N}$ is called \textit{small} if it is null (with respect to ${\mathscr S}$).

Note that there exists a family $\{A_i:i<2^\omega\}$ consisting of infinite subsets of $\mathbb{N}$ such that the intersection $A_i\cap A_j$ is finite for any distinct $i,j<2^\omega$. To see this, arrange the rational numbers into a sequence $q_1,q_2,\ldots$ and for each $i\in\mathbb{R}$ define $A_i=\{n_1,n_2,\ldots\}$ where $q_{n_1},q_{n_2},\ldots$ is a subsequence of $q_1,q_2,\ldots$ which converges to $i$. This known fact will be used in the proof of the following.

Recall that for a collection $\{X_i:i\in I\}$ of algebras the direct sum $\bigoplus_{i\in I}X_i$ is the set of all sequences $\{x_i\}_{i\in I}$ where $x_i\in X_i$ for each $i\in I$ such that $x_i=0$ for all but a finite number of indices $i\in I$. The set $\bigoplus_{i\in I}X_i$ is an algebra with addition, multiplication and scalar multiplication defined component-wise. We denote the sequence $\{x_i\}_{i\in I}$ by a sum $\sum_{i\in I}x_i$. The direct sum $\bigoplus_{i\in I}X_i$ of a collection $\{X_i:i\in I\}$ of normed spaces is defined analogously and is a normed space with the norm given by
\[\bigg\|\sum_{i\in I}x_i\bigg\|=\sup\big\{\|x_i\|_{X_i}:i\in I\big\}.\]

The proof of the following makes use of certain standard properties of the Stone--\v{C}ech compactification as stated in the Introduction.

\begin{theorem}\label{TDK}
Let
\[\mathfrak{s}_{00}=\bigg\{\mathbf{x}\in\ell_\infty:\sum_{\mathbf{x}(n)\neq0}\frac{1}{n}\mbox{ converges}\bigg\}.\]
Then
\begin{itemize}
\item[\rm(1)] $\mathfrak{s}_{00}$ is an ideal in $\ell_\infty$.
\item[\rm(2)] $\mathfrak{s}_{00}$ is non-unital.
\item[\rm(3)] $\mathfrak{s}_{00}$ is isometrically isomorphic to $C_{00}(Y)$ for the subspace
\[Y=\bigcup\bigg\{\mathrm{cl}_{\beta\mathbb{N}}A:A\subseteq\mathbb{N}\mbox{ and }\sum_{n\in A}\frac{1}{n}\mbox{ converges}\bigg\},\]
of $\beta\mathbb{N}$.
\item[\rm(4)] $\mathfrak{s}_{00}$ contains a copy of the normed algebra $\bigoplus_{n=1}^\infty\ell_\infty$.
\item[\rm(5)] $\mathfrak{s}_{00}/c_{00}$ contains a copy of the algebra
\[\bigoplus_{i<2^\omega}\frac{\ell_\infty}{c_{00}}.\]
\end{itemize}
\end{theorem}

\begin{proof}
(1)--(3). Consider the ideal ${\mathscr S}$ of $({\mathscr P}(\mathbb{N}),\subseteq)$. Note that if $\mathbf{x}\in\ell_\infty$ (since $\mathbb{N}$ is discrete) then
\[\mathrm{supp}(\mathbf{x})=\big\{n\in\mathbb{N}:\mathbf{x}(n)\neq0\big\},\]
and $\mathrm{supp}(\mathbf{x})$ is null if and only if it has a null neighborhood in $\mathbb{N}$. Thus $\mathfrak{s}_{00}=C^{\mathscr S}_{00}(\mathbb{N})$.
It now follows from Theorem \ref{TTG} that $\mathfrak{s}_{00}$ is an ideal in $\ell_\infty$. The remaining parts follow from Theorem \ref{UUS}. Observe that $\mathbb{N}$ is locally null (indeed, $\{n\}$ is a null neighborhood of $n$ in $\mathbb{N}$ for each $n\in\mathbb{N}$) and that ${\mathscr S}$ is non-proper (as $\sum 1/n$ diverges). Also, note that (since $\mathbb{N}$ is discrete) every subset $A$ of $\mathbb{N}$ is a cozero-set in $\mathbb{N}$ and since $A$ is open-closed in $\mathbb{N}$ its closure $\mathrm{cl}_{\beta\mathbb{N}}A$ in $\beta\mathbb{N}$ is open in $\beta\mathbb{N}$. Therefore $Y=\lambda_{\mathscr S}\mathbb{N}$.

(4). By (3), we may consider $C_{00}(Y)$ in place of $\mathfrak{s}_{00}$. Let $A$ be an infinite subset of $\mathbb{N}$ such that $\sum_{n\in A}1/n$ converges (which exists, for example, let $A=\{2^n:n\in \mathbb{N}\}$). Let $A_1, A_2,\dots$ be a partition of $A$ into pairwise disjoint infinite subsets. Let $n=1,2,\dots$. We may assume that $C(\mathrm{cl}_{\beta\mathbb{N}}A_n)$ is a subalgebra of $C_{00}(Y)$. (Since $A_n$ is open-closed in $\mathbb{N}$ it has open-closed closure $\mathrm{cl}_{\beta\mathbb{N}}A_n$ in $\beta\mathbb{N}$. Thus each element of $C(\mathrm{cl}_{\beta\mathbb{N}}A_n)$ may be continuously extended over $Y$ by defining it to be identically $0$ elsewhere.) Note that $\mathrm{cl}_{\beta\mathbb{N}}A_i$ and $\mathrm{cl}_{\beta\mathbb{N}}A_j$ are disjoint for any distinct $i,j=1,2,\dots$, as $A_i$ and $A_j$ are disjoint open-closed subspaces (and thus zero-sets) of $\mathbb{N}$. Thus, the inclusion mapping
\[\iota:\bigoplus_{n=1}^\infty C(\mathrm{cl}_{\beta\mathbb{N}}A_n)\rightarrow C_{00}(Y)\]
is an algebra isomorphism (onto its image) and it preserves norms. (For the latter, use the fact that $\mathrm{cl}_{\beta\mathbb{N}}A_i$'s are disjoint for distinct indices.) Note that $\mathrm{cl}_{\beta\mathbb{N}}A_n$ coincides with $\beta A_n$, as $A_n$ is closed in the normal space $\mathbb{N}$. Finally, observe that
\[C(\mathrm{cl}_{\beta\mathbb{N}}A_n)=C(\beta A_n)=C(\beta\mathbb{N})=C_b(\mathbb{N})=\ell_\infty.\]

(5). By (3), we may consider $C_{00}(Y)$ in place of $\mathfrak{s}_{00}$. Let $A$ be an infinite subset of $\mathbb{N}$ such that $\sum_{n\in A}1/n$ converges. Consider a family $\{A_i:i<2^\omega\}$ consisting of infinite subsets of $A$ such that $A_i\cap A_j$ is finite for any distinct $i,j<2^\omega$. Let
\[H=\big\{f\in C_{00}(Y):\mathrm{supp}(f)\subseteq A\big\}\]
and let
\[H_i=\big\{f\in C(\mathrm{cl}_{\beta\mathbb{N}}A_i):\mathrm{supp}(f)\subseteq A_i\big\}\]
for each $i<2^\omega$. As in (4), we may assume that $C(\mathrm{cl}_{\beta\mathbb{N}}A_i)$ is a subalgebra of $C_{00}(Y)$ for each $i<2^\omega$, and thus, we may assume that $H_i\subseteq H$. Note that if $f\in H$ then $\mathrm{supp}(f)$ is finite, as it is a compact subspace of $\mathbb{N}$.

Define
\[\Theta:\bigoplus_{i<2^\omega}\frac{C(\mathrm{cl}_{\beta\mathbb{N}}A_i)}{H_i}\rightarrow\frac{C_{00}(Y)}{H}\]
by
\[\sum_{i<2^\omega}(f_i+H_i)\mapsto\sum_{i<2^\omega}f_i+H\]
where $f_i\in C(\mathrm{cl}_{\beta\mathbb{N}}A_i)$ for each $i<2^\omega$. We show that $\Theta$ is an isometric isomorphism onto its image; since
\[\frac{C(\mathrm{cl}_{\beta\mathbb{N}}A_i)}{H_i}=\frac{\ell_\infty}{c_{00}}\]
for each $i<2^\omega$, this completes the proof.

First, note that $\Theta$ is well defined; to show this, let
\[\sum_{i<2^\omega}(f_i+H_i)=\sum_{i<2^\omega}(g_i+H_i)\]
where $f_i, g_i\in C(\mathrm{cl}_{\beta\mathbb{N}}A_i)$ for each $i<2^\omega$. For each $i<2^\omega$ then $f_i+H_i=g_i+H_i$, or equivalently $f_i-g_i\in H_i$; in particular $f_i-g_i\in H$. Thus
\[\sum_{i<2^\omega}(f_i-g_i)\in H\]
and therefore
\[\Theta\bigg(\sum_{i<2^\omega}(f_i+H_i)\bigg)=\sum_{i<2^\omega}f_i+H=\sum_{i<2^\omega}g_i+H=\Theta\bigg(\sum_{i<2^\omega}(g_i+H_i)\bigg).\]

Now, we show that $\Theta$ preserves product. Let $f_i,g_i\in C(\mathrm{cl}_{\beta\mathbb{N}}A_i)$ for each $i<2^\omega$. Then
\[\Theta\bigg(\sum_{i<2^\omega}(f_i+H_i)\cdot\sum_{i<2^\omega}(g_i+H_i)\bigg)=\Theta\bigg(\sum_{i<2^\omega}(f_ig_i+H_i)\bigg)
=\sum_{i<2^\omega}f_ig_i+H.\]
Note that if $k,l<2^\omega$ with $k\neq l$ then
\begin{eqnarray*}
\mathrm{supp}(f_kg_l)&\subseteq&\mathrm{supp}(f_k)\cap\mathrm{supp}(g_l)\\&\subseteq&\mathrm{cl}_{\beta\mathbb{N}}A_k\cap\mathrm{cl}_{\beta\mathbb{N}}A_l=
\mathrm{cl}_{\beta\mathbb{N}}(A_k\cap A_l)=A_k\cap A_l\subseteq A
\end{eqnarray*}
and thus $f_kg_l\in H$. We have
\begin{eqnarray*}
\Theta\bigg(\sum_{i<2^\omega}(f_i+H_i)\bigg)\cdot\Theta\bigg(\sum_{i<2^\omega}(g_i+H_i)\bigg)&=&\bigg(\sum_{i<2^\omega}f_i+H\bigg)\cdot\bigg(\sum_{i<2^\omega}g_i+H\bigg)
\\&=&\bigg(\sum_{i<2^\omega}f_i\sum_{i<2^\omega}g_i\bigg)+H\\&=&\bigg(\sum_{i<2^\omega}f_ig_i+\sum_{k\neq l}f_kg_l\bigg)+H\\&=&\sum_{i<2^\omega}f_ig_i+H.
\end{eqnarray*}
This together with the above proves that
\[\Theta\bigg(\sum_{i<2^\omega}(f_i+H_i)\cdot\sum_{i<2^\omega}(g_i+H_i)\bigg)=\Theta\bigg(\sum_{i<2^\omega}(f_i+H_i)\bigg)\cdot\Theta\bigg(\sum_{i<2^\omega}(g_i+H_i)\bigg).\]
That $\Theta$ preserves addition and scalar multiplication follows analogously.

Next, we show that $\Theta$ is injective. Let
\[\Theta\bigg(\sum_{i<2^\omega}(f_i+H_i)\bigg)=0\]
where $f_i\in C(\mathrm{cl}_{\beta\mathbb{N}}A_i)$ for each $i<2^\omega$. Then
\[\sum_{i<2^\omega}f_i+H=0,\]
or, equivalently
\[h=\sum_{i<2^\omega}f_i\in H.\]
Suppose that $f_{i_j}$, where $j=1,\ldots,n$, are the possibly non-zero terms. Fix some $k=1,\ldots,n$. Then
\[f_{i_k}=h-\sum_{1\leq j\neq k\leq n}f_{i_j}.\]
We have
\[\mathrm{coz}(f_{i_k})\subseteq\mathrm{coz}(h)\cup\bigcup_{1\leq j\neq k\leq n}\mathrm{coz}(f_{i_j})\subseteq\mathrm{supp}(h)\cup\bigcup_{1\leq j\neq k\leq n}\mathrm{cl}_{\beta\mathbb{N}}A_{i_j}.\]
Note that the latter set is compact, as $\mathrm{supp}(h)$ is finite, since $h\in H$. Therefore
\begin{equation}\label{PJUF}
\mathrm{supp}(f_{i_k})\subseteq\mathrm{supp}(h)\cup\bigcup_{1\leq j\neq k\leq n}\mathrm{cl}_{\beta\mathbb{N}}A_{i_j}.
\end{equation}
Intersecting both sides of (\ref{PJUF}) with $\mathrm{cl}_{\beta\mathbb{N}}A_{i_k}$ yields
\begin{eqnarray*}
\mathrm{supp}(f_{i_k})&\subseteq&\mathrm{supp}(h)\cup\bigcup_{1\leq j\neq k\leq n}(\mathrm{cl}_{\beta\mathbb{N}}A_{i_j}\cap\mathrm{cl}_{\beta\mathbb{N}}A_{i_k})\\&=&\mathrm{supp}(h)\cup\bigcup_{1\leq j\neq k\leq n}\mathrm{cl}_{\beta\mathbb{N}}(A_{i_j}\cap A_{i_k})=\mathrm{supp}(h)\cup\bigcup_{1\leq j\neq k\leq n}(A_{i_j}\cap A_{i_k})
\end{eqnarray*}
with the latter being a subset of $\mathbb{N}$. Thus
\[\mathrm{supp}(f_{i_k})\subseteq\mathbb{N}\cap\mathrm{cl}_{\beta\mathbb{N}} A_{i_k}=A_{i_k}\]
and therefore $f_{i_k}\in H_{i_k}$. This implies that
\[\sum_{j=1}^n(f_{i_j}+H_{i_j})=0.\]
Thus $\Theta$ is injective.
\end{proof}

\begin{theorem}\label{JHF}
Let
\[\mathfrak{s}_0=\bigg\{\mathbf{x}\in\ell_\infty:\sum_{|\mathbf{x}(n)|\geq\epsilon}\frac{1}{n}\mbox{ converges for each }\epsilon>0\bigg\}.\]
Then
\begin{itemize}
\item[\rm(1)] $\mathfrak{s}_0$ is a closed ideal in $\ell_\infty$.
\item[\rm(2)] $\mathfrak{s}_0$ is non-unital.
\item[\rm(3)] $\mathfrak{s}_0$ contains $\mathfrak{s}_{00}$ as a dense subspace.
\item[\rm(4)] $\mathfrak{s}_0$ is isometrically isomorphic to $C_0(Y)$  for the subspace
\[Y=\bigcup\bigg\{\mathrm{cl}_{\beta\mathbb{N}}A:A\subseteq\mathbb{N}\mbox{ and }\sum_{n\in A}\frac{1}{n}\mbox{ converges}\bigg\},\]
of $\beta\mathbb{N}$.
\item[\rm(5)] $\mathfrak{s}_0/c_0$ contains a copy of the normed algebra
\[\bigoplus_{i<2^\omega}\frac{\ell_\infty}{c_0}.\]
\end{itemize}
\end{theorem}

\begin{proof}
The proofs for (1)--(4) are analogous to the proofs for the corresponding parts in Theorem \ref{TDK}, making use of Theorems \ref{DGH} and \ref{UDR}. Note that if $\mathbf{x}\in\ell_\infty$ and $\epsilon>0$ then $|\mathbf{x}|^{-1}([\epsilon,\infty))$ is null (with respect to ${\mathscr S}$) if and only if
\[\sum_{|\mathbf{x}(n)|\geq\epsilon}\frac{1}{n}\]
converges. Thus, in particular $\mathfrak{s}_0=C^{\mathscr S}_0(\mathbb{N})$.

(5). The proof of this part is analogous to the proof of the corresponding part in Theorem \ref{TDK}; we will highlight only the differences.

Let $A$ and $\{A_i:i<2^\omega\}$ be as chosen in the proof of Theorem \ref{TDK}. Let
\[H=\big\{f\in C_0(Y):|f|^{-1}\big([\epsilon,\infty)\big)\subseteq A\mbox{ for each }\epsilon>0\big\}\]
and let
\[H_i=\big\{f\in C(\mathrm{cl}_{\beta\mathbb{N}}A_i):|f|^{-1}\big([\epsilon,\infty)\big)\subseteq A_i\mbox{ for each }\epsilon>0\big\}\]
for each $i<2^\omega$. We consider $C(\mathrm{cl}_{\beta\mathbb{N}}A_i)$ as a subalgebra of $C_0(Y)$, and thus, we may assume that $H_i\subseteq H$ for each $i<2^\omega$. Note that if $i<2^\omega$ and $f\in H_i$ then $|f|^{-1}([\epsilon,\infty))$ is a finite subset of $A_i$, as it is compact (since it is  closed in $\mathrm{cl}_{\beta\mathbb{N}}A_i$). Define
\[\Theta:\bigoplus_{i<2^\omega}\frac{C(\mathrm{cl}_{\beta\mathbb{N}}A_i)}{H_i}\rightarrow\frac{C_0(Y)}{H}\]
by
\[\sum_{i<2^\omega}(f_i+H_i)\mapsto\sum_{i<2^\omega}f_i+H\]
where $f_i\in C(\mathrm{cl}_{\beta\mathbb{N}}A_i)$ for each $i<2^\omega$. Arguing as in the proof of Theorem \ref{TDK} it follows that $\Theta$ is well defined. The proofs that $\Theta$ preserves product is analogous to the corresponding part in the proof of Theorem \ref{TDK}; simply observe that if $f_i,g_i\in C(\mathrm{cl}_{\beta\mathbb{N}}A_i)$ for each $i<2^\omega$ and $\epsilon>0$, then for any $k,l<2^\omega$ with $k\neq l$ we have
\begin{eqnarray*}
|f_kg_l|^{-1}\big([\epsilon,\infty)\big)&\subseteq&\mathrm{coz}(f_kg_l)\\&\subseteq&\mathrm{coz}(f_k)\cap\mathrm{coz}(g_l)
\\&\subseteq&\mathrm{cl}_{\beta\mathbb{N}}A_k\cap\mathrm{cl}_{\beta\mathbb{N}}A_l=\mathrm{cl}_{\beta\mathbb{N}}(A_k\cap A_l)=A_k\cap A_l\subseteq A
\end{eqnarray*}
and thus $f_kg_l\in H$. The proofs that $\Theta$ preserves addition and scalar multiplication are analogous.

Now, we show that $\Theta$ is injective. Let
\[\Theta\bigg(\sum_{i<2^\omega}(f_i+H_i)\bigg)=0\]
where $f_i\in C(\mathrm{cl}_{\beta\mathbb{N}}A_i)$ for each $i<2^\omega$. Suppose that $f_{i_j}$, where $j=1,\ldots,n$, are the possibly non-zero terms. Fix some $k=1,\ldots,n$. Then, as in the proof of Theorem \ref{TDK} we have
\[f_{i_k}=h-\sum_{1\leq j\neq k\leq n}f_{i_j}\]
for some $h\in H$. Let $\epsilon>0$. Then
\begin{eqnarray*}
|f_{i_k}|^{-1}\big([\epsilon,\infty)\big)&\subseteq&|h|^{-1}\big([\epsilon/n,\infty)\big)\cup\bigcup_{1\leq j\neq k\leq n}|f_{i_j}|^{-1}\big([\epsilon/n,\infty)\big)\\&\subseteq& A\cup\bigcup_{1\leq j\neq k\leq n}\mathrm{cl}_{\beta\mathbb{N}}A_{i_j}.
\end{eqnarray*}
Arguing as in the proof of Theorem \ref{TDK}, intersecting both sides of the above with $\mathrm{cl}_{\beta\mathbb{N}}A_{i_k}$ yields
\begin{eqnarray*}
|f_{i_k}|^{-1}\big([\epsilon,\infty)\big)\subseteq A_{i_k}
\end{eqnarray*}
and therefore $f_{i_k}\in H_{i_k}$. This implies that
\[\sum_{j=1}^n(f_{i_j}+H_{i_j})=0.\]
Thus $\Theta$ is injective.

Next, we show that $\Theta$ is an isometry. First, we need to show the following.

\begin{xclaim}
Let $i<2^\omega$ and $f\in C(\mathrm{cl}_{\beta\mathbb{N}}A_i)$. Then
\[\|f+H_i\|=\big\|f|_{Y\setminus\mathbb{N}}\big\|_\infty.\]
\end{xclaim}

\noindent \emph{Proof of the claim.} Suppose to the contrary that
\begin{equation}\label{FSD}
\|f+H_i\|<\big\|f|_{Y\setminus\mathbb{N}}\big\|_\infty.
\end{equation}
Then
\[\alpha=\|f+h\|_\infty<\big\|f|_{Y\setminus\mathbb{N}}\big\|_\infty\]
for some $h\in H_i$. Let $\alpha<\gamma<\|f|_{Y\setminus\mathbb{N}}\|_\infty$. Then
\[B=|f|^{-1}\big([\gamma,\infty)\big)\subseteq|h|^{-1}\big([\gamma-\alpha,\infty)\big)=C;\]
as if $y\in Y$ such that $|f(y)|\geq\gamma$, then
\[\alpha\geq\big|f(y)+h(y)\big|\geq\big|f(y)\big|-\big|h(y)\big|\geq\gamma-\big|h(y)\big|\]
and thus $|h(y)|\geq\gamma-\alpha$. Note that $B$ is a finite subset of $\mathbb{N}$, as $C$ is so, since $h\in H_i$. We have
\[Y=\mathrm{cl}_Y \mathbb{N}=B\cup\mathrm{cl}_Y (\mathbb{N}\setminus B)\]
and thus
\begin{eqnarray*}
|f|(Y\setminus\mathbb{N})\subseteq|f|\big(\mathrm{cl}_Y (\mathbb{N}\setminus B)\big)\subseteq\overline{|f|(\mathbb{N}\setminus B)}&=&\overline{|f|\big(\mathbb{N}\cap|f|^{-1}\big([0,\gamma)\big)\big)}\\&\subseteq&\overline{|f|\big(|f|^{-1}\big([0,\gamma)\big)\big)}
\subseteq\overline{[0,\gamma)}=[0,\gamma],
\end{eqnarray*}
where the bar denotes the closure in $\mathbb{R}$. This implies that $\|f|_{Y\setminus\mathbb{N}}\|_\infty\leq\gamma$, which contradicts the choice of $\gamma$. Thus (\ref{FSD}) is false, that is
\begin{equation}\label{GD}
\|f+H_i\|\geq\big\|f|_{Y\setminus\mathbb{N}}\big\|_\infty.
\end{equation}

Now, we prove the reverse inequality in (\ref{GD}). Let $\delta=\|f|_{Y\setminus\mathbb{N}}\|_\infty$. We first show that
\[D_\epsilon=\mathbb{N}\cap|f|^{-1}\big((\delta+\epsilon,\infty)\big)\]
is finite for every $\epsilon>0$. Suppose the contrary, that is, suppose that $D_\epsilon$ is infinite for some $\epsilon>0$. Note that $\mathrm{cl}_{\beta\mathbb{N}}D_\epsilon\setminus\mathbb{N}$ is non-empty, as $\mathrm{cl}_{\beta\mathbb{N}}D_\epsilon\subseteq\mathbb{N}$ implies that $D_\epsilon=\mathrm{cl}_{\beta\mathbb{N}}D_\epsilon$ is compact and is then finite. Let $p\in\mathrm{cl}_{\beta\mathbb{N}}D_\epsilon\setminus\mathbb{N}$. Note that
\[\mathrm{cl}_{\beta\mathbb{N}}D_\epsilon=\mathrm{cl}_{\beta\mathbb{N}}\big(|f|^{-1}\big((\delta+\epsilon,\infty)\big)\big).\]
We have
\begin{eqnarray*}
|f|(p)\in|f|(\mathrm{cl}_{\beta\mathbb{N}}D_\epsilon)&=&|f|\big(\mathrm{cl}_{\beta\mathbb{N}}\big(|f|^{-1}\big((\delta+\epsilon,\infty)\big)\big)\big)
\\&\subseteq&\overline{|f|\big(|f|^{-1}\big((\delta+\epsilon,\infty)\big)\big)}\subseteq\overline{(\delta+\epsilon,\infty)}=[\delta+\epsilon,\infty),
\end{eqnarray*}
where the bar denotes the closure in $\mathbb{R}$. Therefore
\[\delta=\big\|f|_{Y\setminus\mathbb{N}}\big\|_\infty\geq |f(p)|\geq\delta+\epsilon,\]
which is not possible. This shows that $D_\epsilon$ is finite for every $\epsilon>0$. Now, let $\epsilon>0$. Define $h:Y\rightarrow\mathbb{R}$ such that $h(x)=-f(x)$ if $x\in D_\epsilon$ and $h(x)=0$ otherwise. Note that $D_\epsilon$ is closed in $Y$, as it is finite, and $D_\epsilon$ is open in $Y$, as it is open in $\mathbb{N}$ and $\mathbb{N}$ is open in $\beta\mathbb{N}$ (and thus in $Y$), since $\mathbb{N}$ is locally compact. Therefore $h$ is continuous. Observe that
\[\mathrm{coz}(h)\subseteq D_\epsilon\subseteq\mathbb{N}\cap|f|^{-1}\big([\delta+\epsilon,\infty)\big)\subseteq\mathbb{N}\cap\mathrm{cl}_{\beta\mathbb{N}}A_i=A_i.\]
Thus $h\in H_i$. Note that $f+h=\mathbf{0}$ on $D_\epsilon$ and $f+h= f$ on $Y\setminus D_\epsilon$. Also,
\[\big\|f|_{\mathbb{N}\setminus D_\epsilon}\big\|_\infty\leq\delta+\epsilon\]
by the way we have defined $D_\epsilon$. Now
\[\|f+H_i\|\leq\|f+h\|_\infty=\big\|f|_{Y\setminus D_\epsilon}\big\|_\infty=\max\big\{\big\|f|_{\mathbb{N}\setminus D_\epsilon}\big\|_\infty, \big\|f|_{Y\setminus \mathbb{N}}\big\|_\infty\big\}\leq\delta+\epsilon.\]
Since $\epsilon>0$ is arbitrary, it follows that
\[\|f+H_i\|\leq\delta=\big\|f|_{Y\setminus\mathbb{N}}\big\|_\infty.\]
This together with (\ref{GD}) proves the claim.

\begin{xclaim}
Let $i_1,\ldots,i_n<2^\omega$ and $f_{i_j}\in C(\mathrm{cl}_{\beta\mathbb{N}}A_{i_j})$ for each $j=1,\ldots,n$. Then
\[\bigg\|\sum_{j=1}^nf_{i_j}+H\bigg\|=\bigg\|\bigg(\sum_{j=1}^nf_{i_j}\bigg)\bigg|_{Y\setminus\mathbb{N}}\bigg\|_\infty.\]
\end{xclaim}

\noindent \emph{Proof of the claim.}
This follows by an argument similar to the one we have given in the first claim; one simply needs to replace $f$ by $\sum_{j=1}^nf_{i_j}$, $H_i$ by $H$, and $A_i$ by $A$ throughout.

\medskip

Now, let $i_1,\ldots,i_n<2^\omega$ and $f_{i_j}\in C(\mathrm{cl}_{\beta\mathbb{N}}A_{i_j})$ for each $j=1,\ldots,n$. For simplicity of the notation let
\[f=\sum_{j=1}^nf_{i_j}.\]
Let $1\leq j\neq k\leq n$. Then $A_j\cap A_k$ is finite and thus
\[\mathrm{coz}(f_{i_k})\cap\mathrm{coz}(f_{i_l})\subseteq\mathrm{cl}_{\beta\mathbb{N}}A_{i_k}\cap\mathrm{cl}_{\beta\mathbb{N}}A_{i_l}=\mathrm{cl}_{\beta\mathbb{N}}(A_{i_k}\cap A_{i_l})=A_{i_k}\cap A_{i_l}\subseteq\mathbb{N}.\]
In particular,
\[\mathrm{coz}\big(f_{i_k}|_{Y\setminus\mathbb{N}}\big)\cap\mathrm{coz}\big(f_{i_l}|_{Y\setminus\mathbb{N}}\big)=\emptyset.\]
Now, by the second claim
\[\bigg\|\Theta\bigg(\sum_{j=1}^n(f_{i_j}+H_{i_j})\bigg)\bigg\|=\|f+H\|=\big\|f|_{Y\setminus\mathbb{N}}\big\|_\infty\]
and by the first claim
\begin{eqnarray*}
\big\|f|_{Y\setminus\mathbb{N}}\big\|_\infty&=&\max\big\{\big\|f|_{(\mathrm{cl}_{\beta\mathbb{N}}A_{i_j}\setminus\mathbb{N})}\big\|_\infty:j=1,\ldots,n\big\}
\\&=&\max\big\{\big\|f_{i_j}|_{Y\setminus\mathbb{N}}\big\|_\infty:j=1,\ldots,n\big\}
\\&=&\max\big\{\|f_{i_j}+H_{i_j}\|:j=1,\ldots,n\big\}=\bigg\|\sum_{j=1}^n(f_{i_j}+H_{i_j})\bigg\|.
\end{eqnarray*}
That is, $\Theta$ is an isometry.
\end{proof}

\begin{remark}
Any sequence $f:\mathbb{N}\rightarrow(0,\infty)$ such that $\sum_{n=1}^\infty f(n)$ diverges, determines an ideal
\[{\mathscr S}_f=\bigg\{A\subseteq\mathbb{N}:\sum_{n\in A}f(n)\mbox{ converges}\bigg\}\]
in $({\mathscr P}(\mathbb{N}),\subseteq)$. This provides a more general setting to state and prove Theorems \ref{TDK} and \ref{JHF}.
\end{remark}

Let
\[{\mathscr D}=\bigg\{A\subseteq\mathbb{N}:\limsup_{n\rightarrow\infty}\frac{|A\cap\{1,\ldots,n\}|}{n}=0\bigg\}.\]
Then ${\mathscr D}$ also is an ideal in $({\mathscr P}(\mathbb{N}),\subseteq)$, called the \textit{density ideal} in $\mathbb{N}$. In other words, ${\mathscr D}$ consists of those subsets $D$ of $\mathbb{N}$ such that $D$ has \textit{asymptotic density} zero. Note that every small set in $\mathbb{N}$ has asymptotic density zero, that is ${\mathscr S}\subseteq{\mathscr D}$; the converse, however, does not hold in general. The set of all prime numbers has asymptotic density zero but it is not small. (For more information on the subject, see \cite{F}.)

The following is dual to Theorems \ref{TDK} and \ref{JHF} and may be proved analogously, replacing the ideal ${\mathscr S}$ by the ideal ${\mathscr D}$ throughout the proofs already given.

\begin{theorem}\label{DFDS}
Let
\[\mathfrak{d}_{00}=\bigg\{\mathbf{x}\in\ell_\infty:\limsup_{n\rightarrow\infty}\frac{|\{k\leq n:\mathbf{x}(k)\neq0\}|}{n}=0\bigg\}\]
and
\[\mathfrak{d}_0=\bigg\{\mathbf{x}\in\ell_\infty:\limsup_{n\rightarrow\infty}\frac{|\{k\leq n:|\mathbf{x}(k)|\geq\epsilon\}|}{n}=0\mbox{ for each }\epsilon>0\bigg\}.\]
Then
\begin{itemize}
\item[\rm(1)] $\mathfrak{d}_{00}$ is an ideal in $\ell_\infty$ and $\mathfrak{d}_0$ is a closed ideal in $\ell_\infty$.
\item[\rm(2)] $\mathfrak{d}_{00}$ is dense in $\mathfrak{d}_0$.
\item[\rm(3)] Neither $\mathfrak{d}_{00}$ nor $\mathfrak{d}_0$ is unital.
\item[\rm(4)] $\mathfrak{d}_{00}$ and $\mathfrak{d}_0$ are isometrically isomorphic to $C_{00}(Y)$ and $C_0(Y)$, respectively, for the subspace
\[Y=\bigcup\bigg\{\mathrm{cl}_{\beta\mathbb{N}}A:A\subseteq\mathbb{N}\mbox{ and }\limsup_{n\rightarrow\infty}\frac{|A\cap\{1,\ldots,n\}|}{n}=0\bigg\}\]
of $\beta\mathbb{N}$.
\item[\rm(5)] $\mathfrak{d}_{00}$ contains a copy of the normed algebra $\bigoplus_{n=1}^\infty\ell_\infty$.
\item[\rm(6)] $\mathfrak{d}_{00}/c_{00}$ contains a copy of the algebra
\[\bigoplus_{i<2^\omega}\frac{\ell_\infty}{c_{00}}.\]
\item[\rm(7)] $\mathfrak{d}_0/c_0$ contains a copy of the normed algebra
\[\bigoplus_{i<2^\omega}\frac{\ell_\infty}{c_0}.\]
\end{itemize}
\end{theorem}

\section{Banach subalgebras of $C_b(X)$}\label{KHDS}

In this section we consider certain subalgebras of $C_b(X)$ and we obtain several representation theorems. Recall that $(\mathrm{Coz}(X),\subseteq)$ is a lattice which is closed under formation of countable unions. (See the Introduction.)

\begin{lemma}\label{KHFD}
Let $X$ be a completely regular space and let $H$ be a Banach subalgebra of $C_b(X)$ such that
\begin{itemize}
\item[\rm(1)] For every $x\in X$ there exists some $h\in H$ with $h(x)\neq 0$.
\item[\rm(2)] For every $f\in C_b(X)$ we have $f\in H$, if $\mathrm{coz}(f)\subseteq\mathrm{coz}(h)$ for some $h\in H$.
\end{itemize}
Let
\[{\mathscr H}=\big\{\mathrm{coz}(h):h\in H\big\}.\]
Then $\mathscr{H}$ is a $\sigma$-ideal in $(\mathrm{Coz}(X),\subseteq)$ and
\[C^{\mathscr H}_0(X)=H.\]
\end{lemma}

\begin{proof}
For each positive integer $n$ let $h_n\in H$ and (without any loss of generality) assume that $h_n\neq\mathbf{0}$. Then
\[\bigcup_{n=1}^\infty\mathrm{coz}(h_n)=\mathrm{coz}(h)\]
with
\[h=\sum_{n=1}^\infty\frac{h_n^2}{2^n\|h_n^2\|}.\]
That $h$ is well defined and continuous follows from the Weierstrass $M$-test. Observe that $h\in H$, as $h$ is the limit of a sequence in $H$ and $H$ is closed in $C_b(X)$ (since it is a Banach subalgebra of $C_b(X)$). That is, ${\mathscr H}$ is closed under formation of countable unions. This together with (2) shows that $\mathscr{H}$ is a $\sigma$-ideal in $(\mathrm{Coz}(X),\subseteq)$.

Next, we check that $H=C^{\mathscr H}_0(X)$. It is trivial that $H\subseteq C^{\mathscr H}_0(X)$. Let $f\in C^{\mathscr H}_0(X)$. For each positive integer $n$ let $g_n\in H$ such that
\[|f|^{-1}\big([1/n,\infty)\big)\subseteq\mathrm{coz}(g_n).\]
Without any loss of generality we may assume that $g_n\neq\mathbf{0}$ for each positive integer $n$. (Note that $H$ contains non-zero elements by (1), unless $X$ is empty, in which case the whole discussion will be empty as well!) We have
\begin{equation}\label{JKJTF}
\mathrm{coz}(f)=\bigcup_{n=1}^\infty|f|^{-1}\big([1/n,\infty)\big)\subseteq\bigcup_{n=1}^\infty\mathrm{coz}(g_n)=\mathrm{coz}(g)
\end{equation}
where
\[g=\sum_{n=1}^\infty\frac{g_n^2}{2^n\|g_n^2\|}.\]
Observe that $g\in H$ by the above argument. By (2), it follows from (\ref{JKJTF}) that $f\in H$. This shows that $C^{\mathscr H}_0(X)\subseteq H$ and completes the proof.
\end{proof}

\begin{theorem}\label{HGL}
Let $X$ be a completely regular space and let $H$ be a Banach subalgebra of $C_b(X)$ such that
\begin{itemize}
\item[\rm(1)] For every $x\in X$ there exists some $h\in H$ with $h(x)\neq 0$.
\item[\rm(2)] For every $f\in C_b(X)$ we have $f\in H$, if $\mathrm{coz}(f)\subseteq\mathrm{coz}(h)$ for some $h\in H$.
\end{itemize}
Then $H$ is isometrically isomorphic to $C_0(Y)$ with the locally compact Hausdorff space $Y$ given by
\[Y=\bigcup\big\{\mathrm{int}_{\beta X}\mathrm{cl}_{\beta X}C:C\in\mathrm{Coz}(X)\mbox{ and }\mathrm{cl}_XC\subseteq\mathrm{coz}(h)\mbox{ for some }h\in H\big\},\]
considered as a subspace of $\beta X$. Furthermore,
\begin{itemize}
\item[\rm(a)] $X$ is dense in $Y$.
\item[\rm(b)] $H$ is unital if and only if $H$ contains a non-vanishing element if and only if $Y$ is compact.
\end{itemize}
\end{theorem}

\begin{proof}
Let
\[{\mathscr H}=\big\{\mathrm{coz}(h):h\in H\big\}.\]
Then $\mathscr{H}$ is an ideal in $(\mathrm{Coz}(X),\subseteq)$ and
\[C^{\mathscr H}_0(X)=H\]
by Lemma \ref{KHFD}. Note that $X$ is locally null with respect to $\mathscr{H}$ by (1). By Theorem \ref{UF} it follows that $C^{\mathscr H}_0(X)$ is isometrically isomorphic to $C_0(Y)$ with $Y$ as given in the statement of the theorem. Also, (a)--(b) follow from Theorem \ref{UF}.
\end{proof}

The following is the main result of \cite{Ko11}; we deduce it here as a corollary to Theorem \ref{HGL}.

\begin{corollary}\label{KKK}
Let $X$ be a completely regular space and let $H$ be a Banach subalgebra of $C_b(X)$ such that
\begin{itemize}
\item[\rm(1)] For every $x\in X$ there exists some $h\in H$ with $h(x)\neq 0$.
\item[\rm(2)] For every $f\in C_b(X)$ we have $f\in H$ if $\mathrm{supp}(f)\subseteq\mathrm{supp}(h)$ for some $h\in H$.
\end{itemize}
Then $H$ is isometrically isomorphic to $C_0(Y)$ with the locally compact Hausdorff space $Y$ given by
\[Y=\bigcup\big\{\mathrm{int}_{\beta X}\mathrm{cl}_{\beta X}\mathrm{coz}(h):h\in H\big\},\]
considered as a subspace of $\beta X$.
\end{corollary}

\begin{proof}
Trivially, for any $f\in C_b(X)$ and $h\in H$ we have $\mathrm{supp}(f)\subseteq\mathrm{supp}(h)$ provided that $\mathrm{coz}(f)\subseteq\mathrm{coz}(h)$. That is,  condition (2) here implies condition (2) in Theorem \ref{HGL}. The result now follows from Theorem \ref{HGL} with
\[Y=\bigcup\big\{\mathrm{int}_{\beta X}\mathrm{cl}_{\beta X}C:C\in\mathrm{Coz}(X)\mbox{ and }\mathrm{cl}_XC\subseteq\mathrm{coz}(h)\mbox{ for some }h\in H\big\}.\]
Denote
\[Y'=\bigcup\big\{\mathrm{int}_{\beta X}\mathrm{cl}_{\beta X}\mathrm{coz}(h):h\in H\big\}.\]
To complete the proof we need to show that $Y=Y'$. By the definitions, it is clear that $Y\subseteq Y'$. Let $t\in Y'$. Then $t\in\mathrm{int}_{\beta X}\mathrm{cl}_{\beta X}\mathrm{coz}(h)$ for some $h\in H$. By normality of $\beta X$ there exists open subspaces $U$ and $V$ of $\beta X$ such that
\[t\in U\subseteq\mathrm{cl}_{\beta X}U\subseteq V\subseteq\mathrm{cl}_{\beta X}V\subseteq\mathrm{int}_{\beta X}\mathrm{cl}_{\beta X}\mathrm{coz}(h).\]
By the Urysohn Lemma there exist continuous $f,g:\beta X\rightarrow[0,1]$ such that
\[f(t)=1\;\;\;\;\mbox{ and }\;\;\;\;f|_{\beta X\setminus U}=\mathbf{0}\]
and
\[g|_{\mathrm{cl}_{\beta X}U}=\mathbf{1}\;\;\;\;\mbox{ and }\;\;\;\;g|_{\beta X\setminus V}=\mathbf{0}.\]
We show that
\[g|_X\in H,\;\;\;\;\mbox{ }\;\;\;\;\mathrm{cl}_X\mathrm{coz}(f|_X)\subseteq\mathrm{coz}(g|_X)\;\;\;\;\mbox{ and }\;\;\;\;t\in\mathrm{int}_{\beta X}\mathrm{cl}_{\beta X}\mathrm{coz}(f|_X).\]
This will prove that $t\in Y$. Note that
\[\mathrm{coz}(g|_X)=X\cap\mathrm{coz}(g),\]
and
\[\mathrm{cl}_{\beta X}\big(X\cap\mathrm{coz}(g)\big)=\mathrm{cl}_{\beta X}\mathrm{coz}(g),\]
as $\mathrm{coz}(g)$ is open in $\beta X$ and $X$ is dense in $\beta X$. Since $\mathrm{coz}(g)\subseteq V$ by the definition of $g$, and $\mathrm{cl}_{\beta X}V\subseteq\mathrm{cl}_{\beta X}\mathrm{coz}(h)$ by the choice of $V$, we have
\begin{equation}\label{JY}
\mathrm{cl}_{\beta X}\mathrm{coz}(g|_X)\subseteq\mathrm{cl}_{\beta X}\mathrm{coz}(h).
\end{equation}
Thus $\mathrm{supp}(g|_X)\subseteq\mathrm{supp}(h)$, if we intersect the two sides of (\ref{JY}) with $X$. Therefore $g|_X\in H$ by (2). By the same argument, since $\mathrm{coz}(f)\subseteq U$ and $\mathrm{cl}_{\beta X}U\subseteq\mathrm{coz}(g)$ by the definitions of $f$ and $g$, respectively, we have $\mathrm{cl}_X\mathrm{coz}(f|_X)\subseteq\mathrm{coz}(g|_X)$. Finally, note that
\[t\in\mathrm{int}_{\beta X}\mathrm{cl}_{\beta X}\mathrm{coz}(f|_X),\]
as $t\in\mathrm{int}_{\beta X}\mathrm{cl}_{\beta X}\mathrm{coz}(f)$ (since $t\in\mathrm{coz}(f)$) and $\mathrm{cl}_{\beta X}\mathrm{coz}(f)=\mathrm{cl}_{\beta X}\mathrm{coz}(f|_X)$ by the above argument.
\end{proof}

\begin{remark}
As it is observed in \cite{Ko11}, condition (2) in Corollary \ref{KKK} may be abstractified into the following condition.
\begin{itemize}
\item[\rm(2)$'$] For every $f\in C_b(X)$ we have $f\in H$ if $\mathrm{ann}(h)\subseteq\mathrm{ann}(f)$ for some $h\in H$.
\end{itemize}
Here
\[\mathrm{ann}(f)=\big\{g\in C_b(X): fg=\mathbf{0}\big\}\]
for every $f\in C_b(X)$.

To show that (2) (in Corollary \ref{KKK}) implies (2)$'$, let $f\in C_b(X)$ with $\mathrm{ann}(h)\subseteq\mathrm{ann}(f)$ for some $h\in H$. Suppose to the contrary that $\mathrm{supp}(f)\nsubseteq\mathrm{supp}(h)$. Then $x\notin\mathrm{supp}(h)$ for some $x\in\mathrm{supp}(f)$. There exists a continuous $k:X\rightarrow[0,1]$ with \[k(x)=1\;\;\;\;\mbox{ and }\;\;\;\;k|_{\mathrm{supp}(h)}=\mathbf{0}.\]
Then $k^{-1}((1/2,1])$ is an open neighborhood of $x$ in $X$. Thus
\[A=k^{-1}\big((1/2,1]\big)\cap\mathrm{coz}(f)\neq\emptyset.\]
Let $t\in A$. Then $k(t)f(t)\neq0$. But $k\in\mathrm{ann}(h)$, as $kh=\mathbf{0}$ by the definition of $k$. Thus $k\in\mathrm{ann}(f)$, which is a contradiction. Therefore $\mathrm{supp}(f)\subseteq\mathrm{supp}(h)$. It now follows from (2) that $f\in H$.

Next, we check that (2)$'$ implies (2). Let $f\in C_b(X)$ with $\mathrm{supp}(f)\subseteq\mathrm{supp}(h)$ for some $h\in H$. Suppose to the contrary that $\mathrm{ann}(h)\nsubseteq\mathrm{ann}(f)$. Then $g\notin\mathrm{ann}(f)$ for some $g\in\mathrm{ann}(h)$. Thus $gf\neq\mathbf{0}$. Let $t\in X$ with $g(t)f(t)\neq0$. Let  $r=|g(t)|>0$. Note that $t\in\mathrm{supp}(f)$, as $f(t)\neq0$, and therefore $t\in\mathrm{supp}(h)$. Now $|g|^{-1}((r/2,\infty))$ is an open neighborhood of $t$ in $X$. Thus
\[B=|g|^{-1}\big((r/2,\infty)\big)\cap\mathrm{coz}(h)\neq\emptyset.\]
Let $x\in B$. Then $g(x)h(x)\neq 0$, which is a contradiction, as $gh=\mathbf{0}$ by the way $g$ is chosen. Therefore $\mathrm{ann}(h)\subseteq\mathrm{ann}(f)$. By (2)$'$, it follows that $f\in H$.
\end{remark}

Following \cite{GJ}, we call an ideal $H$ of $C_b(X)$ a \textit{$z$-ideal} if $\mathrm{z}(f)=\mathrm{z}(h)$, where $f\in C_b(X)$ and $h\in H$, implies that $f\in H$.

\begin{theorem}\label{GLS}
Let $X$ be a regular Lindel\"{o}f space. Let $H$ be a closed $z$-ideal in $C_b(X)$ with empty hull. Then $H$ is isometrically isomorphic to $C_0(Y)$ for the locally compact Hausdorff space $Y$ given by
\[Y=\bigcup\big\{\mathrm{int}_{\beta X}\mathrm{cl}_{\beta X}\mathrm{coz}(h):h\in H\big\},\]
considered as a subspace of $\beta X$. Furthermore,
\begin{itemize}
  \item $X$ is dense in $Y$.
  \item $C_0(Y)=C_{00}(Y)$; in particular, $Y$ is countably compact.
  \item $Y$ is compact if and only $H$ is unital.
\end{itemize}
\end{theorem}

\begin{proof}
We verify that $H$ satisfies the assumption of Lemma \ref{KHFD}. First, note that $X$ is completely regular (indeed normal), as it is regular and Lindel\"{o}f, and $H$ is a Banach subalgebra of $C_b(X)$, as it is a closed ideal in $C_b(X)$. Condition (1) in Lemma \ref{KHFD} holds, as $H$ is of empty hull. Let $\mathrm{coz}(f)\subseteq\mathrm{coz}(h)$ where $f\in C_b(X)$ and $h\in H$. Then $\mathrm{coz}(f)=\mathrm{coz}(fh)$, or equivalently, $\mathrm{z}(f)=\mathrm{z}(fh)$. But $fh\in H$, as $H$ is an ideal, and therefore $f\in H$, as $H$ is a $z$-ideal. This shows that condition (2) in Lemma \ref{KHFD} holds. By Lemma \ref{KHFD} the set
\[{\mathscr H}=\big\{\mathrm{coz}(h):h\in H\big\}\]
is a $\sigma$-ideal in $(\mathrm{Coz}(X),\subseteq)$ and
\[C^{\mathscr H}_0(X)=H.\]
Note that $X$ is locally null (with respect to $\mathscr{H}$), as $H$ is of empty hull. Theorem \ref{YUS} now concludes the proof.
\end{proof}

\section{Continuous mappings with measure-zero cozero-set}\label{HGF}

This section deals with continuous mappings with measure-zero cozero-set. The natural setting to state and prove our results is the one of topological measure spaces.

\begin{notation}\label{GGH}
Let $(X,{\mathscr B},\mu)$ be a measure space. Denote
\[{\mathscr M}=\big\{B\in{\mathscr B}:\mu(B)=0\big\}.\]
\end{notation}

Note that if $(X,{\mathscr B},\mu)$ is a measure space then $({\mathscr B},\subseteq)$ is an upper semi-lattice (indeed, $\bigvee{\mathscr C}=\bigcup{\mathscr C}$ for any countable subset ${\mathscr C}$ of ${\mathscr B}$). The following holds trivially.

\begin{lemma}\label{KGGF}
Let $(X,{\mathscr B},\mu)$ be a measure space. Then ${\mathscr M}$ is a $\sigma$-ideal in $({\mathscr B},\subseteq)$.
\end{lemma}

A \textit{topological measure space} is a quadruple $(X,{\mathscr O},{\mathscr B},\mu)$ where $(X,{\mathscr B},\mu)$ is a measure space and $(X,{\mathscr O})$ is a topological space such that ${\mathscr O}\subseteq {\mathscr B}$, that is, every open set (and thus every Borel set) is measurable. Let $(X,{\mathscr O},{\mathscr B},\mu)$ be a topological measure space. By $X$ (if used without any specification) we mean the topological space $(X,{\mathscr O})$.

\begin{definition}\label{KHFD}
Let $(X,{\mathscr O},{\mathscr B},\mu)$ be a topological measure space. The measure $\mu$ is said to be \textit{locally null} if every $x\in X$ has a $\mu$-null neighborhood in $X$.
\end{definition}

The following follows trivially from the definitions.

\begin{lemma}\label{KFFG}
Let $(X,{\mathscr O},{\mathscr B},\mu)$ be a topological measure space. Then $X$ is locally null (with respect to the ideal $\mathscr{M}$) if and only if $\mu$ is locally null.
\end{lemma}

The following theorem is the main results of this section.

\begin{theorem}\label{KOIJF}
Let $(X,{\mathscr O},{\mathscr B},\mu)$ be a topological measure space. Let
\[\mathfrak{M}_{00}(X)=\big\{f\in C_b(X):\mathrm{supp}(f)\mbox{ has a $\mu$-null neighborhood in $X$}\big\},\]
and
\[\mathfrak{M}_0(X)=\big\{f\in C_b(X):\mu\big(\mathrm{coz}(f)\big)=0\big\}.\]
Then
\begin{itemize}
\item[\rm(1)] $\mathfrak{M}_{00}(X)$ is an ideal in $C_b(X)$.
\item[\rm(2)] $\mathfrak{M}_0(X)$ is a closed ideal in $C_b(X)$ containing $\mathfrak{M}_{00}(X)$.
\item[\rm(3)] Let $X$ be completely regular. Then $\mathfrak{M}_{00}(X)$ is of empty hull if and only if $\mathfrak{M}_0(X)$ is of empty hull if and only if $\mu$ is locally null.
\item[\rm(4)] Let $X$ be completely regular and let $\mu$ be locally null. Then $\mathfrak{M}_{00}(X)$ is unital if and only if $\mathfrak{M}_0(X)$ is unital if and only if $\mu$ is trivial.
\item[\rm(5)] Let $X$ be normal and let $\mu$ be locally null. Then $\mathfrak{M}_{00}(X)$ and $\mathfrak{M}_0(X)$ are isometrically isomorphic to $C_{00}(Y)$ and $C_0(Y)$ for the unique (up to homeomorphism) locally compact Hausdorff space $Y$ given by
    \[Y=\bigcup\big\{\mathrm{int}_{\beta X}\mathrm{cl}_{\beta X}C:C\in\mathrm{Coz}(X)\mbox{ and }\mathrm{cl}_XC\mbox{ has a $\mu$-null neighborhood in $X$}\big\},\]
    considered as a subspace of $\beta X$. Furthermore,
    \begin{itemize}
    \item $X$ is dense in $Y$.
    \item $\mathfrak{M}_{00}(X)$ is dense in $\mathfrak{M}_0(X)$.
    \item $\mathfrak{M}_{00}(X)$ is unital if and only if $\mathfrak{M}_0(X)$ is unital if and only if $Y$ is compact.
    \end{itemize}
\end{itemize}
\end{theorem}

\begin{proof}
Note that $\mathrm{Coz}(X)\subseteq{\mathscr B}$. Thus $\mathfrak{M}_0(X)=C^{\mathscr M}_0(X)$ by Lemma \ref{KGGF} and Proposition \ref{JFJG}. The remaining parts of the theorem follow from Lemmas \ref{KGGF} and \ref{KFFG} and Theorems \ref{TTG}, \ref{BBV}, \ref{HGBV}, \ref{UUS}, \ref{DGH}, \ref{KGV}, \ref{HFBY} and \ref{UDR}.
\end{proof}

\begin{remark}
Let $(X,{\mathscr O},{\mathscr B},\mu)$ be a topological measure space. Let
\[{\mathscr H}=\big\{B\in{\mathscr B}:\mu\big(\mathrm{cl}_X(B)\big)=0\big\}.\]
Then ${\mathscr H}$ is an ideal in $({\mathscr B},\subseteq)$. One can now state and prove results analogous to Theorem \ref{KOIJF} with the ideal ${\mathscr H}$ in place of the $\sigma$-ideal ${\mathscr M}$.
\end{remark}

\section{Continuous mappings with support having a certain topological property}\label{KIJGD}

Let $X$ be a space and let $\mathfrak{P}$ be a topological property, where at some points either $X$ or $\mathfrak{P}$ will be subject to certain mild requirements. Then
\[{\mathscr I}_\mathfrak{P}=\{A\subseteq X:\mathrm{cl}_XA\mbox{ has }\mathfrak{P}\}\]
is an ideal in $({\mathscr P}(X),\subseteq)$. For simplicity of the notation when ${\mathscr I}={\mathscr I}_\mathfrak{P}$ we denote $C^{\mathscr I}_{00}(X)$ and $C^{\mathscr I}_0(X)$ by $C^\mathfrak{P}_{00}(X)$ and $C^\mathfrak{P}_0(X)$, respectively, and we denote $\lambda_{\mathscr I}X$ by $\lambda_\mathfrak{P}X$. In this context we will have
\[C^\mathfrak{P}_{00}(X)=\big\{f\in C_b(X):\mathrm{supp}(f)\mbox{ has a closed neighborhood in $X$ with }\mathfrak{P}\big\}\]
and
\[C^\mathfrak{P}_0(X)=\big\{f\in C_b(X):|f|^{-1}\big([1/n,\infty)\big)\mbox{ has }\mathfrak{P}\mbox{ for each }n\big\}.\]
The purpose of this section is to study $C^\mathfrak{P}_{00}(X)$ and $C^\mathfrak{P}_0(X)$. The main results of Section \ref{HFPG} imply that $C^\mathfrak{P}_{00}(X)$ and $C^\mathfrak{P}_0(X)$ are isometrically isomorphic to $C_{00}(\lambda_\mathfrak{P}X)$ and $C_0(\lambda_\mathfrak{P}X)$, respectively, for a locally-$\mathfrak{P}$ space $X$. Furthermore, $X$ is dense in $\lambda_\mathfrak{P}X$ and $C^\mathfrak{P}_{00}(X)$ is dense in $C^\mathfrak{P}_0(X)$. Particular attention will be paid to locally-$\mathfrak{P}$ spaces $X$ with $\mathfrak{Q}$ where $\mathfrak{P}$ and $\mathfrak{Q}$ are topological properties that (besides some other requirements) satisfy
\[\mbox{$\mathfrak{P}$ $+$ $\mathfrak{Q}$ $\rightarrow$ The Lindel\"{o}f property}.\]
In this case, we have
\[C^\mathfrak{P}_{00}(X)=\big\{f\in C_b(X):\mathrm{supp}(f)\mbox{ has }\mathfrak{P}\big\}=C^\mathfrak{P}_0(X),\]
and
\[C_{00}(\lambda_\mathfrak{P}X)=C_0(\lambda_\mathfrak{P}X).\]
In particular, $\lambda_\mathfrak{P}X$ is countably compact. The special case in which $\mathfrak{P}$ is the Lindel\"{o}f property and $\mathfrak{Q}$ is metrizability (or paracompactness) is studied in great detail. Among other things, we show that $\lambda_\mathfrak{P}X$ is non-normal if $X$ is non-$\mathfrak{P}$. Furthermore,
\[\dim C^\mathfrak{P}_{00}(X)=\ell(X)^{\aleph_0},\]
where $\ell(X)$ is the Lindel\"{o}f number of $X$. In addition, we insert a chain of length $\lambda$, in which $\aleph_\lambda=\ell(X)$, consisting of closed ideals $H_\mu$'s of $C_b(X)$ between $C_0(X)$ and $C_b(X)$; each $H_\mu$ is of the form
\[C_{00}(Y_\mu)=C_0(Y_\mu),\]
where $Y_\mu$ is a locally compact countably compact Hausdorff space which contains $X$ densely. The concluding results of this section deal with realcompactness and pseudocompactness. We show that if $\mathfrak{P}$ is realcompactness, then
\[\lambda_\mathfrak{P}X=\beta X\setminus\mathrm{cl}_{\beta X}(\upsilon X\setminus X),\]
where $\upsilon X$ is the Hewitt realcompactification of $X$, and if $\mathfrak{P}$ is pseudocompactness, then
\[\lambda_{\mathscr U}X=\mathrm{int}_{\beta X}\upsilon X\]
for the ideal
\[{\mathscr U}=\big\{A\in\mathscr{RC}(X):A\mbox{ is pseudocompact}\big\}\]
of $(\mathscr{RC}(X),\subseteq)$.

Now we proceed with the formal treatment of the subject. Theorems \ref{JGF} and \ref{GKGF} improve results from \cite{Ko10} and \cite{Ko11}, respectively, Proposition \ref{HJGL} is from \cite{Ko10}, Theorem \ref{CTTF} is actually the main result of \cite{Ko6} rephrased in the new context, Theorem \ref{HGFLK} has been observed in \cite{Ko11}, and Theorems \ref{GGFD}, \ref{PTF} and \ref{RRGY} reproves results from \cite{Ko10}, \cite{Ko4} and \cite{Ko12}, respectively.

\begin{definition}\label{HHI}
Let $\mathfrak{P}$ be a topological property. Then
\begin{itemize}
  \item $\mathfrak{P}$ is \textit{closed hereditary}, if any closed subspace of a space with $\mathfrak{P}$ also has $\mathfrak{P}$.
  \item $\mathfrak{P}$ is \textit{preserved under finite} (resp. \textit{countable}) \textit{closed sums}, if any space which is expressible as a finite (resp. countable) union of its closed subspaces each having $\mathfrak{P}$ also has $\mathfrak{P}$.
\end{itemize}
\end{definition}

\begin{notation}\label{GGH}
Let $X$ be a space and let $\mathfrak{P}$ be a topological property. Denote
\[{\mathscr I}_\mathfrak{P}=\{A\subseteq X:\mathrm{cl}_XA\mbox{ has }\mathfrak{P}\}.\]
\end{notation}

We may use the following lemma without explicitly referring to it.

\begin{lemma}\label{KGF}
Let $X$ be a space and let $\mathfrak{P}$ be a closed hereditary topological property preserved under finite closed sums.
Then $\mathscr{I}_\mathfrak{P}$ is an ideal in $({\mathscr P}(X),\subseteq)$.
\end{lemma}

\begin{proof}
Let $B\subseteq A$ with $A\in\mathscr{I}_\mathfrak{P}$. Then $\mathrm{cl}_XB\subseteq\mathrm{cl}_XA$. Since $\mathrm{cl}_XA$ has $\mathfrak{P}$ and $\mathfrak{P}$ is closed hereditary, the closed subspace $\mathrm{cl}_XB$ of $\mathrm{cl}_XA$ also has $\mathfrak{P}$. That is $B\in\mathscr{I}_\mathfrak{P}$. Next, let $A_i\in\mathscr{I}_\mathfrak{P}$ for each $i=1,\ldots,n$. Since
\[\mathrm{cl}_X(A_1\cup\cdots\cup A_n)=\mathrm{cl}_XA_1\cup\cdots\cup\mathrm{cl}_XA_n\]
and the latter has $\mathfrak{P}$, as $\mathrm{cl}_XA_i$ has $\mathfrak{P}$ for each $i=1,\ldots,n$ and $\mathfrak{P}$ is preserved under finite closed sums, we have $A_1\cup\cdots\cup A_n\in\mathscr{I}_\mathfrak{P}$.
\end{proof}

\begin{example}\label{GLL}
The list of topological properties $\mathfrak{P}$ which are closed hereditary and preserved under finite closed sums (thus satisfying the assumption of Lemma \ref{KGF}) is quite long and include almost all major covering properties (that is, topological properties described in terms of the existence of certain kinds of open subcovers or refinements of a given open cover of a certain type); among them are: (1) compactness (2) $[\theta,\kappa]$-compactness (in particular, countable compactness and the Lindel\"{o}f property) (3) paracompactness (4) metacompactness (5) countable paracompactness (6) subparacompactness (7) $\theta$-refinability (or submetacompactness) (8) the $\sigma$-para-Lindel\"{o}f property (9) $\delta\theta$-refinability (or the submeta-Lindel\"{o}f property) (10) weak $\theta$-refinability, and finally (11) weak $\delta\theta$-refinability. (See \cite{Bu}, \cite{Steph} and \cite{Va} for definitions.) These topological properties are all closed hereditary (this is obvious for (1), see Theorem 7.1 of \cite{Bu} for (3)--(4) and (6)--(11), Theorem 3.1 of \cite{Steph} for (2) and Exercise 5.2.B of \cite{E} for (5)) and are preserved under finite closed sums (this follows from the definition for (2) and is obvious for (1), see Theorems 7.3 and 7.4 of \cite{Bu} for (3)--(4) and (6)--(11), and Theorem 3.7.22 and Exercises 5.2.B and 5.2.G of \cite{E} for (5)).

There are examples of topological properties, not generally considered as a covering property, which are closed hereditary and preserved under finite closed sums. We only mention $\alpha$-boundedness. (A space is called \textit{$\alpha$-bounded}, where $\alpha$ is an infinite cardinal, if every subspace of it of cardinality $\leq\alpha$ has compact closure in it.) That $\alpha$-boundedness is closed hereditary and preserved under finite closed sums follows easily from its definition.
\end{example}

\begin{definition}\label{KI}
Let $\mathfrak{P}$ be a topological property. A space $X$ is called \textit{locally-$\mathfrak{P}$} if each $x\in X$ has a neighborhood in $X$ with $\mathfrak{P}$.
\end{definition}

\begin{lemma}\label{KFFF}
Let $X$ be a regular space and let $\mathfrak{P}$ be a closed hereditary topological property preserved under finite closed sums. Then $X$ is locally null (with respect to the ideal $\mathscr{I}_\mathfrak{P}$) if and only if $X$ is locally-$\mathfrak{P}$.
\end{lemma}

\begin{proof}
Note that $\mathscr{I}_\mathfrak{P}$ is an ideal in $({\mathscr P}(X),\subseteq)$ by Lemma \ref{KGF}. Let $X$ be locally null and let $x\in X$. Then $x\in\mathrm{int}_XU$ for some $U\in\mathscr{I}_\mathfrak{P}$. Thus $\mathrm{cl}_XU$ is a neighborhood of $x$ in $X$ with $\mathfrak{P}$. For the converse, let $x\in X$ have a neighborhood $V$ in $X$ with $\mathfrak{P}$. There exists an open neighborhood $W$ of $x$ in $X$ with $\mathrm{cl}_XW\subseteq V$. Now, $\mathrm{cl}_XW$ has $\mathfrak{P}$, as it is closed in $V$ and $V$ has $\mathfrak{P}$. Thus $W$ is null.
\end{proof}

We introduce the following notation for convenience.

\begin{notation}\label{KHFG}
Let $X$ be a space and let $\mathfrak{P}$ be a closed hereditary topological property preserved under finite closed sums. Denote
\[C^\mathfrak{P}_{00}(X)=C^{\mathscr I_\mathfrak{P}}_{00}(X)\;\;\;\;\mbox{ and }\;\;\;\;C^\mathfrak{P}_0(X)=C^{\mathscr I_\mathfrak{P}}_0(X).\]
Also, denote
\[\lambda_\mathfrak{P}X=\lambda_{\mathscr I_\mathfrak{P}}X.\]
\end{notation}

\begin{theorem}\label{JGF}
Let $X$ be a space and let $\mathfrak{P}$ be a closed hereditary topological property preserved under finite closed sums. Then
\[C^\mathfrak{P}_{00}(X)=\big\{f\in C_b(X):\mathrm{supp}(f)\mbox{ has a closed neighborhood in $X$ with }\mathfrak{P}\big\},\]
and in particular,
\[C^\mathfrak{P}_{00}(X)=\big\{f\in C_b(X):\mathrm{supp}(f)\mbox{ has a neighborhood in $X$ with }\mathfrak{P}\big\},\]
if $X$ is normal. Furthermore,
\begin{itemize}
\item[\rm(1)] $C^\mathfrak{P}_{00}(X)$ is an ideal in $C_b(X)$.
\item[\rm(2)] Let $X$ be completely regular. Then $C^\mathfrak{P}_{00}(X)$ is of empty hull if and only if $X$ is locally-$\mathfrak{P}$.
\item[\rm(3)] Let $X$ be completely regular and locally-$\mathfrak{P}$. Then $C^\mathfrak{P}_{00}(X)$ is unital if and only if $X$ has $\mathfrak{P}$.
\item[\rm(4)] Let $X$ be normal and locally-$\mathfrak{P}$. Then $C^\mathfrak{P}_{00}(X)$ is a normed algebra isometrically isomorphic to $C_{00}(Y)$ for some unique locally compact Hausdorff space $Y$, namely $Y=\lambda_\mathfrak{P}X$. Furthermore,
    \begin{itemize}
    \item $X$ is dense in $Y$.
    \item $C^\mathfrak{P}_{00}(X)$ is unital if and only if $X$ has $\mathfrak{P}$ if and only if $Y$ is compact.
    \end{itemize}
\end{itemize}
\end{theorem}

\begin{proof}
Let $f\in C_b(X)$. Suppose that $\mathrm{supp}(f)$ has a closed neighborhood $U$ in $X$ with $\mathfrak{P}$ then $\mathrm{supp}(f)$ has a null neighborhood in $X$, namely, $U$ itself. Thus $f\in C^\mathfrak{P}_{00}(X)$. On the other hand, if $f\in C^\mathfrak{P}_{00}(X)$, then $\mathrm{supp}(f)$ has a null neighborhood $V$ in $X$. Thus $\mathrm{cl}_X V$ is a closed neighborhood of $\mathrm{supp}(f)$ in $X$ with $\mathfrak{P}$. This shows the first representation of $C^\mathfrak{P}_{00}(X)$. Note that if $X$ is normal, then any neighborhood of $\mathrm{supp}(f)$ in $X$, where $f\in C_b(X)$, contains a closed neighborhood of $\mathrm{supp}(f)$ in $X$. The second representation of $C^\mathfrak{P}_{00}(X)$ then follows, as $\mathfrak{P}$ is closed hereditary. The remaining assertions of the theorem follow from Lemma \ref{KFFF} (and Lemma \ref{KGF}) and Theorems \ref{TTG}, \ref{BBV}, \ref{HGBV} and \ref{UUS}.
\end{proof}

\begin{theorem}\label{HJG}
Let $X$ be a space and let $\mathfrak{P}$ be a closed hereditary topological property preserved under finite closed sums. Then
\[C^\mathfrak{P}_0(X)=\big\{f\in C_b(X):|f|^{-1}\big([1/n,\infty)\big)\mbox{ has }\mathfrak{P}\mbox{ for each }n\big\},\]
and in particular,
\[C^\mathfrak{P}_0(X)=\big\{f\in C_b(X):\mathrm{coz}(f)\mbox{ has }\mathfrak{P}\big\},\]
if $\mathfrak{P}$ is preserved under countable closed sums. Furthermore,
\begin{itemize}
\item[\rm(1)] $C^\mathfrak{P}_0(X)$ is a closed ideal in $C_b(X)$.
\item[\rm(2)] $C^\mathfrak{P}_{00}(X)\subseteq C^\mathfrak{P}_0(X)$.
\item[\rm(3)] Let $X$ be completely regular. Then $C^\mathfrak{P}_0(X)$ is of empty hull if and only if $X$ is locally-$\mathfrak{P}$.
\item[\rm(4)] Let $X$ be completely regular and locally-$\mathfrak{P}$. Then $C^\mathfrak{P}_0(X)$ is unital if and only if $X$ has $\mathfrak{P}$.
\item[\rm(5)] Let $X$ be completely regular and locally-$\mathfrak{P}$. Furthermore, suppose that either
\begin{itemize}
\item $X$ is normal, or
\item $\mathfrak{P}$ is preserved under countable closed sums.
\end{itemize}
Then $C^\mathfrak{P}_0(X)$ is a Banach algebra isometrically isomorphic to $C_0(Y)$ for some unique locally compact Hausdorff space $Y$, namely $Y=\lambda_\mathfrak{P}X$. Furthermore,
    \begin{itemize}
    \item $X$ is dense in $Y$.
    \item $C^\mathfrak{P}_{00}(X)$ is dense in $C^\mathfrak{P}_0(X)$.
    \item $C^\mathfrak{P}_0(X)$ is unital if and only if $X$ has $\mathfrak{P}$ if and only if $Y$ is compact.
    \end{itemize}
\end{itemize}
\end{theorem}

\begin{proof}
Let $n$ be a positive integer. Then $|f|^{-1}([1/n,\infty))$ (since it is closed in $X$) is null if and only if it has $\mathfrak{P}$. The given representation of $C^\mathfrak{P}_0(X)$ then follows from Proposition \ref{JJG} (and Lemma \ref{KGF}). We will consider only the case when $\mathfrak{P}$ is preserved under countable closed sums; the other assertions and remaining cases follow from Lemma \ref{KFFF} and Theorems \ref{DGH}, \ref{KGV}, \ref{HFBY} and \ref{UDR}.

Let $X$ be completely regular and let $\mathfrak{P}$ be preserved under countable closed sums. Let $f\in C^\mathfrak{P}_0(X)$. By the representation of $C^\mathfrak{P}_0(X)$ we know that $|f|^{-1}([1/n,\infty))$ has $\mathfrak{P}$ for each positive integer $n$. But
\[\mathrm{coz}(f)=\bigcup_{n=1}^\infty|f|^{-1}\big([1/n,\infty)\big).\]
Thus $\mathrm{coz}(f)$ is a countable union of its closed subspaces each with $\mathfrak{P}$. Since $\mathfrak{P}$ is preserved under countable closed sums it follows that $\mathrm{coz}(f)$ has $\mathfrak{P}$. Next, let $g\in C_b(X)$ such that $\mathrm{coz}(g)$ has $\mathfrak{P}$. Then $|g|^{-1}([1/n,\infty))$ has $\mathfrak{P}$ for each positive integer $n$, as it is closed in $\mathrm{coz}(g)$ and $\mathfrak{P}$ is closed hereditary. Therefore $g\in C^\mathfrak{P}_0(X)$ by the first representation of $C^\mathfrak{P}_0(X)$. This proves the (second) representation for $C^\mathfrak{P}_0(X)$.

Next, let $X$ be completely regular and locally-$\mathfrak{P}$ and let $\mathfrak{P}$ be preserved under countable closed sums. Note that $C^\mathfrak{P}_0(X)$ is a Banach subalgebra of $C_b(X)$. We verify that $C^\mathfrak{P}_0(X)$ satisfies the assumption of Theorem \ref{HGL}. Let $x\in X$. Since $X$ is locally-$\mathfrak{P}$ there exists a neighborhood $U$ of $x$ in $X$ with $\mathfrak{P}$. Let $f:X\rightarrow[0,1]$ be continuous and such that
\[f(x)=1\;\;\;\;\mbox{ and }\;\;\;\;f|_{X\setminus\mathrm{int}_X U}=\mathbf{0}.\]
Then $f^{-1}([1/n,1])$ has $\mathfrak{P}$ for each positive integer $n$, as it is closed in $U$. Arguing as above it follows that $\mathrm{coz}(f)$ has $\mathfrak{P}$. Thus $f\in C^\mathfrak{P}_0(X)$ by (5). This shows that $C^\mathfrak{P}_0(X)$ satisfies condition (1) in Theorem \ref{HGL}. Next, let $\mathrm{coz}(g)\subseteq\mathrm{coz}(h)$ with $g\in C_b(X)$ and $h\in C^\mathfrak{P}_0(X)$. Then $|g|^{-1}([1/n,\infty))$ has $\mathfrak{P}$ for each positive integer $n$, as it is closed in $\mathrm{coz}(h)$ and the latter has $\mathfrak{P}$ by (5). By the above argument $\mathrm{coz}(g)$ has $\mathfrak{P}$. Therefore $g\in C^\mathfrak{P}_0(X)$ by the representation of $C^\mathfrak{P}_0(X)$. This shows that $C^\mathfrak{P}_0(X)$ satisfies condition (2) in Theorem \ref{HGL} as well. Now (5) follows from Theorem \ref{HGL}.
\end{proof}

The following theorem introduces conditions under which $C^\mathfrak{P}_0(X)=C^\mathfrak{P}_{00}(X)$. It also simplifies the representationes of $C^\mathfrak{P}_0(X)$ and $C^\mathfrak{P}_{00}(X)$. Examples of topological properties $\mathfrak{P}$ and $\mathfrak{Q}$ satisfying the assumption of the next theorem are given in Example \ref{GGK}. (See also Theorems \ref{CTTF} and \ref{HGFLK}.) Recall that if $X$ is a locally compact Hausdorff space, then $C_0(X)=C_{00}(X)$ implies that $X$ is countably compact; this will be used in the following.

\begin{theorem}\label{GKGF}
Let $\mathfrak{P}$ and $\mathfrak{Q}$ be topological properties such that
\begin{itemize}
\item $\mathfrak{P}$ and $\mathfrak{Q}$ are closed hereditary.
\item A space with both $\mathfrak{P}$ and $\mathfrak{Q}$ is Lindel\"{o}f.
\item $\mathfrak{P}$ is preserved under countable closed sums.
\item A space with $\mathfrak{Q}$ having a dense subspace with $\mathfrak{P}$ has $\mathfrak{P}$.
\end{itemize}
Let $X$ be a locally-$\mathfrak{P}$ space with $\mathfrak{Q}$.
\begin{itemize}
\item[\rm(1)] Let $X$ be regular. Then
\[C^\mathfrak{P}_0(X)=\big\{f\in C_b(X):\mathrm{supp}(f)\mbox{ has }\mathfrak{P}\big\}=C^\mathfrak{P}_{00}(X).\]
\item[\rm(2)] Let $X$ be normal. Then $C^\mathfrak{P}_0(X)$ is a Banach algebra isometrically isomorphic to $C_0(Y)$ for some unique (up to homeomorphism) locally compact Hausdorff space $Y$, namely $Y=\lambda_\mathfrak{P}X$. Furthermore,
    \begin{itemize}
    \item $Y$ is countably compact.
    \item $C_0(Y)=C_{00}(Y)$.
    \end{itemize}
\end{itemize}
\end{theorem}

\begin{proof}
We first verify that $\mathscr{I}_\mathfrak{P}$ is a $\sigma$-ideal in $({\mathscr P}(X),\subseteq)$. Since $\mathscr{I}_\mathfrak{P}$ is an ideal in $({\mathscr P}(X),\subseteq)$ by Lemma \ref{KGF}, we only check that $\mathscr{I}_\mathfrak{P}$ is closed under formation of countable unions. Let $A_n\in\mathscr{I}_\mathfrak{P}$ for each positive integer $n$. Then
\[G=\bigcup_{n=1}^\infty\mathrm{cl}_XA_n\]
has $\mathfrak{P}$, as it is a countable union of its closed subspaces each with $\mathfrak{P}$ and $\mathfrak{P}$ is preserved under countable closed sums. Note that
\[H=\mathrm{cl}_X\bigg(\bigcup_{n=1}^\infty A_n\bigg)\]
has $\mathfrak{Q}$, as it is closed in $X$ and $X$ has $\mathfrak{Q}$. Since $H$ contains $G$ as a dense subspace, using our assumption it follows that $H$ has $\mathfrak{P}$. Therefore
\[\bigcup_{n=1}^\infty A_n\in\mathscr{I}_\mathfrak{P}.\]

(1). We first show that
\begin{equation}\label{FGH}
C^\mathfrak{P}_{00}(X)=\big\{f\in C_b(X):\mathrm{supp}(f)\mbox{ has }\mathfrak{P}\big\}.
\end{equation}
Let $f\in C^\mathfrak{P}_{00}(X)$. Then $\mathrm{supp}(f)$ has a null neighborhood in $X$, that is, $\mathrm{supp}(f)\subseteq\mathrm{int}_XU$ for some subspace $U$ of $X$ such that $\mathrm{cl}_XU$ has $\mathfrak{P}$. Therefore $\mathrm{supp}(f)$ has $\mathfrak{P}$, as it is closed in $\mathrm{cl}_XU$.

Next, let $f\in C_b(X)$ such that $\mathrm{supp}(f)$ has $\mathfrak{P}$. Note that $X$ is locally null by Lemma \ref{KFFF}. For each $x\in X$ let $U_x$ be a null neighborhood of $x$ in $X$. Then
\[\mathrm{supp}(f)\subseteq\bigcup_{x\in X}\mathrm{int}_XU_x.\]
Note that $\mathrm{supp}(f)$ has $\mathfrak{Q}$, as it closed in $X$ and $X$ has $\mathfrak{Q}$. Since $\mathrm{supp}(f)$ has $\mathfrak{P}$, it is then Lindel\"{o}f by our assumption. Let $x_1,x_2,\ldots\in X$ such that
\[\mathrm{supp}(f)\subseteq\bigcup_{n=1}^\infty\mathrm{int}_XU_{x_n}.\]
Note that $\mathscr{I}_\mathfrak{P}$ is a $\sigma$-ideal. Therefore
\[\bigvee_{n=1}^\infty U_{x_n}\bigg(=\bigcup_{n=1}^\infty U_{x_n},\mbox{ by the above argument}\bigg)\]
is a null neighborhood of $\mathrm{supp}(f)$ in $X$. Thus $f\in C^\mathfrak{P}_{00}(X)$. This shows (\ref{FGH}).

Now, we show that $C^\mathfrak{P}_0(X)=C^\mathfrak{P}_{00}(X)$. Since $C^\mathfrak{P}_{00}(X)\subseteq C^\mathfrak{P}_0(X)$ by Theorem \ref{DGH}, we only check that $C^\mathfrak{P}_0(X)\subseteq C^\mathfrak{P}_{00}(X)$. Let $f\in C^\mathfrak{P}_0(X)$. Then $|f|^{-1}([1/n,\infty))$ has $\mathfrak{P}$ for each positive integer $n$ (as it has a null neighborhood in $X$). But then $\mathrm{coz}(f)$ has $\mathfrak{P}$, as it is a countable union of its closed subspaces $|f|^{-1}([1/n,\infty))$ each with $\mathfrak{P}$. Note that $\mathrm{supp}(f)$ has $\mathfrak{Q}$, as it is closed in $X$ and $X$ has $\mathfrak{Q}$. Therefore $\mathrm{supp}(f)$ has $\mathfrak{P}$, as it contains $\mathrm{coz}(f)$ as a dense subspace and $\mathrm{coz}(f)$ has $\mathfrak{P}$. Thus $f\in C^\mathfrak{P}_{00}(X)$ by (\ref{FGH}).

(2). Let $Y=\lambda_\mathfrak{P} X$. Let
\[\phi:C^\mathfrak{P}_0(X)\rightarrow C_0(Y)\]
denote the isometric isomorphism as defined in the proof of Theorem \ref{UDR}. As it is observed in the proof Theorem \ref{UDR}, the mapping $\phi$, when restricted to $C^\mathfrak{P}_{00}(X)$, induces an isometric isomorphism between $C^\mathfrak{P}_{00}(X)$ and $C_{00}(Y)$. But then $C_0(Y)=C_{00}(Y)$, as $C^\mathfrak{P}_0(X)=C^\mathfrak{P}_{00}(X)$ by (1). This in particular implies that $Y$ is countably compact.
\end{proof}

\begin{remark}\label{PF}
Observe that in Theorem \ref{GKGF} we indeed have
\[C^\mathfrak{P}_0(X)=C^{\mathfrak{P}+\mathfrak{Q}}_0(X)=C^{\mathfrak{P}+\mathfrak{Q}}_{00}(X)=C^\mathfrak{P}_{00}(X).\]
To show this, it suffices to observe that $\mathscr{I}_\mathfrak{P}=\mathscr{I}_{\mathfrak{P}+\mathfrak{Q}}$. It is obvious that $\mathscr{I}_{\mathfrak{P}+\mathfrak{Q}}\subseteq\mathscr{I}_\mathfrak{P}$. Let $A\in\mathscr{I}_\mathfrak{P}$. Then $\mathrm{cl}_XA$ has $\mathfrak{P}$. But $\mathrm{cl}_XA$ also has $\mathfrak{Q}$, as it is closed in $X$ and $X$ has $\mathfrak{Q}$. Thus $A\in\mathscr{I}_{\mathfrak{P}+\mathfrak{Q}}$.
\end{remark}

In the following we give examples of topological properties $\mathfrak{P}$ and $\mathfrak{Q}$ satisfying the assumption of Theorem \ref{GKGF}.

\begin{example}\label{GGK}
For any open covers $\mathscr{U}$ and $\mathscr{V}$ of a space $X$ we say that $\mathscr{U}$ is a \textit{refinement} of $\mathscr{V}$ if each element of $\mathscr{U}$ is contained in an element of $\mathscr{V}$. An open cover $\mathscr{U}$ of $X$ is called \textit{locally finite} if each point of $X$ has an open neighborhood in $X$ intersecting only a finite number of the elements of $\mathscr{U}$. A regular space $X$ is called \textit{paracompact} if for every open cover $\mathscr{U}$ of $X$ there is an open cover of $X$ which refines $\mathscr{U}$. Paracompact spaces are generally considered as the simultaneous generalizations of compact Hausdorff spaces and metrizable spaces. Every metrizable space as well as every compact Hausdorff space is paracompact and every paracompact space is normal. (See Theorems 5.1.1, 5.1.3 and 5.1.5 of \cite{E}.)

Let $\mathfrak{P}$ be the Lindel\"{o}f property and let $\mathfrak{Q}$ be either metrizability or paracompactness. Note that the Lindel\"{o}f property is closed hereditary (see Theorem 3.8.4 of \cite{E}) and is preserved under countable closed sums, as obviously, any space which is the union of a countable number of its Lindel\"{o}f subspaces is Lindel\"{o}f. Note that any subspace of a metrizable space is metrizable and a closed subspace of a paracompact space is paracompact. (See Corollary 5.1.29 of \cite{E}.) Also, any paracompact space with a dense Lindel\"{o}f space is Lindel\"{o}f (see Theorem 5.1.25 of \cite{E}), since any metrizable space is paracompact, it then follows that any metrizable space with a dense Lindel\"{o}f space is Lindel\"{o}f. Therefore, the pair $\mathfrak{P}$ and $\mathfrak{Q}$ satisfy the assumption of Theorem \ref{GKGF}.

Note that in the realm of metrizable spaces the Lindel\"{o}f property coincides with separability and second countability; denote the latter two by $\mathfrak{S}$ and $\mathfrak{C}$, respectively, and denote the Lindel\"{o}f property by $\mathfrak{L}$. Thus, if $X$ is a metrizable space, then
\[\mathscr{I}_\mathfrak{L}=\mathscr{I}_\mathfrak{S}=\mathscr{I}_\mathfrak{C},\]
which implies that
\[C^\mathfrak{L}_{00}(X)=C^\mathfrak{S}_{00}(X)=C^\mathfrak{C}_{00}(X)\;\;\;\;\mbox{ and }\;\;\;\;C^\mathfrak{L}_0(X)=C^\mathfrak{S}_0(X)=C^\mathfrak{C}_0(X).\]
\end{example}

Observe that for a paracompact locally-$\mathfrak{P}$ space $X$, where $\mathfrak{P}$ is the Lindel\"{o}f property, we have
\[C^\mathfrak{P}_{00}(X)=\big\{f\in C_b(X):\mathrm{supp}(f)\mbox{ has }\mathfrak{P}\big\},\]
by Theorem \ref{GKGF} (and Example \ref{GGK}). As we see in the following, for the case when $\mathfrak{P}$ is the Lindel\"{o}f property the above representation remains true even if one requires $X$ to be only normal.

\begin{proposition}\label{HJGL}
Let $X$ be a normal locally-$\mathfrak{P}$ space, where $\mathfrak{P}$ is the Lindel\"{o}f property. Then
\[C^\mathfrak{P}_{00}(X)=\big\{f\in C_b(X):\mathrm{supp}(f)\mbox{ has }\mathfrak{P}\big\}.\]
\end{proposition}

\begin{proof}
It is trivial that $\mathrm{supp}(f)$ has $\mathfrak{P}$ for any $f\in C^\mathfrak{P}_{00}(X)$, as $\mathrm{supp}(f)$ has a neighborhood in $X$ with $\mathfrak{P}$.

Next, let $f\in C_b(X)$ such that $\mathrm{supp}(f)$ has $\mathfrak{P}$. For each $x\in X$, let $U_x$ be a neighborhood of $x$ in $X$ with $\mathfrak{P}$. Then
\[\mathrm{supp}(f)\subseteq U_{x_1}\cup U_{x_2}\cup\cdots=W,\]
for some $x_1,x_2,\ldots\in\ X$. Since $X$ is normal, there exists an open subspace $V$ of $X$ with
\[\mathrm{supp}(f)\subseteq V\subseteq\mathrm{cl}_X V\subseteq W.\]
Thus, $V$ is a neighborhood of $\mathrm{supp}(f)$ in $X$, whose closure $\mathrm{cl}_X V$ has $\mathfrak{P}$, as it is closed in $W$ and $W$ has $\mathfrak{P}$. Therefore $f\in C^\mathfrak{P}_{00}(X)$.
\end{proof}

Let $X$ be a locally separable metrizable space. The study of the Banach subalgebra $C_s(X)$ of $C_b(X)$, consisting of those elements of $C_b(X)$ with separable support, constitutes the subject matter of our next result. We need some preliminaries before we further proceed.

The \textit{density} of a space $X$, denoted by $d(X)$, is defined by
\[d(X)=\min\big\{|D|:D\mbox{ is dense in }X\big\}+\aleph_0.\]
In particular, a space $X$ is separable if and only if $d(X)=\aleph_0$. Observe that in any metrizable space separability coincides with second countability (and the Lindel\"{o}f property); thus any subspace of a separable metrizable space is separable. A theorem of Alexandroff states that any locally separable metrizable space $X$ can be represented as a disjoint union
\[X=\bigcup_{i\in I}X_i,\]
where $I$ is an index set, and $X_i$ is a non-empty separable (and thus Lindel\"{o}f) open-closed subspace of $X$ for each $i\in I$. (See Problem 4.4.F of \cite{E}.) Note that $d(X)=|I|$, if $I$ is infinite.

Consider an uncountable discrete space $D$. Denote by $D_\lambda$ the subspace of $\beta D$ consisting of all elements in the closure in $D$ of countable subsets of $D$. It is known that there exists a continuous (2-valued) mapping $f:D_\lambda\setminus D\rightarrow[0,1]$ which is not continuously extendible over $\beta D\setminus D$. (See \cite{W}.) In particular, $D_\lambda$ is then not normal. (To see this, suppose the contrary. Note that $D_\lambda\setminus D$ is closed in $D_\lambda$, as $D$ is open in $\beta D$, since $D$ is locally compact. By the Tietze--Urysohn Extension Theorem, $f$ is extendible to a continuous bounded mapping over $D_\lambda$, and therefore over the whole $\beta D_\lambda$; note that $\beta D_\lambda=\beta D$, as $D\subseteq D_\lambda\subseteq\beta D$. This, however, is not possible.)

A theorem of Tarski ensures for an infinite set $I$ the existence of a collection $\mathscr{I}$ of cardinality $|I|^{\aleph_0}$ consisting of countable infinite subsets of $I$, such that the intersection of any two distinct elements of $\mathscr{I}$ is finite. (See \cite{Ho}.) Note that the collection of all subsets of cardinality at most $\mathfrak{m}$ in a set with cardinality $\mathfrak{n}\geq\mathfrak{m}$ is of cardinality at most $\mathfrak{n}^\mathfrak{m}$.

Observe that if $X$ is a space and $D$ is a subspace of $X$, then
\[U\cap\mathrm{cl}_XD=\mathrm{cl}_X(U\cap D)\]
for every open-closed subspace $U$ of $X$. This simple observation will be used in the following.

\begin{theorem}\label{CTTF}
Let $X$ be a locally separable metrizable space. Let
\[C_s(X)=\big\{f\in C_b(X):\mathrm{supp}(f)\mbox{ is separable}\big\}.\]
Then $C_s(X)$ is a Banach algebra isometrically isomorphic to $C_0(Y)$ for some unique (up to homeomorphism) locally compact Hausdorff space $Y$, namely
\[Y=\bigcup\{\mathrm{cl}_{\beta X}S:S\subseteq X\mbox{ is separable}\}.\]
Furthermore,
\begin{itemize}
\item[\rm(1)] $Y$ is countably compact.
\item[\rm(2)] $Y$ is non-normal if $X$ is non-separable.
\item[\rm(3)] $C_0(Y)=C_{00}(Y)$.
\item[\rm(4)] $\dim C_s(X)=d(X)^{\aleph_0}$.
\end{itemize}
\end{theorem}

\begin{proof}
Let $\mathfrak{P}$ be the Lindel\"{o}f property and let $\mathfrak{Q}$ be metrizability. Then, as we have seen in Example \ref{GGK}, the pair $\mathfrak{P}$ and $\mathfrak{Q}$ satisfy the assumption of Theorem \ref{GKGF}. Consider a representation
\[X=\bigcup_{i\in I}X_i,\]
of $X$, where $X_i$'s are disjoint non-empty separable (and therefore Lindel\"{o}f) open-closed subspaces of $X$ and $I$ is an index set. For convenience, denote
\[H_J=\bigcup_{i\in J}X_i\]
for any $J\subseteq I$. Observe that each $H_J$ is open-closed in $X$, thus it has open-closed closure in $\beta X$. Also, as we will see now
\begin{equation}\label{DSY}
\lambda_\mathfrak{P} X=\bigcup\{\mathrm{cl}_{\beta X}H_J:J\subseteq I\mbox{ is countable}\}.
\end{equation}
Let $C\in\mathrm{Coz}(X)$ such that $\mathrm{cl}_X C$ has a neighborhood $U$ in $X$ such that $\mathrm{cl}_X U$ is Lindel\"{o}f. Then $\mathrm{cl}_X C$ itself is Lindel\"{o}f, as it is closed in $\mathrm{cl}_X U$, and therefore $\mathrm{cl}_X C\subseteq H_J$ for some countable $J\subseteq I$. Thus $\mathrm{cl}_{\beta X}C\subseteq\mathrm{cl}_{\beta X}H_J$. On the other hand, if $J\subseteq I$ is countable, then $H_J$ is a cozero-set in $X$, as it is open-closed in $X$, and it is Lindel\"{o}f, as it is a countable union of Lindel\"{o}f subspaces. Since $\mathrm{cl}_{\beta X}H_J$ is open in $\beta X$ we have
\[\mathrm{cl}_{\beta X}H_J=\mathrm{int}_{\beta X}\mathrm{cl}_{\beta X}H_J\subseteq\lambda_\mathfrak{P} X.\]
This shows (\ref{DSY}). Note that
\[\lambda_\mathfrak{P} X=\bigcup\{\mathrm{cl}_{\beta X}S:S\subseteq X\mbox{ is separable}\};\]
by (\ref{DSY}) in particular; simply observe that any separable subspace of $X$ is Lindel\"{o}f, and is then contained in $H_J$ for some countable $J\subseteq I$, and on the other hand, $H_J$ is separable for any countable $J\subseteq I$, as it is a countable union of separable spaces.

Note that (3) implies (1); we prove (3). (Observe that (1) and (3) follow from Theorem \ref{GKGF}; the proof given here, however, is independent from Theorem \ref{GKGF}.)

(3). It suffices to check that every $\sigma$-compact subspace of $\lambda_\mathfrak{P} X$ is contained in a compact subspace of $\lambda_\mathfrak{P} X$. Let \[A=\bigcup_{n=1}^\infty A_n,\]
where each $A_n$ is compact, be a $\sigma$-compact subspace of $\lambda_\mathfrak{P} X$. Using (\ref{DSY}), by compactness we have
\begin{equation}\label{DJB}
A_n\subseteq\mathrm{cl}_{\beta X}H_{J_1}\cup\cdots\cup\mathrm{cl}_{\beta X}H_{J_{k_n}}
\end{equation}
for some countable $J_1,\ldots,J_{k_n}\subseteq I$. Let
\begin{equation}\label{FFB}
J=\bigcup_{n=1}^\infty(J_{k_1}\cup\cdots\cup J_{k_n}).
\end{equation}
Then $J$ is countable and $A\subseteq\mathrm{cl}_{\beta X}H_J$. That is $A$ is contained in the compact subspace $\mathrm{cl}_{\beta X}H_J$ of $\lambda_\mathfrak{P} X$.

(2). Let $x_i\in X_i$ for each $i\in I$. Then $D=\{x_i:i\in I\}$ is a closed discrete subspace of $X$, and since $X$ is non-separable, it is uncountable. Suppose to the contrary that $\lambda_\mathfrak{P} X$ is normal. Then, using (\ref{DSY}), it follows that
\[\lambda_\mathfrak{P} X\cap\mathrm{cl}_{\beta X}D=\bigcup\{\mathrm{cl}_{\beta X}H_J\cap\mathrm{cl}_{\beta X}D:J\subseteq I\mbox{ is countable}\}\]
is normal, as it is closed in $\lambda_\mathfrak{P} X$. Now, let $J\subseteq I$ be countable. Since $\mathrm{cl}_{\beta X}H_J$ is open-closed in $\beta X$ (using the observation preceding the statement of the theorem) we have
\[\mathrm{cl}_{\beta X}H_J\cap\mathrm{cl}_{\beta X}D=\mathrm{cl}_{\beta X}(\mathrm{cl}_{\beta X}H_J\cap D)=\mathrm{cl}_{\beta X}(H_J\cap D)=\mathrm{cl}_{\beta X}\big(\{x_i:i\in J\}\big).\]
But $\mathrm{cl}_{\beta X}D=\beta D$, as $D$ is closed in $X$ and $X$ is normal. Therefore
\[\mathrm{cl}_{\beta X}\big(\{x_i:i\in J\}\big)=\mathrm{cl}_{\beta X}\big(\{x_i:i\in J\}\big)\cap\mathrm{cl}_{\beta X}D=\mathrm{cl}_{\beta D}\big(\{x_i:i\in J\}\big).\]
Thus
\[\lambda_\mathfrak{P} X\cap\mathrm{cl}_{\beta X}D=\bigcup\{\mathrm{cl}_{\beta D}E:E\subseteq D\mbox{ is countable}\}=D_\lambda,\]
contradicting the fact that $D_\lambda$ is not normal.

(4). Since $X$ is non-separable, $I$ is infinite and $d(X)=|I|$. Let $\mathscr{I}$ be a collection of cardinality $|I|^{\aleph_0}$ consisting of countable infinite subsets of $I$, such that the intersection of any two distinct elements of $\mathscr{I}$ is finite. Let $f_J=\chi_{H_J}$ for any $J\in\mathscr{I}$. No element in
\[\mathscr{F}=\{f_J:J\in\mathscr{I}\}\]
is a linear combination of other elements (as each element of $\mathscr{I}$ is infinite and each pair of distinct elements of $\mathscr{I}$ has finite intersection). Observe that $\mathscr{F}$ is of cardinality $|\mathscr{I}|$. Thus
\[\dim C_s(X)\geq|\mathscr{I}|=|I|^{\aleph_0}=d(X)^{\aleph_0}.\]
On the other hand, if $f\in C_s(X)$, then $\mathrm{supp}(f)$ is Lindel\"{o}f (as it is separable) and thus $\mathrm{supp}(f)\subseteq H_J$, where $J\subseteq I$ is countable; therefore, we may assume that $f\in C_b(H_J)$. Conversely, if $J\subseteq I$ is countable, then each element of $C_b(H_J)$ can be extended trivially to an element of $C_s(X)$ (by defining it to be identically $0$ elsewhere). Thus $C_s(X)$ may be viewed as the union of all $C_b(H_J)$, where $J$ runs over all countable subsets of $I$. Note that if $J\subseteq I$ is countable, then $H_J$ is separable; thus any element of $C_b(H_J)$ is determined by its value on a countable set. This implies that for each countable $J\subseteq I$, the set $C_b(H_J)$ is of cardinality at most $\mathfrak{c}^{\aleph_0}=2^{\aleph_0}$. There are at most $|I|^{\aleph_0}$ countable $J\subseteq I$. Therefore
\begin{eqnarray*}
\dim C_s(X)\leq\big|C_s(X)\big|&\leq& \Big|\bigcup\big\{C_b(H_J):J\subseteq I\mbox{ is countable}\big\}\Big|\\&\leq& 2^{\aleph_0}\cdot|I|^{\aleph_0}=|I|^{\aleph_0}=d(X)^{\aleph_0}.
\end{eqnarray*}
\end{proof}

\begin{remark}\label{OLJDF}
As it is remarked previously, the three notions of separability, second countability and the Lindel\"{o}f property coincide in the class of metrizable spaces. In particular, Theorem \ref{CTTF} holds true if one replaces separability by either second countability or the Lindel\"{o}f property.
\end{remark}

The \textit{Lindel\"{o}f number} of a space $X$, denoted by $\ell(X)$, is defined by
\[\ell(X)=\min\{\mathfrak{n}:\mbox{any open cover of $X$ has a subcover of cardinality}\leq\mathfrak{n}\}+\aleph_0.\]
In particular, a space $X$ is Lindel\"{o}f if and only if $\ell(X)=\aleph_0$. Any locally compact paracompact space $X$ may be represented as
\[X=\bigcup_{i\in I}X_i,\]
where $X_i$'s are disjoint non-empty Lindel\"{o}f open-closed subspaces of $X$ and $I$ is an index set. (See Theorem 5.1.27 of \cite{E}.) Note that $\ell(X)=|I|$ if
$I$ is infinite, and $\ell(X)=d(X)$ if $X$ is a locally separable metrizable space.

The next theorem is (in part) a corollary of Theorem \ref{GKGF} and may be viewed as a dual to Theorem \ref{CTTF}.

\begin{theorem}\label{HGFLK}
Let $X$ be a locally Lindel\"{o}f paracompact space. Let
\[C_l(X)=\big\{f\in C_b(X):\mathrm{supp}(f)\mbox{ is Lindel\"{o}f}\,\big\}.\]
Then $C_l(X)$ is a Banach algebra isometrically isomorphic to $C_0(Y)$ for some unique (up to homeomorphism) locally compact Hausdorff space $Y$, namely
\[Y=\bigcup\{\mathrm{cl}_{\beta X}L:L\subseteq X\mbox{ is Lindel\"{o}f}\,\}.\]
Furthermore, if $X$ is also locally compact then
\begin{itemize}
\item[\rm(1)] $Y$ is countably compact.
\item[\rm(2)] $Y$ is non-normal if $X$ is non-Lindel\"{o}f.
\item[\rm(3)] $C_0(Y)=C_{00}(Y)$.
\item[\rm(4)] $\dim C_l(X)=\ell(X)^{\aleph_0}$.
\end{itemize}
\end{theorem}

\begin{proof}
Let $\mathfrak{P}$ be the Lindel\"{o}f property and let $\mathfrak{Q}$ be paracompactness. By Example \ref{GGK}, the pair $\mathfrak{P}$ and $\mathfrak{Q}$ satisfy the assumption of Theorem \ref{GKGF}. Note that if $X$ is also locally compact then it assumes a representation as given in the proof of Theorem \ref{CTTF}. The theorem now follows by an argument analogous to the one given in the proof of Theorem \ref{CTTF}.
\end{proof}

\begin{remark}\label{KHGDF}
A point $x$ of a space $X$ is called a \textit{complete accumulation point} of a subspace $A$ of $X$ if $|U\cap A|=|A|$ for every neighborhood $U$ of $x$ in $X$. A space $X$ is called \textit{linearly Lindel\"{o}f} if every linearly ordered (by set theoretic inclusion $\subseteq$) open cover of $X$ has a countable subcover, or equivalently, if every uncountable subspace of $X$ has a complete accumulation point in $X$. As it is observed in \cite{Ko5}, the notions of $\sigma$-compactness, the Lindel\"{o}f property and the linearly Lindel\"{o}f property all coincide in the class of locally compact paracompact spaces; this we will now show. Since $\sigma$-compactness and the Lindel\"{o}f property are identical in the realm of locally compact spaces (see Exercise 3.8.C of \cite{E}) and by the definitions it is clear that the Lindel\"{o}f property implies the linearly Lindel\"{o}f property, we check only that the linearly Lindel\"{o}f property implies the linearly Lindel\"{o}f property in any locally compact paracompact space. Let $X$ be a locally compact paracompact space. Assume a representation for $X$ as
\[X=\bigcup_{i\in I}X_i,\]
where $X_i$'s are disjoint non-empty Lindel\"{o}f open subspaces of $X$ and $I$ is an index set. Suppose that $X$ is not Lindel\"{o}f. Then $I$ is uncountable. Let $A =\{x_i:i\in I\}$, where $x_i\in X_i$ for each $i\in I$. Then $A$ is an uncountable subspace of $X$ with no accumulation point in $X$, as any accumulation point of $A$ should be contained in $X_i$ for some $i\in I$ and $X_i$ is an open subspace of $X$ which intersects $A$ in only one element. Thus, $X$ cannot be linearly Lindel\"{o}f either.

In Theorem \ref{HGFLK} for a locally compact paracompact space $X$ we have considered the set of all $f\in C_b(X)$ such that $\mathrm{supp}(f)$ is Lindel\"{o}f. Since local compactness and paracompactness are both closed hereditary, $\mathrm{supp}(f)$, for each $f\in C_b(X)$, is locally compact and paracompact, and thus is Lindel\"{o}f if and only if it is linearly Lindel\"{o}f if and only if it is $\sigma$-compact. In other words, Theorem \ref{HGFLK} holds true (provided that we assume that $X$ is locally compact from the outset) if one replaces the Lindel\"{o}f property by either the linearly Lindel\"{o}f property or $\sigma$-compactness.
\end{remark}

A regular space is said to be \textit{$\mu$-Lindel\"{o}f}, where $\mu$ is an infinite cardinal, if every open cover of it contains a subcover of cardinality $\leq\mu$. The $\mu$-Lindel\"{o}f property weakens as $\mu$ grows. The Lindel\"{o}f property is identical to the $\aleph_0$-Lindel\"{o}f property; in particular, every Lindel\"{o}f space is $\mu$-Lindel\"{o}f.

The aim of the next theorem is to insert a chain of certain type of ideals between $C_0(X)$ and $C_b(X)$; the length of the chain will be determined by the topology of the space $X$.

\begin{theorem}\label{GGFD}
Let $X$ be a non-Lindel\"{o}f space such that it is either
\begin{itemize}
  \item locally compact paracompact, or
  \item locally Lindel\"{o}f metrizable.
\end{itemize}
Then there is a chain
\[C_0(X)\subsetneqq H_0\subsetneqq H_1\subsetneqq\cdots\subsetneqq H_\lambda=C_b(X)\]
of closed ideals of $C_b(X)$ such that $H_\mu$, for each $\mu\leq\lambda$, is isometrically isomorphic to
\[C_0(Y_\mu)=C_{00}(Y_\mu),\]
where $Y_\mu$ is a locally compact countably compact Hausdorff space which contains $X$ densely. Furthermore,
\[\aleph_\lambda=\ell(X).\]
\end{theorem}

\begin{proof}
Let $\mu$ be an ordinal. Let $\mathfrak{P}_\mu$ be the $\aleph_\mu$-Lindel\"{o}f property and denote
\[H_\mu=C^{\mathfrak{P}_\mu}_0(X).\]
Observe that $X$ is locally Lindel\"{o}f and thus locally $\aleph_\mu$-Lindel\"{o}f. It also follows from the definitions that the $\aleph_\mu$-Lindel\"{o}f property is closed hereditary and is preserved under countable closed sums. Theorem \ref{HJG} now implies that $H_\mu$ is a closed ideal in $C_b(X)$ which is isometrically isomorphic to $C_0(Y_\mu)$ with
\[Y_\mu=\lambda_{\mathfrak{P}_\mu}X.\]
By Theorem \ref{HJG} the space $Y_\mu$ is locally compact Hausdorff and contains $X$ as a dense subspace. We show that
\begin{equation}\label{KJHD}
C_0(Y_\mu)=C_{00}(Y_\mu);
\end{equation}
this in particular proves that $X$ is countably compact. To do this, it suffices to check that every $\sigma$-compact subspace of $Y_\mu$ is contained in a compact subspace of $Y_\mu$. Observe that the space $X$ may be represented as
\[X=\bigcup_{i\in I}X_i,\]
where $X_i$'s are disjoint non-empty Lindel\"{o}f open (and thus closed either) subspaces of $X$ and $I$ is an index set. Note that if $J\subseteq I$ is of cardinality $\leq\aleph_\mu$, then
\[\bigcup_{i\in J}X_i\]
is $\aleph_\mu$-Lindel\"{o}f, and conversely, any $\aleph_\mu$-Lindel\"{o}f subspace of $X$ is contained in such a subspace. An argument similar to the one we have given in the proof of Theorem \ref{CTTF} shows that
\[\lambda_{\mathfrak{P}_\mu}X=\bigcup\bigg\{\mathrm{cl}_{\beta X}\bigg(\bigcup_{i\in J}X_i\bigg):J\subseteq I\mbox{ is of cardinality }\leq\aleph_\mu\bigg\}.\]
Using the above representation and arguing as in the proof of Theorem \ref{CTTF}, it follows that every $\sigma$-compact subspace of $Y_\mu$ is contained in a compact subspace of $Y_\mu$. This proves (\ref{KJHD}).

To conclude the proof we verify that $H_\mu$'s ascend as $\mu$ increases. It is clear that $H_\mu\subseteq H_\kappa$, whenever $\mu\leq\kappa$, as the $\aleph_\mu$-Lindel\"{o}f property is stronger than the $\aleph_\kappa$-Lindel\"{o}f property. Let $\lambda$ be such that $\aleph_\lambda=\ell(X)$. Note that $\ell(X)=|I|$. In particular, $X$ is $\aleph_\lambda$-Lindel\"{o}f, as it is the union of $\aleph_\lambda$ many of its Lindel\"{o}f subspaces. Therefore $H_\lambda=C_b(X)$, as $\mathrm{supp}(f)$ is $\aleph_\lambda$-Lindel\"{o}f for any $f\in C_b(X)$, since it is closed in $X$. Next, we show that all chain inclusions are proper. Note that $C_0(X)\subsetneqq H_0$, as for the characteristic mapping $\chi_A$, where
\[A=\bigcup_{i\in K}X_i\]
and $K\subseteq I$ is infinite and countable, we have $\chi_A\in H_0$ while $\chi_A\notin C_0(X)$. Also, if $\mu<\kappa\leq\lambda$, then $\chi_B\in H_\kappa$ while $\chi_B\notin H_\mu$, where
\[B=\bigcup_{i\in L}X_i\]
and $L\subseteq I$ is of cardinality $\aleph_\kappa$.
\end{proof}

\begin{remark}\label{HFJ}
A \textit{weakly inaccessible cardinal} is defined as an uncountable limit regular cardinal. The existence of weakly inaccessible cardinals cannot be proved within \textsf{ZFC}, though, that such cardinals exist is not known to be inconsistent with \textsf{ZFC}. The existence of weakly inaccessible cardinals is sometimes assumed as an additional axiom. Observe that weakly inaccessible cardinals are necessarily the fixed points of the aleph function, that is, $\aleph_\lambda=\lambda$, if $\lambda$ is a weakly inaccessible cardinal. In Theorem \ref{GGFD} we have inserted a chain of ideals between $C_0(X)$ and $C_b(X)$; this chain will therefore have its length equal to the Lindel\"{o}f number $\ell(X)$ of $X$ provided that $\ell(X)$ is weakly inaccessible.
\end{remark}

\begin{remark}\label{KJHG}
Theorems \ref{CTTF}, \ref{HGFLK} and \ref{GGFD} make essential use of the fact that the spaces under consideration (locally separable metrizable spaces as well as locally compact paracompact spaces) are representable as a disjoint union of open Lindel\"{o}f subspaces. Much of Theorems \ref{CTTF}, \ref{HGFLK} and \ref{GGFD} still remain valid if one replaces the Lindel\"{o}f property by a more or less general topological property $\mathfrak{P}$, provided that the spaces under consideration are representable as a disjoint union of open subspaces each with $\mathfrak{P}$.
\end{remark}

The topological properties considered so far have all been closed hereditary; we now consider pseudocompactness. (Recall that a completely regular space $X$ is called \textit{pseudocompact} if every continuous $f:X\rightarrow\mathbb{R}$ is bounded.) Pseudocompactness is not a closed hereditary topological property; however, it is hereditary with respect to regular closed subspaces, that is, every regular closed subspace of a pseudocompact space is pseudocompact. (See Exercise 3.10.F(e) of \cite{E}.) What makes pseudocompactness so distinct is the known structure of $\lambda_{\mathscr U} X$ (with the precise definition of the ideal ${\mathscr U}$ given in Notation \ref{JJJ}).

Recall that $\mathscr{RC}(X)$ denotes the set of all regular closed subspaces of a space $X$. As noted previously, $(\mathscr{RC}(X),\subseteq)$ is an upper semi-lattice (indeed, a lattice) with $A\vee B=A\cup B$ for any $A,B\in\mathscr{RC}(X)$.

\begin{notation}\label{JJJ}
Let $X$ be a completely regular space. Denote
\[{\mathscr U}=\big\{A\in\mathscr{RC}(X):A\mbox{ is pseudocompact}\big\}.\]
\end{notation}

\begin{lemma}\label{FCG}
Let $X$ be a completely regular space. Then ${\mathscr U}$ is an ideal in $(\mathscr{RC}(X),\subseteq)$.
\end{lemma}

\begin{proof}
Note that ${\mathscr U}$ is non-empty, as it contains $\emptyset$. Let $B\subseteq A$ with $A\in\mathscr{U}$ and $B\in\mathscr{RC}(X)$. Note that $B$ is regular closed in $A$. (Since $B$ is regular closed in $X$ we have $B=\mathrm{cl}_XU$ for some open subspace $U$ of $X$. But then $U$ is open in $A$ and $B=\mathrm{cl}_AU$.) Since $A$ is pseudocompact and pseudocompactness is hereditary with respect to regular closed subspaces then $B$ is pseudocompact. That is $B\in\mathscr{U}$. Next, let $C,D\in\mathscr{U}$. Then $C\vee D=C\cup D$ is pseudocompact, as each $C$ and $D$ is so. Thus $C\vee D\in\mathscr{U}$.
\end{proof}

The following is dual to Lemma \ref{KFFF}.

\begin{lemma}\label{HF}
Let $X$ be a completely regular space. Then $X$ is locally null (with respect to the ideal $\mathscr{U}$) if and only if $X$ is locally pseudocompact.
\end{lemma}

\begin{proof}
The proof is similar to that of Lemma \ref{KFFF}. Observe that (since $X$ is regular) for every $x\in X$ each neighborhood of $x$ in $X$ contains a regular closed neighborhood of $x$ in $X$, that is, a neighborhood in $X$ of the form $\mathrm{cl}_XU$ where $U$ is open in $X$.
\end{proof}

Considering the dualities between Lemmas \ref{KGF} and \ref{FCG} and between Lemmas \ref{KFFF} and \ref{HF}, one can state and prove results dual to Theorems \ref{TTG}, \ref{BBV}, \ref{HGBV}, \ref{UUS}, \ref{DGH}, \ref{KGV} and \ref{HFBY}. We will now proceed with determining $\lambda_{\mathscr U}X$. We need some preliminaries first.

A completely regular space $X$ is said to be \textit{realcompact} if it is homeomorphic to a closed subspaces of some product $\mathbb{R}^\mathfrak{m}$ of the real line. Realcompactness is a closed hereditary topological property. Every regular Lindel\"{o}f space (in particular, every compact Hausdorff space) is realcompact. Also, every realcompact pseudocompact space is compact. To every completely regular space $X$ there corresponds a realcompact space $\upsilon X$ (called the \textit{Hewitt realcompactification of $X$}) which contains $X$ as a dense subspace and is such that every continuous $f:X\rightarrow\mathbb{R}$ is continuously extendible over $\upsilon X$; we may assume that $\upsilon X\subseteq\beta X$. Note that a completely regular space $X$ is realcompact if and only if $X=\upsilon X$. (See Section 3.11 of \cite{E} for further information.)

The following lemma, which may be considered as a dual result of Lemma \ref{HDHD}, is due to A. W. Hager and D. G. Johnson in \cite{HJ}; a direct proof may be found in \cite{C}. (See also Theorem 11.24 of \cite{W}.)

\begin{lemma}[Hager--Johnson \cite{HJ}]\label{A}
Let $X$ be a completely regular space. Let $U$ be an open subspace of $X$ such that $\mathrm{cl}_{\upsilon X} U$ is compact. Then $\mathrm{cl}_X U$ is pseudocompact.
\end{lemma}

Observe that any completely regular space $X$ with a dense pseudocompact subspace $A$ is pseudocompact; as for any continuous $f:X\rightarrow\mathbb{R}$ we have
\[f (X)=f(\mathrm{cl}_XA)\subseteq \mathrm{cl}_\mathbb{R}f(A)\]
and $f(A)$ is bounded (since $A$ is pseudocompact).

\begin{lemma}\label{HGA}
Let $X$ be a completely regular space and let $A$ be regular closed in $X$. Then $\mathrm{cl}_{\beta X} A\subseteq\upsilon X$ if and only if $A$ is pseudocompact.
\end{lemma}

\begin{proof}
One half follows from Lemma \ref{A}, as if $\mathrm{cl}_{\beta X} A\subseteq\upsilon X$ then $\mathrm{cl}_{\upsilon X} A=\mathrm{cl}_{\beta X} A$ is compact, since it is closed in $\beta X$. For the other half, note that if $A$ is pseudocompact then so is $\mathrm{cl}_{\upsilon X} A$, as it contains $A$ as a dense subspace.  But $\mathrm{cl}_{\upsilon X} A$ is realcompact (as it is closed in $\upsilon X$ and $\upsilon X$ is so) and therefore is compact. Thus $\mathrm{cl}_{\beta X} A\subseteq\mathrm{cl}_{\upsilon X} A$.
\end{proof}

\begin{theorem}\label{PTF}
Let $X$ be a completely regular space. Then
\[\lambda_{\mathscr U}X=\mathrm{int}_{\beta X}\upsilon X.\]
\end{theorem}

\begin{proof}
Suppose that $C\in\mathrm{Coz}(X)$ is such that $\mathrm{cl}_X C$ has a pseudocompact neighborhood $U$ in $X$. Since $U$ is pseudocompact and $\mathrm{cl}_X C$ is regular closed in $X$ (and thus in $U$) then $\mathrm{cl}_X C$ is pseudocompact. Thus $\mathrm{cl}_{\beta X} C\subseteq\upsilon X$, by Lemma \ref{HGA}, and therefore $\mathrm{int}_{\beta X}\mathrm{cl}_{\beta X} C\subseteq\mathrm{int}_{\beta X}\upsilon X$.

To show the reverse inclusion, let $t\in\mathrm{int}_{\beta X}\upsilon X$. By the Urysohn Lemma there is a continuous $f:\beta X\rightarrow[0,1]$ with $f(t)=0$ and $f|_{\beta X\setminus\mathrm{int}_{\beta X}\upsilon X}=\mathbf{1}$. Then
\[C=X\cap f^{-1}\big([0,1/2)\big)\in\mathrm{Coz}(X)\]
and $t\in\mathrm{int}_{\beta X}\mathrm{cl}_{\beta X} C$, as $t\in f^{-1}([0,1/2))$ and $f^{-1}([0,1/2))\subseteq\mathrm{int}_{\beta X}\mathrm{cl}_{\beta X} C$ by Lemma \ref{LKG}. Note that if $V=X\cap f^{-1}([0,2/3))$ then $\mathrm{cl}_XV$ is a neighborhood of $\mathrm{cl}_X C$ in $X$, and $\mathrm{cl}_X V$ is pseudocompact by Lemma \ref{HGA}, as it is regular closed in $X$ and
\[\mathrm{cl}_{\beta X}V\subseteq f^{-1}\big([0,2/3]\big)\subseteq\upsilon X.\]
\end{proof}

In our final result we will be dealing with realcompactness. Despite the fact that realcompactness is closed hereditary, it is unfortunately not preserved under finite closed sums in the realm of completely regular spaces. (In \cite{M} -- a correction in \cite{M1} -- S. Mr\'{o}wka describes a completely regular space which is not realcompact but it can be represented as the union of two of its closed realcompact subspaces; a simpler example is given by A. Mysior in \cite{My}.) So, our previous results are not applicable as long as the underlying space is required to be only completely regular. As we will see, the situation changes if we confine ourselves to the class of normal spaces.

Recall that a subspace $A$ of a space $X$ is called \textit{$C$-embedded} in $X$ if every continuous $f:A\rightarrow\mathbb{R}$ is continuously extendible over $X$.

\begin{lemma}[Gillman--Jerison \cite{GJ}]\label{DDJD}
Let $X$ be a completely regular space. Let $A$ be $C$-embedded in $X$. Then $\mathrm{cl}_{\upsilon X}A=\upsilon A$.
\end{lemma}

By ${\mathscr Q}$ in the following we simply mean ${\mathscr I}_\mathfrak{P}$, as defined in Notation \ref{GGH}, with $\mathfrak{P}$ assumed to be realcompactness; the re-definition is for convenience. (This is also consistent with the initial terminology once used to refer to realcompact spaces; realcompact spaces were originally introduced by E. Hewitt in \cite{H}; they were then called \textit{$Q$-spaces}.)

\begin{notation}\label{GYH}
Let $X$ be a space. Denote
\[{\mathscr Q}=\{A\subseteq X:\mathrm{cl}_XA\mbox{ is realcompact}\}.\]
\end{notation}

Recall that a completely regular space $X$ is realcompact if and only if $X=\upsilon X$. Observe that in a normal space each closed subspace is $C$-embedded. (See Problem 3.D.1 of \cite{GJ}.) This observation will be used in the following.

\begin{lemma}\label{KGPF}
Let $X$ be a normal space. Then $\mathscr{Q}$ is an ideal in $({\mathscr P}(X),\subseteq)$.
\end{lemma}

\begin{proof}
Note that ${\mathscr Q}$ is non-empty, as it contains $\emptyset$. Let $B\subseteq A$ with $A\in\mathscr{Q}$. Then $\mathrm{cl}_XB\subseteq\mathrm{cl}_XA$. Since $\mathrm{cl}_XA$ is realcompact and realcompactness is closed hereditary then $\mathrm{cl}_XB$ is realcompact. That is $B\in\mathscr{Q}$. Next, let $C,D\in\mathscr{Q}$. Since $X$ is normal, every closed subspace of $X$ is $C$-embedded in $X$. Thus, using Lemma \ref{DDJD} we have
\begin{eqnarray*}
\upsilon\big(\mathrm{cl}_X(C\cup D)\big)&=&\mathrm{cl}_{\upsilon X}(C\cup D)\\&=&\mathrm{cl}_{\upsilon X}C\cup\mathrm{cl}_{\upsilon X}D\\&=&\upsilon(\mathrm{cl}_X C)\cup\upsilon(\mathrm{cl}_XD)\\&=&\mathrm{cl}_X C\cup\mathrm{cl}_XD=\mathrm{cl}_X(C\cup D).
\end{eqnarray*}
That is $\mathrm{cl}_X(C\cup D)$ is realcompact. Therefore $C\cup D\in{\mathscr Q}$.
\end{proof}

Once one states and proves a lemma dual to Lemma \ref{KFFF} it will be then possible to state and prove results dual to Theorems \ref{TTG}, \ref{BBV}, \ref{HGBV}, \ref{UUS}, \ref{DGH}, \ref{KGV} and \ref{HFBY}. Our concluding result determines $\lambda_{\mathscr Q}X$. As in the case of pseudocompactness, it turns out that $\lambda_{\mathscr Q}X$ is a familiar subspace of $\beta X$ as well.

\begin{theorem}\label{RRGY}
Let $X$ be a normal space. Then
\[\lambda_{\mathscr Q}X=\beta X\setminus\mathrm{cl}_{\beta X}(\upsilon X\setminus X).\]
\end{theorem}

\begin{proof}
Suppose that $C\in\mathrm{Coz}(X)$ is such that $\mathrm{cl}_X C$ has a realcompact neighborhood $U$ in $X$. Then $\mathrm{cl}_X C$ is realcompact, as it is closed in $U$. Since $\mathrm{cl}_X C$ is $C$-embedded in $X$, as $X$ is normal, by Lemma \ref{DDJD} we have $\mathrm{cl}_{\upsilon X}C=\upsilon(\mathrm{cl}_X C)=\mathrm{cl}_X C$. But then $\mathrm{int}_{\beta X}\mathrm{cl}_{\beta X} C\cap(\upsilon X\setminus X)$ is empty, as
\[\mathrm{cl}_{\beta X} C\cap(\upsilon X\setminus X)=\mathrm{cl}_{\upsilon X} C\cap(\upsilon X\setminus X)=\emptyset\]
and thus $\mathrm{int}_{\beta X}\mathrm{cl}_{\beta X} C\cap\mathrm{cl}_{\beta X}(\upsilon X\setminus X)$ is empty, that is
\[\mathrm{int}_{\beta X}\mathrm{cl}_{\beta X} C\subseteq\beta X\setminus\mathrm{cl}_{\beta X}(\upsilon X\setminus X).\]

To show the reverse inclusion, let $t\in\beta X\setminus\mathrm{cl}_{\beta X}(\upsilon X\setminus X)$. Let $f:\beta X\rightarrow[0,1]$ be continuous with $f(t)=0$ and $f|_{\mathrm{cl}_{\beta X}(\upsilon X\setminus X)}=\mathbf{1}$. Then
\[C=X\cap f^{-1}\big([0,1/2)\big)\in\mathrm{Coz}(X)\]
and $t\in\mathrm{int}_{\beta X}\mathrm{cl}_{\beta X} C$, as $t\in f^{-1}([0,1/2))$ and $f^{-1}([0,1/2))\subseteq\mathrm{int}_{\beta X}\mathrm{cl}_{\beta X} C$ by Lemma \ref{LKG}. Now let $V=X\cap f^{-1}([0,2/3))$. Then $\mathrm{cl}_XV$ is a neighborhood of $\mathrm{cl}_X C$ in $X$. Since $\mathrm{cl}_{\beta X}V\cap(\upsilon X\setminus X)$ is empty, as $\mathrm{cl}_{\beta X}V\subseteq f^{-1}([0,2/3])$, we have
\[\mathrm{cl}_X V=X\cap\mathrm{cl}_{\beta X}V=\upsilon X\cap\mathrm{cl}_{\beta X}V=\mathrm{cl}_{\upsilon X}V.\]
Therefore $\mathrm{cl}_X V$ is realcompact, as it is closed in $\upsilon X$.
\end{proof}

\subsubsection*{Acknowledgement}

This work has been initiated during author's short visit to Chamran University of Ahvaz. The author wishes to express his gratitude to Professor F. Azarpanah for making this visit possible.

This research was in part supported by a grant from IPM.



\end{document}